\documentclass[a4paper,11pt]{article}
\usepackage{graphicx}        
\usepackage{multicol}        
\usepackage[bottom]{footmisc}
\usepackage{amsmath}
\usepackage{epsfig}
\usepackage{authblk}
\usepackage{graphicx}
\usepackage{float}
\usepackage{diagbox}
\usepackage{nicefrac}
\usepackage[applemac]{inputenc}
\usepackage[T1]{fontenc}
\usepackage{blkarray}
\usepackage{mathtools,amsthm,amsfonts}
\usepackage{helvet}  

\usepackage{newtxtext}
\usepackage[upint,varg]{newtxmath}

\usepackage{hhline}
\usepackage{tikz}

\usepackage{color}

\usepackage{bm}
\usepackage{nicefrac}
\usepackage[hmargin={28mm,28mm},vmargin={30mm,35mm}]{geometry}
\usepackage{array}
\usepackage{pdfpages}
\usepackage{subfig}
\usepackage{enumitem}
\usepackage{bbold}
\usepackage{xparse}								
\usepackage{stmaryrd}							
\usepackage{sidecap}								
\usepackage{upgreek}
\usepackage{accents}
\allowdisplaybreaks
\usepackage{tcolorbox}
\usepackage{algorithmic}
\usepackage[ocgcolorlinks,allcolors={blue}]{hyperref}

\usepackage{pgfplots,pgfplotstable}

\usetikzlibrary{calc,external}

\usetikzlibrary{cd}
\tikzcdset{every diagram/.append style = {font=\normalsize}}
\tikzcdset{every label/.append style = {font = \small}}
\tikzset{ shorten <>/.style={ shorten >=#1, shorten <=#1 } }

\usetikzlibrary{positioning}
\usepackage{alltt}

\graphicspath{{Figures/}}
\graphicspath{{Figures/}}
\usepackage[bbgreekl]{mathbbol}

\DeclareSymbolFontAlphabet{\mathbb}{AMSb}
\DeclareSymbolFontAlphabet{\mathbbl}{bbold}

\def\R{\mathbb R}
\def\Po{\mathbb P}

\def\Q{\mathbb Q}

\def\D{\mathcal{D}}

\def\<{\langle}
\def\>{\rangle}

\def\Chi{\raise .3ex \hbox{\large $\chi$}} 

\def\ov{\overline}
\def\TT{\mathbb{T}}

\def\norm{\mathbf{n}}
\def\tang{{\boldsymbol{\tau}}}

\newcommand{\HGbracket}[2]{\langle #1,#2\rangle_{\Gamma}}

\def\[{\Bigl [}
\def\]{\Bigr ]}
\def\({\Bigl (}
\def\){\Bigr )}

\def\dsp{\displaystyle}
\def\x{\mathbf{x}}

\def\n{\mathbf{n}}

\def\U{\mathbf{U}}
\def\p{\mathbf{p}}
\def\q{\mathbf{q}}

\def\T{{\mathbf{T}}}

\def\l{{\boldsymbol{\lambda}}}
\def\m{{\boldsymbol{\mu}}}

\def\cells{\mathcal{M}}
\def\faces{\mathcal{F}}
\def\nodes{\mathcal{V}}
\def\edges{\mathcal{E}}

\def\x{{\bf x}}
\def\dsp{{\displaystyle x}}
\def\d{{\rm d}}
\def\div{{\rm div}}

\def\K{\mathbb{K}}
\def\n{\mathbf{n}}
\def\dsp{\displaystyle}
\def\bu{\mathbf{u}}
\def\bv{\mathbf{v}}
\def\bw{\mathbf{w}}
\def\bV{\mathbf{V}}

\def\gamnpm{\gamma_{\norm}^{\pm}}
\def\gamnp{\gamma_{\norm}^{+}}
\def\gamnm{\gamma_{\norm}^{-}}

\newcommand{\UD}{\mathbf{U}_\D}
\newcommand{\UDz}{\mathbf{U}^0_\D}
\newcommand{\MD}{\mathbf{M}_\D}
\newcommand{\CD}{\mathbf{C}_\D}

\newcommand{\IvD}{\mathcal I_{\mathcal V,\D}} 
\newcommand{\ID}{\mathcal I_\D} 

\newcommand{\NORM}[2]{\|#2\|_{#1}}

\def\G{\Gamma}

\def\Ks{{\mathcal{K}s}}
\def\Ksig{{K\!\sigma}}

\def\Lsig{{L\!\sigma}}
\def\sige{{\sigma e}}

\def\bbsig{\bbsigma}
\def\bbeps{\bbespilon}

\newcommand{\jump}[1]{\llbracket #1 \rrbracket}

\newcommand{\nf}{\nicefrac}

\newcommand{\email}[1]{\href{mailto:#1}{#1}}

\begingroup
    \makeatletter
    \@for\theoremstyle:=definition,remark,plain\do{%
        \expandafter\g@addto@macro\csname th@\theoremstyle\endcsname{%
            \addtolength\thm@preskip\parskip
            }%
        }
\endgroup

\newtheorem{theorem}{Theorem}
\numberwithin{theorem}{section}
\newtheorem{proposition}[theorem]{Proposition}
\newtheorem{lemma}[theorem]{Lemma}

\theoremstyle{remark}
\newtheorem{remark}[theorem]{Remark}
\theoremstyle{definition}

\def\grad{\nabla}

\def\div{{\rm div}}

\definecolor{labelkey}{rgb}{0.6,0,1}

\usepackage[normalem]{ulem}
 \newcounter{corr}
 \definecolor{violet}{rgb}{0.580,0.,0.827}
 \newcommand{\corr}[3]{\typeout{Warning : a correction remains in page
 \thepage}
 				\stepcounter{corr}        
 				{\color{blue}\ifmmode\text{\,\sout{\ensuremath{#1}}\,}\else\sout{#1}\fi}
         {\color{red}#2}
         {\color{violet} \ifmmode\text{#3}\else #3\fi }}

\usepackage{parskip}
\usepackage{algorithmic}

\usepackage{diagbox}
\usepackage{booktabs}

\makeatletter
\def\thm@space@setup{%
  \thm@preskip=\parskip \thm@postskip=0pt
}
\makeatother

\newcommand{\doi}[1]{\href{https://dx.doi.org/#1}{DOI:#1}}

\begin{document}
\title{A bubble VEM-fully discrete polytopal scheme for mixed-dimensional poromechanics with frictional contact at matrix--fracture interfaces

}

\author[2]{{J\'er\^ome Droniou}\footnote{\email{jerome.droniou@umontpellier.fr}}}
\author[3]{{Guillaume Ench\'ery}\footnote{\email{guillaume.enchery@ifpen.fr}}}
\author[3]{{Isabelle Faille}\footnote{\email{isabelle.faille@ifpen.fr}}}
\author[1]{{Ali Haidar}\footnote{Corresponding author, \email{ali.haidar@univ-cotedazur.fr}}}
\author[1]{{Roland Masson}\footnote{\email{roland.masson@univ-cotedazur.fr}}}
\affil[1]{Universit\'e C\^ote d'Azur, Inria, Laboratoire J.A. Dieudonn\'e, team Coffee, Nice, France}%
\affil[2]{IMAG, Univ. Montpellier, CNRS, Montpellier, France. School of Mathematics, Monash University, Australia}%
\affil[3]{IFP Energies nouvelles, department of mathematics, 92500 Rueil-Malmaison, France}%
\date{}
\maketitle
\vspace{-.5cm}
{\footnotesize{
\centering
}}

\begin{abstract}
  This article addresses the discretisation of fractured/faulted poromechanical models using 3D polyhedral meshes in order to cope with the geometrical complexity of faulted geological models.
  A new polytopal scheme is proposed for contact-mechanics, based on a mixed formulation combining a fully discrete space and suitable reconstruction operators for the displacement field together with a face-wise constant approximation of the Lagrange multiplier, accounting for the surface tractions along the fracture/fault network. To ensure the inf--sup stability of the mixed formulation, a bubble-like degree of freedom is included in the discrete space of displacements and used in the reconstruction operators. This fully discrete scheme for the displacement is equivalent to a low-order Virtual Element scheme, with a bubble enrichment of the virtual space. This $\Po^1$-bubble VEM--$\Po^0$ mixed discretisation is combined with an Hybrid Finite Volume scheme for the Darcy flow.
The proposed approach is adapted to complex geometry with a network of planar faults/fractures which can include corners, tips and intersections; it leads to efficient semi-smooth Newton solvers for the contact-mechanics, and preserves the dissipative properties of the fully coupled model. The scheme is numericalluy investigated in terms of convergence and robustness on several 2D and 3D test cases, using either analytical or numerical reference solutions and considering both the stand alone static contact-mechanics model and the fully coupled poromechanical model.    
\end{abstract}

\section{Introduction}

Hydro-mechanical models in faulted/fractured porous media play an important role in many applications in geosciences, such as geothermal systems or geological storages. This is in particular the case for CO$_2$ sequestration, in which the pressure build-up due to CO$_2$ injection can potentially lead to fault reactivation with risks of induced seismicity or loss of storage integrity, issues which must be carefully investigated. Numerical modelling is an essential tool to better assess and control these type of risks. It involves the simulation of processes coupling the flow along the faults and in the surrounding porous rock, the rock mechanical deformation, and the mechanical behaviour of the faults related to contact-mechanics.

The objective of this work is to design a robust numerical method which must both preserve the mathematical structure of the coupled system of PDEs, in particular its dissipative properties, and cope with the geometrical complexity of faulted geological models. 
Generating a mesh which represents the complex geometry of heterogeneous geological formations -- including stratigraphic layering, erosions and faults -- is a difficult task and a current topic of intensive research \cite{WELLMANN20181}. Today's geomodels are mostly based on Corner Point Geometries (CPG), leading to hexahedral meshes with edge degeneracy accounting for erosions and non-matching interfaces at faults. CPG can be represented as conformal polyhedral meshes by typically cutting each non planar quadrangular faces into two triangles, and by co-refinement of the fault surfaces \cite{Farmer2005}. However, standard numerical methods for mechanical models are based on Finite Element Methods (FEM) that cannot cope with polyhedral meshes. Alternatively, poromechanical models can be discretised on two different meshes, typically using a FEM mesh for the mechanics and a CPG mesh combined with a Finite Volume scheme for the flow \cite{settari94}. This type of approach induces additional interpolation errors and computational costs, which makes the design of numerical methods applicable to single polyhedral meshes desirable. 
In this direction, one can benefit from the active research field on polytopal discretisations such as Discontinuous Galerkin \cite{hansbo-larson}, Hybrid High Order \cite{hho-book}, MultiPoint Flux and Stress Approximations \cite{SBN12,keilegavlen-nordbotten}, Hybrid Mimetic Methods \cite{dipietro-lemaire}, and Virtual Element Methods (VEM) \cite{beirao-brezzi-marini,da2015virtual}. Among those, VEM, as a natural extension of FEM to polyhedral meshes, has received a notable attention from the numerical mechanics community, and has been applied to various problems including in geomechanics \cite{AHR17}, poromechanics \cite{coulet2020fully,BORIO2021,CRMECA_2023__351_S1_A28_0}, contact-mechanics \cite{Wriggers2016,Wriggerscontact2024}, and fracture mechanics \cite{Wriggers2024}. 

The main objective of this work is to extend the first order VEM to contact-mechanics in the framework of poromechanical models in fractured/faulted porous media. Compared with previous works based on nodal Lagrange multipliers and focusing on single interfaces, such as \cite{Wriggers2016,Wriggerscontact2024}, our purpose is to develop a formulation more adapted to fracture networks including intersections, tips and corners.
The faults/fractures are represented by a network of planar surfaces connected to the surrounding matrix domain, leading to the so-called mixed-dimensional models which have been the object of many recent works \cite{NEJATI2016123,tchelepi-castelletto-2020,GKT16,GH19,contact-norvegiens,thm-bergen,GDM-poromeca-cont,GDM-poromeca-disc,BDMP:21,BoonNordbotten22}. 
Different formulations of the contact-mechanics have been developed to take into account Coulomb frictional contact at matrix--fracture interfaces; these include mixed or stabilized mixed formulations \cite{haslinger-96,Wohlmuth11,Lleras-2009,Wriggers2016}, augmented Lagrangian \cite{BHL2023}, and Nitsche methods \cite{Chouly2017,CHLR2020,beaude2023mixed}.
Following \cite{Hild17,tchelepi-castelletto-2020,BDMP:21}, our choice is based on the  mixed formulation combined with face-wise constant Lagrange multipliers representing surface tractions along the fracture network. 
It allows us to deal with fracture networks including corners, tips and intersections, to use efficient semi-smooth Newton nonlinear solvers, and to preserve at the discrete level the dissipative properties of the contact terms.
The combination of a first order VEM discretisation of the displacement field together with a face-wise constant approximation of the Lagrange multiplier requires a bespoke stabilisation to satisfy the inf-sup compatibility condition. This is achieved by extending to the VEM polyhedral framework the principles of $\Po^1$-bubble FEM discretisations \cite{Renard03}. It relies here on the enrichment of the discrete displacement space by addition of a bubble unknown on one side of each fracture face.
The discretisation of the contact-mechanics is first derived in a fully discrete framework, based on vector spaces of discrete unknowns and reconstruction operators in the spirit of Hybrid High Order or Discrete De Rham methods \cite{hho-book,ddr-variant}. This discretisation is then shown to be equivalent to a VEM formulation. For the coupled poromechanical model, the Darcy flow is discretised using a Finite Volume scheme which, thanks to its flux conservativity properties, could easily be adapted to more advanced models such as multiphase Darcy flows. To fix ideas, the Hybrid Finite Volume scheme \cite{Eymard2009} is used; it belongs to the family of Hybrid Mimetic Mixed methods \cite{Droniou2010} and was adapted to mixed-dimensional models in \cite{BHMS2016}.

This paper is organised as follows. The coupled poromechanical model with Coulomb frictional contact a matrix--fracture interface is first described in Section \ref{sec:model}, based on the mixed-dimensional setting.
Then, its fully discrete approximation is introduced in Section \ref{sec:discretisation}, starting with the mixed-dimensional Darcy flow in Section \ref{subsec:HFV} followed by the contact-mechanical model in Section \ref{sec:discretecontactmechanics}. The equivalent VEM formulation of the contact-mechanical discretisation is derived in Section \ref{equiv_VEM}; the detailed proofs of the connections between both formulations are reported to Appendix \ref{appendice_VEM}.
The numerical Section \ref{num.experiments} investigates the convergence and robustness of the discretisation based on both 2D and 3D analytical or numerical reference solutions. Section \ref{test_meca_standalone} first considers stand-alone static contact-mechanics test cases. Section \ref{test_poromeca} extends the numerical assessment of the proposed scheme to the fully coupled poromechanical model.

\section{Mixed-dimensional poromechanical model}\label{sec:model}

We first present the geometrical setting, and then the poromechanical model in strong form, starting with the mixed-dimensional Darcy flow
followed by the contact-mechanical model and the coupling laws. 

\subsection{Mixed-dimensional geometry and function spaces}\label{sec:notation}

In what follows, scalar fields are represented by lightface letters and vector fields by boldface letters.
We let $\Omega\subset\R^d$, $d\in\{2,3\}$, denote a bounded polytopal domain, partitioned
into a fracture domain $\Gamma$ and a matrix domain $\Omega\backslash\ov\G$.
The network of fractures is defined by 
$$
\ov \Gamma = \bigcup_{i\in I} \ov \Gamma_i,
$$  
where each fracture $\Gamma_i\subset \Omega$, $i\in I$, is a planar polygonal simply connected open domain (Figure \ref{mixed_dime_geom}).
The two sides of a given fracture of $\Gamma$ are denoted by $\pm$ in the matrix domain, with unit normal vectors $\n^\pm$ oriented outward from the sides $\pm$. We denote by $\gamma^\pm$ the trace operators on the side $\pm$ of $\Gamma$ for functions in $H^1(\Omega{\setminus}\ov\Gamma)^r$ with $r=1$ or $d$ (in the latter case, the trace is taken component-wise). The jump operator on $\Gamma$ is defined as
$$
\jump{\bu} = \gamma^+ \bu - \gamma^- \bu\qquad\forall \bu\in H^1(\Omega{\setminus}\ov\Gamma)^d,
$$
and we denote by
$$
\jump{\bu}_\n = \jump{\bu}\cdot\n^+ \quad \mbox{ and } \quad \jump{\bu}_\tang = \jump{\bu} -\jump{\bu}_\n \n^+
$$
its normal and tangential components. The notation $\jump{\bV_m}_\n = \gamnp \bV_{m} +  \gamnm \bV_{m}$ will also be used to denote the normal jump of  functions $\bV_m \in H_{\div}(\Omega\setminus\ov\Gamma)$, where $\gamma^\pm_\n$ is the normal trace operator on the side $\pm$ of $\Gamma$.
The tangential gradient and divergence along the fractures are respectively denoted by $\nabla_\tang$ and $\div_\tang$. The symmetric gradient operator $\bbeps$ is defined such that $\bbeps(\bv) = {1\over 2} (\nabla \bv + \prescript{t}{} (\nabla \bv) )$ for a given vector field $\bv \in H^1(\Omega\backslash\ov\Gamma)^d$.

Let $d^c_f: \Gamma \to (0,+\infty)$ be the fracture aperture in contact state (see Figure \ref{fig:aperture}). The function $d^c_f$ is assumed to be continuous with zero limits at
$\partial \Gamma \setminus (\partial\Gamma\cap \partial\Omega)$ (i.e.,~the tips of $\Gamma$) and strictly positive limits at $\partial\Gamma\cap \partial\Omega$.

The primary unknowns of the poromechanical model are the matrix pressure $p_m$ in the matrix domain, the  fracture pressure $p_f$ along the fracture network, and the displacement vector field $\bu$ in the matrix domain (see Figure \ref{mixed_dime_geom}), for which we introduce the following function spaces. We denote by  $H_{c}^1(\Gamma)$ the space of functions $q_f \in L^2(\Gamma)$ such that $(d^c_f)^{\nf 3 2} \nabla_\tang q_f$ belongs to  $L^2(\Gamma)^{d-1}$,
and whose traces are continuous at fracture intersections. The weight $(d^c_f)^{\nf 3 2}$ in the definition of $H_{c}^1(\Gamma)$ accounts for the fact that the fracture aperture $d_f  \geq d^c_f$ can vanish at the tips and that only the $L^2(\Gamma)^{d-1}$ norm of $d_f^{\nf 3 2}  \nabla_\tang p_f$ will be controlled. The space for the displacement is
\begin{equation*}
\U = H^1(\Omega\backslash\ov\Gamma)^d, 
\end{equation*}
and we denote by 
$$
\U_0=\{\bv\in\U\,:\,\bv_{|\partial\Omega}=0\}
$$ 
its subspace of vanishing displacement at the boundary $\partial \Omega$. Assuming that $\Omega \backslash \ov\Gamma$ is connected, the semi-norm $\|\bv \|_{\U} = \|\nabla \bv \|_{L^2(\Omega)^d}$ defines a norm on $\U_0$.
The space for the pair of matrix/fracture pressures is
\begin{equation*}
V = V_m \times V_f\quad \mbox{ with }\quad
V_m = H^1(\Omega{\setminus}\ov\Gamma) \  \mbox{ and } \
V_f = H_{c}^1(\Gamma).
\end{equation*}
For $q =(q_m, q_f) \in V$, let us denote the jump operator on the side $\pm$ of the fracture by
$$
\jump{q}^\pm = \gamma^\pm q_m - q_f.
$$

\begin{figure}
	\centering
	\resizebox{0.28\textwidth}{!}{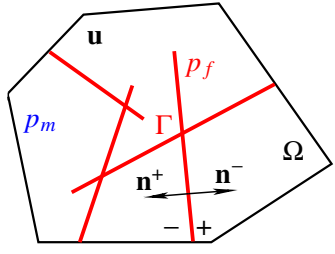} 
	\caption{Mixed-dimensional geometry with the fracture network $\Gamma$ and the matrix domain $\Omega\setminus\ov\Gamma$. The poromechanical unknowns are defined by the matrix pressure $p_m$, the fracture pressure $p_f$ and the displacement vector field $\bu$ in the matrix domain.}
	\label{mixed_dime_geom}
\end{figure}

\begin{figure}
	\begin{center}
		\includegraphics[scale=.65]{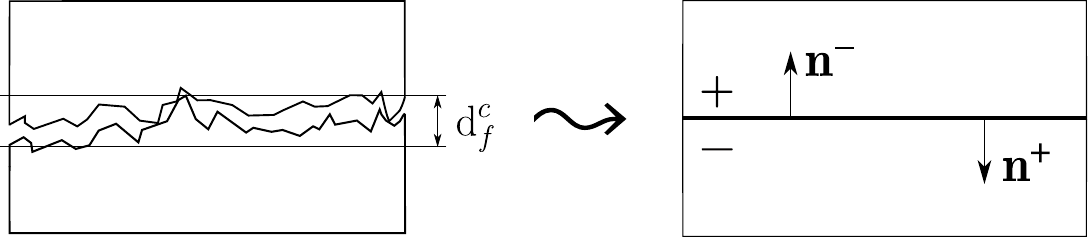}
		\caption{Conceptual fracture model with contact at asperities. $\d_{f}^{c}$ is the fracture aperture at contact state.}
		\label{fig:aperture}
	\end{center}  
\end{figure}

\subsection{Mixed-dimensional Darcy flow}

The flow model is a mixed-dimensional model assuming an incompressible fluid. It based on volume conservation equations, on the Darcy law for the velocity field $\bV_m$ in the matrix, and on the Poiseuille law for the velocity field $\bV_f$ along the fractures. Additionally, the model incorporates transmission conditions to account for the interaction and exchange of fluid between the matrix and fractures. Denoting by $(0,T)$ the time interval, we obtain the following flow equations 
\begin{equation}
	\label{eq_edp_darcy} 
	\left\{\!\!\!\!
	\begin{array}{lll}
		&\partial_t \phi_{m}+\operatorname{div} \bV_{m}=h_{m} & \mbox{ on } (0,T)\times \Omega{\setminus}\ov\Gamma,\\[1ex]
		& \bV_{m}=-\frac{ \mathbb{K}_{m}}{\eta} \nabla p_{m} & \mbox{ on } (0,T)\times \Omega{\setminus}\ov\Gamma,\\[1ex]
		&\partial_t \d_{f}+\operatorname{div}_{\tang} \bV_{f}-\jump{\bV_{m} }_{n}=h_{f}  & \mbox{ on } (0,T)\times \Gamma,\\[1ex]
		& \bV_{f}=\frac{C_{f}}{\eta} \nabla_{\tang} p_{f}, & \mbox{ on } (0,T)\times \Gamma,\\[1ex]
		& \gamnpm \bV_{m} =\Lambda_{f} \jump{ p }^{\pm} & \mbox{ on } (0,T)\times \Gamma. 
	\end{array}
	\right.
\end{equation}
In \eqref{eq_edp_darcy}, $\eta$ is the constant fluid dynamic viscosity, $\phi_m$ is the matrix porosity, and $\mathbb{K}_m$ is the matrix permeability tensor. The right hand sides $h_m$ and $h_f$ account for injection or production source terms. The fracture aperture, denoted by $\d_f$, yields the fracture conductivity $C_f$, typically given by the Poiseuille law $C_{f}=\d_{f}^{3}/12$, and the fracture normal transmissivity $\Lambda_{f}=2 K_{f,\n}/\d_f$, where $K_{f,\n}$ is the fracture normal permeability. 
To fix ideas, homogeneous Dirichlet boundary conditions are imposed for $p_m$ on $\partial\Omega$ and for $p_f$ on $\partial \Gamma \cap \partial \Omega$, while homogeneous Neumann boundary conditions are imposed for $p_f$ at the fracture tips on $\partial \Gamma \backslash \partial \Omega$.

\subsection{Quasi static contact-mechanical model and coupling laws}

The quasi static contact-mechanical model accounts for the poromechanical equilibrium equation, with a Biot isotropic linear elastic constitutive law and a Coulomb frictional contact model at matrix--fracture interfaces:
\begin{equation}
	\label{eq_edp_contact_meca} 
	\left\{\!\!\!\!
	\begin{array}{lll}
		& -\operatorname{div} \bbsigma^{\top}\left(\bu, p_{m}\right)= \mathbf{f}  & \mbox{ on } (0,T)\times \Omega{\setminus}\ov\Gamma,\\[1ex]
		& \T^{+}+\T^{-}=0& \mbox{ on } (0,T)\times \Gamma,\\[1ex]
		& T_{\n} \leqslant 0,~ \jump{\bu }_{\n} \leqslant 0,~ \jump{\bu }_{\n}T_{\n} =0, & \mbox{ on } (0,T)\times \Gamma,\\[1ex]
		& \left|\T_{\tang}\right| \leqslant-F ~T_{\n}& \mbox{ on } (0,T)\times \Gamma,\\[1ex]
		& \T_{\tang}  \cdot \partial_t \jump{\bu }_{\tang}-F ~T_{\n}\left| \partial_t \jump{\bu }_{\tang}\right|=0 & \mbox{ on } (0,T)\times \Gamma.
	\end{array}
	\right.
\end{equation}
The total stress tensor $\bbsigma^{\top}$ is defined in terms of the effective stress tensor $\bbsigma$  and the matrix pressure $p_m$ as follows
\begin{equation}
	\label{eq_edp_stress} 
	\left\{\!\!\!\!
	\begin{array}{lll}
	& \bbsigma^{\top}\left(\bu, p_{m}\right)=\bbsigma(\bu)-b p_{m} ~\mathbb{I}  & \mbox{ on } (0,T)\times \Omega{\setminus}\ov\Gamma,\\[1ex]
	& \bbsigma(\bu)=2 \mu \bbeps(\bu)+\lambda \operatorname{div} \bu ~\mathbb{I}  & \mbox{ on } (0,T)\times \Omega{\setminus}\ov\Gamma. 
	\end{array}
	\right.
\end{equation}
In \eqref{eq_edp_contact_meca}--\eqref{eq_edp_stress}, $b$ is the Biot coefficient, $\mu$ and $\lambda$ are the Lam\'e parameters, $F$ is the friction coefficient, and the surface tractions are defined by
\begin{equation}
	\label{eq_def_T} 
	\left\{\!\!\!\!
	\begin{array}{lll}
		& \T^{ \pm}=  \gamnpm  {\bbsigma}^{\top}\left(\bu, p_{m}\right)      + p_{f} \n^{ \pm}, & \mbox{ on } (0,T)\times \Gamma, \\[1ex]
		& T_{\n} = {\bf T}^{+}\cdot \n^+, & \mbox{ on } (0,T)\times \Gamma, \\[1ex]
		& {\bf T}_{\tang} = {\bf T}^{+} - ({\bf T}^{+}\cdot \n^+)\n^+ & \mbox{ on } (0,T)\times \Gamma,
	\end{array}
	\right.
\end{equation}
where the trace $\gamnpm  {\bbsigma}^{\top}$ is taken row-wise. 
Homogeneous Dirichlet boundary conditions are imposed on $\partial \Omega$ for the displacement field $\bu$. The model is closed by the following coupling laws 
\begin{equation}\label{coupling_law}
	\left\{\!\!\!\!
	\begin{array}{lll}
	& \partial_t \phi_{m}=b \operatorname{div} \left(\partial_t \bu \right) + \frac{1}{M} \partial_{t} p_{m}  & \mbox{ on } (0,T)\times \Omega{\setminus}\ov\Gamma,\\[1ex]
	& \d_f = \d_{f}^{c} - \jump{\bu }_{\n} & \mbox{ on } (0,T)\times \Gamma.
	\end{array}
\right.
\end{equation}
The first equation accounts for the linear isotropic poroelastic constitutive law for the porosity $\phi_m$, with $M$ denoting the Biot modulus. The second equation specifies the fracture aperture $\d_f$, assuming to fix ideas that the contact aperture $\d_{f}^c$ is reached for $\jump{\bu }_{\n} = 0$, see Figure \ref{fig:aperture}. 

Following~\cite{Wohlmuth11}, the poromechanical model  with Coulomb frictional contact is formulated in mixed form using a vector Lagrange multiplier ${\boldsymbol{\lambda}}: \Gamma \to \R^d$ at matrix--fracture interfaces.
Denoting for $r\in\{1,d\}$ the duality pairing of $H^{-1/2}(\Gamma)^r$ and $H^{1/2}(\Gamma)^r$ by $\HGbracket{\cdot}{\cdot}$, we define the dual cone 
\begin{align*}
\bm{C}_f(\lambda_\n) = \Big\{\m{} \in  H^{-1/2}(\Gamma)^d \,:\,
 \HGbracket{\m}{\bv} \leq \HGbracket{F \lambda_\n}{|\bv_{\tang}|} \mbox{ for all } \bv \in (H^{1/2}(\Gamma))^d \mbox{ with } v_\n \leq 0\Big\}. 
\end{align*}

The Lagrange multiplier formulation of \eqref{eq_edp_contact_meca}--\eqref{eq_edp_stress}--\eqref{eq_def_T} then formally reads, dropping any consideration of regularity in time: find $\bu:[0,T]\to \U_0$ and $\l :[0,T]\to \bm{C}_f(\lambda_\n)$ such that for all $\bv:[0,T]\to \U_0$ and $\m :[0,T]\to \bm{C}_f(\lambda_\n)$, 
\begin{equation}
  \label{Lagrange_meca_contactfriction}
\begin{aligned}
 & \dsp \int_\Omega \( \bbsig( \bu): \bbeps( \bv) - b ~ p_m \div( \bv)\) 
+  \HGbracket{\l}{\jump{ \bv}} 
+ \int_\G  p_f ~\jump{ \bv}_\n 
\dsp =  \int_\Omega \mathbf{f}\cdot \bv,\\[.8em]
& \dsp \HGbracket{{\mu}_\n-{\lambda}_\n}{\jump{ \bu}_\n} + \HGbracket{\m_{\tang} - \l_{\tang}}{ \jump{\partial_t \bu}_{\tang}}\leq 0.   
\end{aligned}
\end{equation}
It can be checked that this variation formulation links the Lagrange multiplier and the traction by $\l = -{\bf T}^+ = {\bf T}^-$.

\section{Discretisation}\label{sec:discretisation}

We consider here the discretisation of the coupled model \eqref{eq_edp_darcy}--\eqref{coupling_law} on conforming polyhedral meshes defined in Section \ref{subsec:mesh}.  To make the presentation more concrete, the discretisation of the mixed-dimensional Darcy flow is based on the Hybrid Finite Volume (HFV) scheme introduced in \cite{BHMS2016} and briefly recalled in Section \ref{subsec:HFV}, but other finite volume methods could be considered as well. Sections \ref{sec:discretecontactmechanics} and \ref{equiv_VEM} deal with the core of this work which is the discretisation of the contact-mechanical model based on a mixed $\Po^1$-bubble VEM--$\Po^0$ formulation. It is first introduced in Section \ref{sec:discretecontactmechanics}  using a discrete setting, and an equivalent Virtual Element formulation is described in Section \ref{equiv_VEM}. The discrete coupling conditions are presented in Section \ref{subsec:discrete_coupl_cond}.

\subsection{Space and time discretisation}\label{subsec:mesh}

We consider a polyhedral mesh of the domain $\Omega$, conforming to the fracture network $\Gamma$. For each cell $K$, we denote by $h_K$ its diameter and by $|K|$ its measure; we also denote by $|\sigma|$ the $(d-1)$-dimensional measure of a face $\sigma$. The set of cells $K$, the set of faces $\sigma$,  the set of edges $e$ and the set of nodes $s$ are denoted respectively by  $\cells$, $\faces$, $\edges$ and $\nodes$. 
For any subset $A$ of $\R^d$ and $\mathcal X\in\{\cells,\faces,\edges,\nodes\}$, we denote by $\mathcal X_A$ the set of elements in $\mathcal X$ that are included in $A$ or that contain $A$; hence, $\faces_K$ is the set of faces of the element $K\in\cells$, $\edges_\sigma$ is the set of edges of the face $\sigma\in\faces$, and $\cells_\sigma$ is the set of cells that contain the face $\sigma$.
We assume the existence of a subset of faces $\faces_{\Gamma} \subset \faces$ such that
$$
\ov{\Gamma} = \bigcup_{\sigma \in \faces_{\Gamma}} \ov{\sigma}.
$$
The mesh is assumed conforming in the sense that the set $\cells_{\sigma}$ of neighboring cells of $\sigma \in \faces$  is either $\cells_{\sigma} = \{ K,L\}$ for an interior face $\sigma \in \faces^{\text{int}}$ (in which case we write $\sigma=K|L$), or  $\cells_{\sigma} = \{ K\}$ for a boundary face $\sigma \in \faces^{\text{ext}}$.
It is assumed that $\faces_{\Gamma} \subset  \faces^{\text{int}}$ and, if $\sigma=K|L$, that $K$ and $L$ are ordered such that $\n_{\Ksig} = \n^+$ and $\n_{\Lsig} = \n^-$, where $\n_{\Ksig} $ (resp.~$\n_{\Lsig}$) is the unit  normal vector to $\sigma$ oriented outward of $K$ (resp.~$L$). 

We denote by $\nodes^{\text{ext}}$ and $\edges^{\text{ext}}$ the boundary nodes and edges.
For each $\sigma \in \faces$,  $\norm_{\sige}$ is the unit normal vector to $e \in \edges_{\sigma}$ in the plane spanned by $\sigma$ oriented outward to $\sigma$.
For each $K \in \cells$ and $\sigma \in \faces_K$ we denote by $\gamma^{\Ksig}$ the trace operator on $\sigma$ for functions in $H^1(K)$; similarly, for each $\sigma \in \faces$ and $e \in \edges_{\sigma}$, $\gamma^{\sige}$ is the trace operator on $e$ for functions in $H^1(\sigma)$.

For the time discretisation, we consider a partition $(t^n)_{0\leq n\leq N}$ of the time interval $[0, T]$ with $t^0=0$, $t^N= T$ and $t^{n}-t^{n-1}=\Delta t^n>0$, $n=1,...,N$. For a family $w=(w^n)_{n=0,\ldots,N}$, we let $\delta_t^n w \coloneq \frac{w^n - w^{n-1}}{\Delta t^n}$.

We assume in the following that the Lam\'e parameters $\mu$, $\lambda$ and the permeability tensor $\mathbb{K}_m$ are cell-wise constant. The fracture normal permeability $K_{f,\n}$ and friction coefficient $F$ will both be assumed face-wise constant.

\subsection{Mixed-dimensional Darcy flow discretisation}\label{subsec:HFV}

We consider the Hybrid Finite Volume (HFV) discretisation of the mixed-dimensional Darcy flow model introduced in \cite{BHMS2016}. It is based on the vector space $X_\D = X_{\D_m} \times X_{\D_f}$ of discrete pressures $p_\D = (  p_{\D_m}, p_{\D_f} )$ defined by 
\begin{align*}
X_{\D_m} &= \left\{p_{\D_m}= \((p_K)_{K\in\cells},(p_\sigma)_{\sigma \in \faces\setminus \faces_{\Gamma}},(p_\Ksig)_{\sigma\in\faces_\Gamma,\,K\in\cells_\sigma} \)\,:\,p_K \in \R\,,\; p_{\sigma} \in \R\,,\; p_{\Ksig} \in \R \right\}, \\
X_{\D_f} &= \left\{ p_{\D_f} = \( (p_\sigma)_{\sigma\in \faces_\Gamma}, (p_e)_{e\in \edges_\Gamma} \)\,:\, p_\sigma \in \R\,,\; p_e \in \R \right\}. 
\end{align*}
We denote by $X^0_{\D_m}$ (resp.~$X^0_{\D_f}$)  the subspace of $X_{\D_m}$ (resp.~$X_{\D_f}$) with vanishing values at the boundary $\faces^{\text{ext}}$ (resp.~$\edges^{\text{ext}} \cap \edges_{\Gamma}$), and we set $X_\D^0 = X^0_{\D_m}\times X^0_{\D_f}$. 
The HFV scheme is obtained by replacing, in the primal variational formulation of \eqref{eq_edp_darcy}, the continuous operators by discrete reconstruction operators $\nabla_{\D_m}$, $\Pi_{\D_m}$ in the matrix and $\nabla_{\D_f}$, $\Pi_{\D_f}$, $\jump{\cdot}_\D^\pm$ along the fractures, defined as follows.

The matrix gradient reconstruction operator $\nabla_{\D_m}~:~ X_{\D_m} \longrightarrow L^2 ( \Omega )^d$ is such that, for a suitable symmetric positive definite tensor $(\TT_{K}^{\nu \nu'})_{\nu,\nu'}$, for all $p_{\D_m} \in X_{\D_m}$,
$$
\int_{K} \mathbb{K}_m  \nabla_{\D_m} p_{\D_m} \cdot \nabla_{\D_m} q_{\D_m} =  \sum_{\nu \in I_K} \sum_{\nu' \in I_K} \TT_{K}^{\nu \nu'}\left( p_{\nu} - p_{K}\right) \left( q_{\nu'} - q_{K}\right)\qquad\forall q_{\D_m} \in X_{\D_m},
$$ 
with 
$$
I_K = \left\{    \sigma \in \faces_K\setminus \faces_{\Gamma}\right\} \cup  \left\{ \Ksig,~\sigma \in \faces_K \cap \faces_{\Gamma} \right\}.
$$
The fracture tangential gradient operator $\nabla_{\D_f}~:~X_{\D_f} \longrightarrow L^2 (\Gamma )^{d-1}$ is such that, for a suitable symmetric positive definite tensor $(\TT_{\sigma}^{e e'})_{e,e'}$, for all $p_{\D_f} \in X_{\D_f}$,
$$
\int_{\sigma} C_{f,\D}  \nabla_{\D_f} p_{\D_f} \cdot \nabla_{\D_f} q_{\D_f}  =  \sum_{e \in \edges_{\sigma}} \sum_{e' \in \edges_{\sigma}} \TT_{\sigma}^{e e'}\left( p_{e'} - p_{\sigma}\right) \left( q_{e} - q_{\sigma}\right)\qquad \forall q_{\D_f} \in X_{\D_f},
$$ 
with the face-wise constant approximation of the fracture conductivity given by
$$
C_{f,\D} = {(\d_{f,\D})^3\over 12}
$$
(where $\d_{f,\D}$ is the face-wise constant approximation of the fracture aperture specified in \eqref{discrete_coupling_law}).
The detailed definitions of $(\mathbb{T}_K^{\nu \nu'})_{\nu,\nu'}$ and $(\mathbb{T}_{\sigma}^{e e'})_{e,e'}$ can be found in \cite{BHMS2016}.

The piecewise constant matrix and fracture function reconstruction operators $\Pi_{\D_m}:X_{\D_m} \to   L^2(\Omega )$ and  $\Pi_{\D_f}:X_{\D_f} \to   L^2 (\Gamma)$ are defined by
$$
\begin{aligned}	
&\Pi_{\D_m} q_{\D_m}(\x) = q_K, ~ &\forall \x \in K,~K \in \cells,\\
&\Pi_{\D_f} q_{\D_f}(\x) = q_{\sigma}, ~ &\forall \x \in \sigma,~\sigma \in \faces_{\Gamma}, 
\end{aligned}
$$
and the face-wise constant jump reconstruction operators $\jump{\cdot }^{\pm}_{\D}:X_{\D}  \to L^2 ( \Gamma )$ by  
$$
\jump{q_{\D} }_{\D}^+ (\x)=  q_{\Ksig} - q_{\sigma}, \quad 
\jump{q_{\D} }_{\D}^- (\x)=  q_{\Lsig} - q_{\sigma}, \quad  \forall \x \in \sigma = K|L, \, \sigma\in \faces_\Gamma.    
$$
Let us also define the face-wise constant approximation of the fracture normal transmissivity
$$
\Lambda_{f,\D} = \frac{2 K_{f,\n}}{\d_{f,\D}}. 
$$
Then, the HFV scheme can be expressed as the following discrete variational formulation: find $(p_\D^n)_{n=1,\ldots,N}\in (X_\D^0)^N$ such that, for all $q_{\D} \in X_{\D}^0$ and all $n = 1,\cdots,N$, it holds 
\begin{equation}\label{GD_DarcyFlow}
	\begin{array}{lll}
		& \dsp  \int_{\Omega}  \( \delta_t^n \phi_{\D} \, \Pi_{\D_m} q_{\D_m} + \frac{\mathbb{K}_m}{\eta} \nabla_{\D_m} p_{\D_m}^n \cdot \nabla_{\D_m} q_{\D_m} \)  \\[2ex]
		&\dsp   + \int_{\Gamma}   \( \delta_t^n \d_{f,\D} \, \Pi_{\D_f} q_{\D_f} + \frac{C_{f,\D}^{n-1}}{\eta} \nabla_{\D_f} p_{\D_f}^n \cdot \nabla_{\D_f} q_{\D_f} \) \\[2ex]
		&  \dsp + \sum_{\mathfrak{a} \in \{+,-\}} \int_{\Gamma} \Lambda_{f, \D}^{n-1} \jump{ p_{\D}^n}_{\D}^{\mathfrak{a}}~ \jump{q_\D}_{\D}^{\mathfrak{a}}  
	\dsp = \int_{\Omega} h_m \Pi_{\D_m} q_{\D_m}  + \int_{\Gamma} h_f \nabla_{\D_f} q_{\D_f}, 
	\end{array}
\end{equation}
where  the approximations of the porosity $\phi_{\D}$ and fracture aperture $\d_{f,\D}$  are defined by the coupling laws specified in \eqref{discrete_coupling_law}.

\subsection{Contact-mechanics discretisation and coupling conditions}\label{sec:discretecontactmechanics}

The discretisation of the contact-mechanics \eqref{eq_edp_contact_meca}--\eqref{eq_edp_stress}--\eqref{eq_def_T} is based on a mixed variational formulation set on the spaces of displacement field and of Lagrange multipliers accounting for the surface traction $-\T^+$ along the fractures.
Following \cite{BDMP:21,beaude2023mixed}, we focus on a face-wise constant approximation of the Lagrange multipliers, which allows us to readily deal with fracture networks including intersections, corners and tips. This choice also provides a local expression of the discrete contact conditions leading to the preservation of the dissipative property of the contact term, as well as to efficient non-linear solvers based on semi-smooth Newton algorithms.

In this section, we describe the discretisation of the displacement field using a similar framework as for the Darcy flow based on a vector space of discrete displacement and reconstruction operators. An equivalent Virtual Element formulation is provided in Section \ref{equiv_VEM}.

\subsubsection{Discrete unknowns and spaces}

Let us first define a partition $\ov{\cells}_s$ of the set of cells $\cells_s$ around a given node $s\in \nodes$. For a given cell $K\in \cells_s$ we denote by $\Ks\in \ov{\cells}_s$ the subset of $\cells_s$ such that $\bigcup_{L\in \Ks} \ov{L}$ is the closure of the connected component of $(\bigcup_{L\in \cells_s} \ov{L})\setminus \Gamma$ containing the cell $K$ (denoted by $\Omega_{\Ks}$ in  Figure \ref{dof_3}). In other words, $\Ks$ is the set of cells in $\cells_s$ that are on the same side of $\Gamma$ as $K$. 

To account for the discontinuity of the discrete displacement field at matrix--fracture interfaces, a nodal displacement unknown $\bv_{\Ks}$ is defined for each $\Ks \in \ov{\cells}_s$. There is a unique nodal displacement unknown $\bv_{\Ks}$ at a node $s$ not belonging to $\Gamma$, since $\ov{\cells}_s = \cells_s$ in that case.  On the other hand, for $s\in \nodes_\Gamma$, the nodal displacement unknown $\bv_{\Ks}$ is the one on the side $K$ of the set of fractures connected to $s$. 

The additional bubble displacement unknown $\bv_{\Ksig}$ is on the ``+'' side of a fracture face $\sigma$, and is therefore linked with the following (possibly empty) set $\faces_K^{\Gamma,+}$ of fracture faces of $K$ such that $K$ is on their $+$ side:
\begin{equation*}
	\faces_{\Gamma,K}^+ = 
	\Big\{ \sigma \in \faces_{\Gamma} \cap \faces_{K}| ~\n_{\Ksig} \cdot \n^+  > 0  \Big\}.
\end{equation*}
The vector space $\UD$ of discrete displacements is then defined as 
 \begin{equation*}
\UD = \left\{ \bu_\D = \( (\bv_{\Ks})_{\Ks \in \ov{\cells}_s, s\in \nodes}, (\bv_\Ksig)_{\sigma \in \faces_{\Gamma,K}^+,~ K \in \cells } \)\,:\, \bv_\Ks \in \R^d\,,\; \bv_{\Ksig} \in \R^d \right\},
\end{equation*}
and $\UDz$ is its subspace of vectors having vanishing nodal values at all $ s \in \nodes^{\text{ext}}$.

\begin{figure}[H]
	\begin{center}
	  \includegraphics[width=14cm]{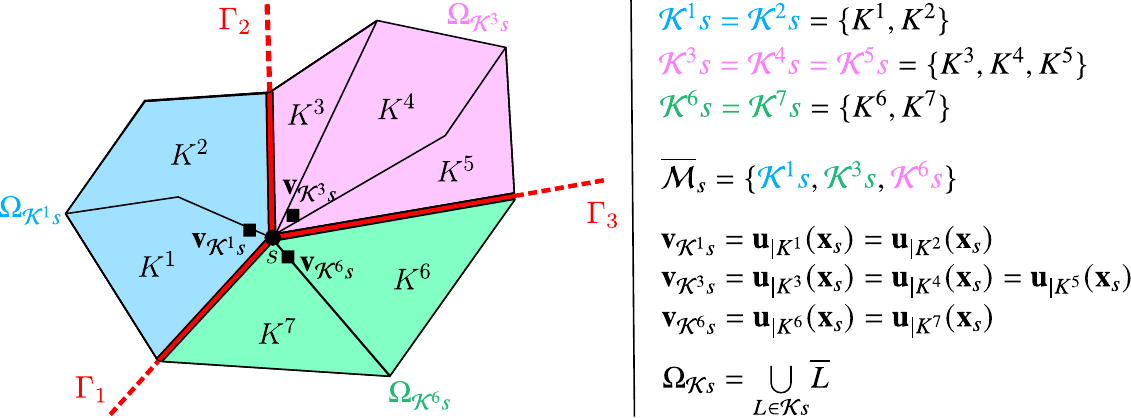}
		\caption{Degrees of freedom $\bv_\Ks$, $\Ks\in \overline{\cells}_s$, at a given node $s \in \nodes_\Gamma$ with three intersecting fractures. $\bv_\Ks$ corresponds to the nodal unknown at node s located on the side $K$ of the fractures. Here, $\bu$ is a fictive continuous function that $\bv$ could interpolate, and is used to give a clearer meaning to the degrees of freedom on each side of the fracture. }\label{dof_3}
	\end{center}
\end{figure}

\begin{remark}[Two-sided bubbles vs.~one-sided bubbles]\label{rem:twobubbles}
It is also possible to define the vector space $\UD$ with bubble unknowns on both sides of each fracture face (two-sided bubbles), which amounts to replacing $\faces_{\Gamma,K}^+$ by $\faces_\Gamma \cap \faces_K$ in the definition of this space. This can lead to a better stabilisation of the Lagrange multiplier, as exhibited in the numerical section in Figure \ref{compression_comparison}, at the price of additional unknowns. Moreover, two-sided bubbles raise additional difficulties for the extension of the scheme to non-matching meshes at matrix--fracture interfaces, as opposed to the one-sided bubble case where both Lagrange multipliers and bubble unknowns can be defined on the same side of the interface.
\end{remark}

The unknowns $\bv_{\Ks}$ correspond to the nodal displacements (see Figure \ref{dof_3}), while the bubble unknowns $\bv_{\Ksig}$ -- which we recall are located on the "+" side of the fracture face $\sigma$ (see Figure \ref{fig_bul}) -- correspond to a correction of the face mean value with respect to the linear nodal reconstruction $\Pi^{\Ksig} \bv_{\D}$ defined below in \eqref{PiKsigma}. These additional bubble unknowns are required to ensure the stability of the mixed variational formulation based on face-wise constant Lagrange-multiplier along the network $\Gamma$.

Let $\mathcal{C}^0( \Omega \setminus \ov{\Gamma})$ denote the space of continuous functions $f:\Omega \setminus \ov{\Gamma} \to \mathbb{R}$ with finite limits on $\partial \Omega$ and on each side of $\Gamma$. 
We define two interpolators $\IvD:\mathcal{C}^0(\Omega \setminus \ov{\Gamma})^d\to\UD$ and $\ID:\mathcal{C}^0(\Omega \setminus \ov{\Gamma})^d\to\UD$. The first one only interpolates at the vertices: for $\bv \in \mathcal{C}^0(\Omega \setminus \ov{\Gamma})^d$, $\IvD\bv\in\UD$ is given by
\begin{equation*}
\begin{array}{ll}
(\IvD \bv)_{\Ks}  = \bv_{|K}(\x_s)&\mbox{ for all } s \in \nodes ,\, \Ks \in  \ov{\cells}_s, \, K\in \Ks,\\
(\IvD \bv )_{\Ksig} = 0& \mbox{ for all } \sigma \in \faces_{\Gamma,K}^+, \,  K \in \cells.
\end{array}
\end{equation*}
The second interpolator keeps these vertex values, and provides face values that correct the trace of the function to take into account the vertex values: $\ID\bv\in\UD$ is defined by
\begin{equation*}
\begin{array}{ll}
(\ID \bv )_{\Ks}  = \bv_{|K}(\x_s)& \mbox{ for all }s \in \nodes ,\, \Ks \in  \ov{\cells}_s,\, K\in \Ks,\\[0.5em]
(\ID \bv)_{\Ksig} = \dsp \frac{1}{|\sigma|} \int_{\sigma} ( \gamma^{\Ksig}{\bv} - \Pi^{\Ksig}(\IvD\bv ) ) &\mbox{ for all }  \sigma \in \faces_{\Gamma,K}^+, \,  K \in \cells.
		\end{array}
\end{equation*}
Here, $\Pi^{\Ksig}$ is the reconstructed face value defined by \eqref{PiKsigma} below (note that this reconstructed value only depends on the degrees of freedom at the vertices).

The vectorial Lagrange multiplier represents the approximations of the surface traction $-\T^+$ on $\Gamma$. Its discretisation is defined by the space $\MD$ of face-wise constant vectorial functions
$$ 
  \MD = \big\{  \l_{\D} \in L^2 (  \Gamma )^d~:~ \l_{\D}(\x) = \l_{\sigma} ~ \forall \sigma \in \faces_{\Gamma}, \forall \x \in \sigma \big\}.
$$
Let us defined the normal and tangential components of $\l_{\D}  \in \MD$ by
$$
  \lambda_{\D,\n}(\x) = \lambda_{\sigma,\n} = \l_{\sigma} \cdot \n_{\Ksig} \,\,\mbox{ and  }\,\, \l_{\D,\tang}(\x) = \l_{\sigma,\tang} = \l_{\sigma} - \lambda_{\sigma,\n} \n_{\Ksig}, \quad \text{for all}~ \x \in \sigma = K|L, ~ \sigma \in \faces_{\Gamma}.
$$
If  $\l_{\D} \in \MD$,  we define the discrete dual cone of admissible Lagrange multipliers as 
$$ 
\CD\left(\lambda_{\D,\n} \right) = \big\{ \m_{\D} = \left( \mu_{\D,\n}, \m_{\D,\tang} \right) \in \MD ~:~ \mu_{\D,\n} \ge 0,~ |\m_{\D,\tang}| \le F \lambda_{\D,\n}  \big\}.
$$

\subsubsection{Reconstruction operators}

For each face $\sigma  \in \faces_K$, $K \in \cells$, the tangential gradient reconstruction operator based on the nodal unknowns is defined by
$$
\begin{aligned}
	&  \nabla^{\Ksig} :  \UD  \to   \Po^{\,0} ( \sigma )^{d \times d}  \mbox{ such that, for all $\bv_\D\in\UD$,} \\
	 & \nabla^{\Ksig} \bv_{\D} = 	\dsp  \frac{1}{|\sigma|} \sum_{e= s_1 s_2 \in \edges_{\sigma}} |e| {\bv_{\Ks_1} + \bv_{\Ks_2} \over 2} \otimes \n_{\sige},
\end{aligned}
$$
and the linear function reconstruction operator by
\begin{equation}\label{PiKsigma}
\begin{aligned}	
  & \Pi^{\Ksig} :  \UD \to  \Po^1 ( \sigma )^d \mbox{ such that, for all $\bv_\D\in\UD$,} \\
  & \Pi^{\Ksig} \bv_{\D}(\x)  = \nabla^{\Ksig} \bv_{\D}  \left(  \x - \ov{\x}_{\sigma} \right) + \ov{\bv}_{\Ksig}\qquad\forall\x\in\sigma,
  \end{aligned}
\end{equation}
with 
\begin{equation}\label{v_x_bar_sigma}
\ov{\bv}_{\Ksig}=\sum_{s \in \nodes_{\sigma} } \omega_{s}^{\sigma}\bv_{\Ks}  ~~\text{and}~~ \ov{\x}_{\sigma} = \sum_{s \in \nodes_{\sigma} } \omega_{s}^{\sigma} \x_s.
\end{equation}
Here, $\omega_{s}^{\sigma}$ are non-negative weights associated to the center of mass of the face $\sigma$ (so that $\ov{\x}_{\sigma}$ is the center of mass of $\sigma$ and $\sum_{s\in\nodes_\sigma}\omega_s^\sigma=1$). 

For each fracture face  $\sigma= K|L \in \faces_{\Gamma}$, we introduce the displacement jump operator: 
\begin{equation*}
	\begin{aligned}
	&\jump{\cdot}_{\sigma} : \UD \to  \Po^{\,0}( \sigma )^d\mbox{ such that, for all $\bv_\D\in\UD$,}\\
	&\jump{\bv_{\D}}_{\sigma}  =  \frac{1}{|\sigma|} \int_{\sigma}\left(  \Pi^{\Ksig}  \bv_{\D} - \Pi^{\Lsig} \bv_{\D} \right) \d\sigma  +  \bv_{\Ksig},
	\end{aligned}
\end{equation*}
as well as its normal and tangential components $\jump{\cdot}_{\sigma,\n} =  \jump{\cdot}_{\sigma} \cdot \n_{\Ksig}$ and $\jump{\cdot}_{\sigma,\tang} =  \jump{\cdot}_{\sigma} - \jump{\cdot}_{\sigma,\n}\, \n_{\Ksig}$.

\begin{remark}[Discrete jump]
It can easily be checked that, if $\bv\in C^0(\overline{\Omega}\backslash\Gamma)^d$, then
$$
  \jump{\ID\bv}_{\sigma}=\frac{1}{|\sigma|} \int_{\sigma}\left(  \gamma^{\Ksig} \bv - \Pi^{\Lsig} \ID\bv \right),
$$
showing that $\jump{\ID\bv}_{\sigma}$ is an approximation of the (average of the exact) jump, using the exact trace of $\bv$ on the positive side and an approximate trace on the negative side.

In case of two bubbles (see Remark \ref{rem:twobubbles}), the jump is defined by
$$
\jump{\bv_{\D}}_{\sigma}  =  \frac{1}{|\sigma|} \int_{\sigma}\left(  \Pi^{\Ksig}  \bv_{\D} - \Pi^{\Lsig} \bv_{\D} \right) \d\sigma  +  (\bv_\Ksig-\bv_\Lsig),
$$
and, when applied to interpolated vectors, provides the average of the exact continuous jump:
$$
  \jump{\ID\bv}_{\sigma}=\frac{1}{|\sigma|}\int_\sigma\jump{\bv}.
$$
\end{remark}

For each cell $K$, we define the gradient reconstruction operator
\begin{equation}\label{eq:def.nablaK}
	\begin{aligned}
	& \nabla^{K} : \UD  \to   \Po^{\,0} ( \sigma )^{d \times d}  \mbox{ such that, for all $\bv_\D\in\UD$,}\\
	& \nabla^{K} \bv_{\D}   = \dsp \sum_{\sigma \in \faces_{\Gamma,K}^+}  \frac{|\sigma|}{|K|} \bv_{\Ksig} \otimes \n_{\Ksig}  +  \sum_{\sigma \in \faces_{K}} \frac{|\sigma|}{|K|}\ov{\bv}_{\Ksig}\otimes \n_\Ksig,
\end{aligned}
\end{equation}
where $\ov{\bv}_\Ksig$ is defined by \eqref{v_x_bar_sigma}.
The linear function reconstruction operator is
\begin{equation}\label{eq:def.piK}
  	\begin{aligned}
& \Pi^{K} :\UD \to \Po^{1}( K )^{d} \mbox{ such that, for all $\bv_\D\in\UD$,}\\
&  \Pi^{K} \bv_{\D}(\x)   = \dsp \nabla^{K} \bv_{\D}\left(  \x - \ov{\x}_{K}\right) + \ov{\bv}_{K}\quad\forall\x\in K,
\end{aligned}
\end{equation}
with 
\begin{equation*}
\ov{\bv}_{K} = \sum_{s \in \nodes_K } \omega_{s}^K \bv_{\Ks} ~~\text{and}~~ \ov{\x}_K=\sum_{s \in \nodes_K }  \omega_{s}^K  \x_s,
\end{equation*}
where $\omega_{s}^{K}$ are non-negative weights such that $\ov{\x}_K$ is the center of mass of $K$ and $\sum_{s\in\nodes_K}\omega_s^K=1$.
By construction, all these local reconstruction operators are exact on linear functions in the sense that
$\nabla^K \ID \q = \nabla \q$ and $\Pi^K \ID \q = \q$ for all $\q \in \Po^1(K)^d$, and $\nabla^\Ksig \ID \q = \nabla_\tang \q$ and $\Pi^\Ksig \ID \q = \q$, for all $\q \in \Po^1(\sigma)^d$. 

Let us now define the global discrete reconstruction $\Pi_{\D} : \UD   \to L^2 (\Omega)^d$ and the global discrete gradient $\nabla_{\D} : \UD  \to L^2(\Omega)^{d\times d}$ such that
$$
\Pi_{\D} \bv_{\D} (\x) =\Pi^K \bv_{\D}(\x) \quad \mbox{ and } \quad  \nabla_{\D} \bv_{\D} (\x)  =  \nabla^K \bv_{\D}  
\quad \mbox{ for all } \x \in K,~ K \in \cells.
$$
From the discrete gradient, we deduce the following discrete symmetric gradient, divergence and stress: 
$$
\bbeps_{\D}(\cdot) =\frac{1}{2}(\nabla_{\D}(\cdot)+     \prescript{t}{}{\nabla_{\D}}(\cdot)),~ \div_{\D}(\cdot) = \text{Tr}\left( \bbeps_{\D}(\cdot)\right) ~ \text{and}~ \bbsigma_{\D}(\cdot) = 2\mu\bbeps_{\D}(\cdot) + \lambda \div_{\D}(\cdot) \mathbb{I}.
$$
Finally, we define the discrete global displacement jump operator
\begin{equation*}
		\begin{aligned}
&	 \jump{\cdot}_{\D} : \UD  \to  L^{2} ( \Gamma )^d  \mbox{ such that, for all $\bv_\D\in\UD$,}\\
&  \jump{\bv_{\D}}_{\D}(\x)  =   \jump{\bv_{\D}}_{\sigma},~ \forall \x \in \sigma, ~ \sigma \in \faces_{\Gamma},
\end{aligned}	
\end{equation*}
as well as its normal and tangential components,
$$
\jump{\bv_{\D}}_{\D,\n} =  \jump{\bv_{\D}}_{\D} \cdot \n^{+}\quad\mbox{ and }\quad  \jump{\bv_{\D}}_{\D,\tang} =  \jump{\bv_{\D}}_{\D} - \jump{\bv_{\D}}_{\D,\n} \n^{+}. 
$$

The discrete $H^1$-like semi-norm on $\UD$ is defined by: for all $\bu_\D\in \x_\D$,
$$
\NORM{\D}{\bu_\D}:= \left(\sum_{K\in\mathcal M}\NORM{L^2(K)}{\nabla^K\bu_\D}^2+S_K(\bu_\D,\bu_\D)\right)^{1/2},
$$
where the local stabilisation term is given by the bilinear form $S_K: \UD\times\UD \to\R$ such that, for all $\bu_\D,\bv_\D\in\UD$,
\begin{equation}\label{eq:StabK}
S_K(\bu_\D,\bv_\D)= h_{K}^{d-2}  \sum_{s \in \nodes_K} (\bu_{\Ks} -\Pi^{K}\bu_{\D} (\x_s) ) \cdot ( \bv_{\Ks} -\Pi^{K}\bv_{\D}  (\x_s) ) 
		 + \sum_{\sigma \in \faces_{\Gamma,K}^+} h_{K}^{d-2} \,\bu_{\Ksig}  \cdot  \bv_{\Ksig}.
\end{equation}

\subsubsection{Discrete mixed variational formulation}
                 
We can now introduce the mixed variational discretisation of the contact-mechanical problem \eqref{Lagrange_meca_contactfriction}: Find $( \bu_{\D}^n, \l_{\D}^n )_{n=1,\ldots,N} \in (\UDz \times \CD( \lambda_{\D,\n}^n))^N$ such that, for all $( \bv_{\D}, \m_{\D} ) \in \UDz \times \CD(\lambda_{\D,\n}^n)$ and all $n=1,\cdots,N$, it holds 
  \begin{subequations}\label{mixed_discrete}
  	\begin{align}
  		& \int_{\Omega} \bbsigma_{\D} (\bu_{\D}^n): \bbeps_{\D}\left(\bv_{\D} \right)  + S_{\mu,\lambda,\D}\left( \bu_{\D}^n, \bv_{\D} \right) - \int_{\Omega} b \Pi_{\D_m} p_{\D_m}^n \div_{\D} \bv_{\D}  \nonumber \\[1.ex]
  		& + \int_{\Gamma} \Pi_{\D_f} p_{\D_f}^n  \jump{\bv_{\D}}_{\D,\n}   + \int_{\Gamma} \l_{\D}^n \cdot \jump{\bv_{\D}}_{\D}  = \sum_{K \in \cells}   \int_{K}  \mathbf{f}_K^n \cdot \Pi_{\D} \bv_{\D},  \label{eq:meca.var} \\[2.ex] 
  		& \int_{\Gamma} \left( ( \mu_{\D,\n} - \lambda_{\D,\n}^n ) \jump{\bu_{\D}^n}_{\D,\n} + ( \m_{\D,\tang} - \l_{\D,\tang}^n ) \cdot\jump{\delta_t^n \bu_{\D}}_{\D,\tang} \right) \le 0, \label{eq:meca.ineq.contact}
  	\end{align}
  \end{subequations}
  with $\dsp \mathbf{f}_K^n = \frac{1}{ |K| \Delta t^n } \int_{t^{n-1}}^{t^n} \int_{K} \mathbf{f}$ and
  the scaled stabilisation form $S_{\mu,\lambda,\D}$ defined by
\begin{equation*}
		 S_{\mu,\lambda,\D}( \bu_{\D}, \bv_{\D}) = \sum_{K\in\mathcal M}(2 \mu_K+\lambda_K)S_K(\bu_\D,\bv_D).
\end{equation*}
Thanks to the fracture face-wise constant Lagrange multiplier, the variational inequality  \eqref{eq:meca.ineq.contact} together with $\l_{\D}^n \in \CD ( \lambda_{\D,\n}^n )$ can equivalently be replaced by the following non linear equations:
\begin{equation}\label{eq:semismoothnewton}
	\left\{
	\begin{array}{l}
		\lambda_{\D,\n}^n = [ \lambda_{\D,\n}^n + \beta_{\D,\n} \jump{\bu_{\D}^n}_{\D,\n} ]_{\mathbb{R}^+} \\[0.5ex]
		\lambda_{\D,\tang}^n = [ \lambda_{\D,\tang}^n + \beta_{\D,\tang} \jump{\delta_t^n \bu_{\D}}_{\D,\tang} ]_{F\lambda_{\D,\n}^n}, 
	\end{array}
	\right.
\end{equation}
where $\beta_{\D,\n} > 0$ and $\beta_{\D,\tang} > 0$ are arbitrarily chosen face-wise constant functions along $\Gamma$, $[x]_{\R^+} = \max\{0,x\}$, and $[\cdot]_\alpha$ is the projection on the ball of radius $\alpha$ centered at $0$, that is:
$$
[{\bf x}]_\alpha = 
\left\{\begin{array}{ll}
{\bf x} & \mbox{if } |{\bf x}| \leq \alpha,\\
\displaystyle \alpha { {\bf x} \over |{\bf x}| } & \mbox{otherwise}.
\end{array}\right.
$$
The equations \eqref{eq:semismoothnewton} can be expressed locally to each fracture face and lead to an efficient semi-smooth Newton solver for the discrete contact-mechanics; see \cite{beaude2023mixed} for more details.

\subsubsection{Discretisation of the coupling conditions}\label{subsec:discrete_coupl_cond}

The set of equations \eqref{GD_DarcyFlow}--\eqref{eq:meca.var}--\eqref{eq:semismoothnewton} is combined with the following coupling conditions: 
\begin{equation}\label{discrete_coupling_law}
	\begin{array}{ll}
		&\delta_t^n \phi_{\D} = b~ \div_{\D} \left( \delta_t^n \bu_{\D}\right) + \frac{1}{M} \Pi_{\D_m} \delta_t^n p_{\D_m}, \\[1ex]
		& \d_{f,\D}^n =   \d_{f,\D}^c - \jump{\bu_{\D}^n }_{\D,\n},
	\end{array}
\end{equation}
defining the discrete porosity $\phi_{\D}^{n}$ and fracture aperture $\d_{f,\D}^n$ for all $n \ge 0$, given a cell-wise constant approximation $\phi_{\D}^0$ of the initial porosity $\phi^0$ and a face-wise constant approximation $\d_{f,\D}^c$ of the contact aperture $\d_{f}^c$.   

\subsubsection{Discrete energy estimate}\label{subsec:energyestimate}

\begin{proposition}[Discrete energy estimate] \label{prop:energy.estimate}
  Any solution $(p^n_\D,\bu^n_\D,\l_D^n)_{n=1,\ldots,N}\in (X_{\D}^0\times \UDz \times \CD( \lambda_{\D,\n}^n))^N$  of the fully coupled scheme \eqref{GD_DarcyFlow}--\eqref{eq:meca.var}--\eqref{eq:semismoothnewton}--\eqref{discrete_coupling_law} satisfies the following discrete energy estimates for all $n=1,\cdots,N$:
\begin{equation}\label{eq:eergy.estimate}
  \begin{aligned}
    &  \delta_t^n \int_{\Omega}  {1 \over 2} \( \bbsigma_{\D} (\bu_{\D}): \bbeps_{\D}\left(\bu_{\D} \right)  + S_{\mu,\lambda,\D}\left( \bu_{\D}, \bu_{\D} \right)  + {1 \over M}  |\Pi_{\D_m} p_{\D_m}|^2  \) 
      + \int_\Gamma F \lambda_{\D,\n}^n |\jump{\delta_t^n \bu_{\D}}_{\D,\tang}|  \\
    &  + \int_{\Omega} \frac{\mathbb{K}_m}{\eta} \nabla_{\D_m} p_{\D_m}^n \cdot \nabla_{\D_m} p^n_{\D_m} + 
       \int_{\Gamma}    \frac{C_{f,\D}^{n-1}}{\eta} |\nabla_{\D_f} p_{\D_f}^n|^2 
       + \sum_{\mathfrak{a} \in \{+,-\}} \int_{\Gamma} \Lambda_{f, \D}^{n-1} (\jump{ p_{\D}^n}_{\D}^{\mathfrak{a}} )^2 \\       
       & \leq \int_{\Omega} h_m \Pi_{\D_m} p^n_{\D_m}  + \int_{\Gamma} h_f \nabla_{\D_f} p^n_{\D_f}
       + \sum_{K \in \cells}   \int_{K}  \mathbf{f}_K^n \cdot \Pi_{\D} \delta_t^n \bu_{\D}.  
  \end{aligned}  
\end{equation}  
In addition, the discrete fracture aperture satisfies the lower bound
$\d_{f,\D}^n \geq   \d_{f,\D}^c$, which ensures the positivity of the fracture conductivity $C_{f,\D}^{n-1}$ and normal transmissivity $\Lambda_{f, \D}^{n-1}$. 
\end{proposition}

\begin{proof}
  We only recall the main steps of the proof, which follows the lines of the proof of \cite[Eq.~(19)]{BDMP:21}.  We first recall \cite[Lemma~4.1]{BDMP:21} which shows that an equivalent form of \eqref{eq:semismoothnewton} is
  \begin{equation}\label{eq:equiv.var}
  \begin{aligned}
    &  \lambda_{\D,\n}^n \geq 0, \quad  \jump{\bu_{\D}^n}_{\D,\n} \leq 0, \quad \lambda_{\D,\n}^n \jump{\bu_{\D}^n}_{\D,\n} = 0,\\
  &   |\lambda_{\D,\tang}^n| \leq F \lambda_{\D,\n}^n, \quad \lambda_{\D,\tang}^n\cdot \jump{\delta_t^n \bu_{\D}}_{\D,\tang} - F \lambda_{\D,\n}^n |\jump{\delta_t^n \bu_{\D}}_{\D,\tang} | = 0. 
  \end{aligned}
  \end{equation}
    As in the proof of \cite[Theorem~4.2]{BDMP:21}, the following discrete persistency condition follows from \eqref{eq:equiv.var}: $\lambda_{\D,\n}^n \jump{\delta_t^n \bu_{\D}}_{\D,\n} \geq 0$. In turn, this condition yields the dissipative property of the contact term:
  \begin{equation}\label{eq:dissipiativity.contact}
  \int_\Gamma \l_{\D}^n \cdot \jump{\delta_t^n\bu_{\D}}_{\D}  \geq \int_\Gamma F \lambda_{\D,\n}^n |\jump{\delta_t^n \bu_{\D}}_{\D,\tang}|  \geq 0. 
  \end{equation}
   Then, setting $\bv_\D = \delta_t^n\bu_{\D}$ in \eqref{eq:meca.var} and $q_{\D_m} = p^n_{\D_m}$ in  \eqref{GD_DarcyFlow}, taking into account the coupling equations  \eqref{discrete_coupling_law} and \eqref{eq:dissipiativity.contact}, we obtain \eqref{eq:eergy.estimate}. The lower bound on the discrete fracture aperture follows directly from \eqref{discrete_coupling_law} and $\jump{\bu_{\D}^n}_{\D,\n} \leq 0$ as stated in \eqref{eq:equiv.var}.
\end{proof}

Following \cite{BDMP:21}, in order to deduce from \eqref{eq:eergy.estimate} a priori estimates and the existence of a discrete solution, we need to establish a discrete Korn inequality in $\UDz$, and a discrete inf-sup condition for the bilinear form $\int_{\Gamma} \l_{\D} \cdot \jump{\bv_{\D}}_{\D}$ in $\UDz\times \MD$. This is a work in progress, which requires new developments related to the additional bubble unknowns and to fracture networks including tips and intersections.

\subsection{An equivalent VEM formulation of the contact-mechanics}\label{equiv_VEM}

We show here that the previous discretisation of the mechanics has an equivalent VEM formulation, based on the same displacement degrees of freedom. The scheme we propose can therefore be interpreted as a $\Po^1$-bubble VEM discretisation. 

The VEM framework provides an extension of Finite Element Methods (FEM) to polyhedral meshes, see the seminal paper \cite{beirao2013basic}. As for FEM, it builds a subspace $\bV_h$ of $\U$ by gluing together local spaces $\bV_h^K\subset H^1(K)^d$ defined in each cell $K\in \cells$. On the other hand, the basis functions do not have in general an analytical expression, and only certain projections of them onto polynomial spaces can be explicitly calculated from the degrees of freedom. The bilinear form is then obtained from the continuous one using these projections, and by stabilising their kernel. In the following, we first exhibit the connection between the VEM face $\pi^{\Ksig}$ and cell $\pi^K$ projectors and the reconstruction operators $\Pi^{K\sigma}$ and $\Pi^K$. Then, the VEM local and global spaces are defined leading to the stabilised bilinear form and the equivalent VEM mixed variational formulation of the contact-mechanical problem. The unisolvence of the degrees of freedom in the VEM space $\bV_h$ is also shown. Detailed proofs are reported to Appendix \ref{appendice_VEM}.

\subsubsection{VEM projectors, function spaces and equivalent mixed variational formulation}

For each $K \in \cells$, let $\pi^K$ be the local projection operator defined  by:
\begin{equation}\label{piK}
	\begin{aligned}	
	&	\pi^K : \mathcal{C}^0(\ov{K})^d  \to  \Po^1(K)^{d}  \mbox{ such that, for all $\bv \in \mathcal{C}^0 ( \ov{K})^d$},\\
	& \pi^K  \bv   =\Pi^K \circ \ID \bv ,
	\end{aligned}
\end{equation}
where, by abuse of notation, the interpolator $\ID$ is applied to the extension by zero of $\bv$ outside $\ov{K}$. 
Similarly, for each $\sigma \in \faces_K$, $K \in \cells$,  let $\pi^{\Ksig}$ be the local projection operator defined by:
\begin{equation}\label{pisigma}
	\begin{aligned}	
	&	\pi^{\Ksig} : \mathcal{C}^0 ( \ov{\sigma})^d  \to \Po^1(\sigma)^{d}  \mbox{ such that, for all $\bv \in \mathcal{C}^0 ( \ov{\sigma} )^d$},\\
		& \pi^{\Ksig} \bv     =\Pi^{\Ksig} \circ \ID \bv .
\end{aligned}
\end{equation}
The local VEM space for the displacement field on each cell $K \in \cells$ is
\begin{equation}\label{eq:def.VhK}
	\begin{split}
		\bV_{h}^{K} = 
		\Big\{ \bv \in  \mathcal{C}^0 ( \ov{K} )^d|~ & \gamma^{\Ksig}{\bv} \in  \bV_{h}^{\Ksig} ,~ \forall  \sigma \in \faces_K;\\
                &\Delta \bv \in \Po^1 ( K )^d;~ \int_{K} (\pi^{K} \bv) \cdot \mathbf{p} =  \int_{K}  \bv \cdot \mathbf{p}, ~ \forall \mathbf{p} \in \Po^1(K)^d  \Big\},
	\end{split} 
\end{equation}
with
\begin{equation}\label{eq:def.VKsig}
	\begin{split}
		\bV_{h}^{\Ksig} = 
		\Big\{ \bv \in  \mathcal{C}^0(\ov{\sigma})^d |~& \gamma^{\sige}{\bv}  \in  \Po^1(e)^d,~ \forall e \in \edges_{\sigma};\\ &\Delta_{\tang} \bv \in \Po^1(\sigma)^d ~\text{in}~ \sigma;~ \int_{\sigma} (\pi^{\Ksig} \bv) \cdot \mathbf{p} =  \int_{\sigma}  \bv \cdot \mathbf{p}, ~ \forall \mathbf{p} \in (\Q_{\Ksig})^d   \Big\},
	\end{split} 
\end{equation}
where $\Q_\Ksig = \Po^1(\sigma)$ if $\sigma \in \faces_K \setminus \faces_{\Gamma,K}^+$ -- corresponding to the no-bubble case -- and
$\Q_\Ksig$ is a complementary space of constant functions in $\Po^{1}(\sigma)$ if $\sigma \in \faces_{\Gamma,K}^+$ -- corresponding to the bubble case.   

\begin{lemma}[Link between discrete reconstructions and elliptic projectors] \label{lem:reconstruction.projection}
For all $K \in \cells$, the projector $\pi^K:\bV_h^K\to \Po^1(K)^d$ is the elliptic projector, that is, it satisfies: for all $ \bv\in\bV_h^K$,
\begin{subequations}
	\label{proj_p1_K}
	\begin{align}
		& \int_{K} \nabla (\pi^K\bv): \nabla \q =\int_{K} \nabla \bv: \nabla \q \qquad \forall \q \in \Po^1(K)^{d},\label{proj_1_1}\\[1.5ex]
    	&	  (\pi^K\bv)(\ov{\x}_K)  = 	\sum_{s \in \nodes_{K} } \omega_{s}^{K}  \bv(\x_s).\label{proj_1_2}
	\end{align}
\end{subequations}
For all $\sigma\in \faces_K$, the projector $\pi^\Ksig:\bV_h^\Ksig\to\Po^1(\sigma)^d$ is the elliptic projector for the tangential gradient, that is, it satisfies: for all $\bv\in\bV_h^\Ksig$,
\begin{subequations}
	\label{proj_p1_sigma}
	\begin{align}
		& \int_{\sigma} \nabla_{\tang} (\pi^{\Ksig}\bv): \nabla_{\tang}\q  =\int_{\sigma} \nabla_{\tang} \bv: \nabla_{\tang}\q \qquad \forall \q \in \Po^1(\sigma)^{d},\label{proj_2_1}\\[1.5ex]
		& (\pi^{\Ksig}\bv)(\ov{\x}_{\sigma})  = 	\sum_{s \in \nodes_{\sigma} } \omega_{s}^{\sigma}  \bv(\x_s).\label{proj_2_2}
	\end{align}
\end{subequations}
\end{lemma}

\begin{proof}
See Appendix \ref{App_proj}.
\end{proof}

As a consequence of \eqref{proj_p1_K} and the fact that $\bbsigma$ has constant coefficients on each cell, we have, for all $K\in\cells$ and $\bv\in\bV_h^K$,
\begin{equation}\label{eq:proj_sigma}
\int_{K} \bbsigma (\pi^K\bv): \bbeps (\q ) =\int_{K} \bbsigma (\bv): \bbeps (\q ) \qquad \forall \q \in \Po^1(K)^{d}.
\end{equation}
For $\bv \in \bV^K_h$, the VEM local degrees of freedom are the same as in the unknows of the fully discrete setting, namely the nodal value $\bv_{\Ks} = \bv(\x_s) = (\ID \bv)_\Ks$ at each node $s\in \nodes_K$ and the bubble value $\bv_\Ksig = (\ID \bv)_{\Ksig}$ for each $\sigma \in \faces_{\Gamma,K}^+$  (see Figure \ref{fig_bul}).  Note that the bubble unknown can also be expressed using the local projector $\pi^{\Ksig}$ as follows:
\begin{equation}\label{dof_bubble}
\bv_\Ksig = (\ID\bv)_{\Ksig} = \frac{1}{|\sigma|} \int_{\sigma} \( \gamma^{\Ksig}{\bv} - \pi^{\Ksig}({\gamma^{\Ksig}\bv} )\).	
\end{equation}

\begin{figure}[H]
	\begin{center}
		\includegraphics[width=10cm]{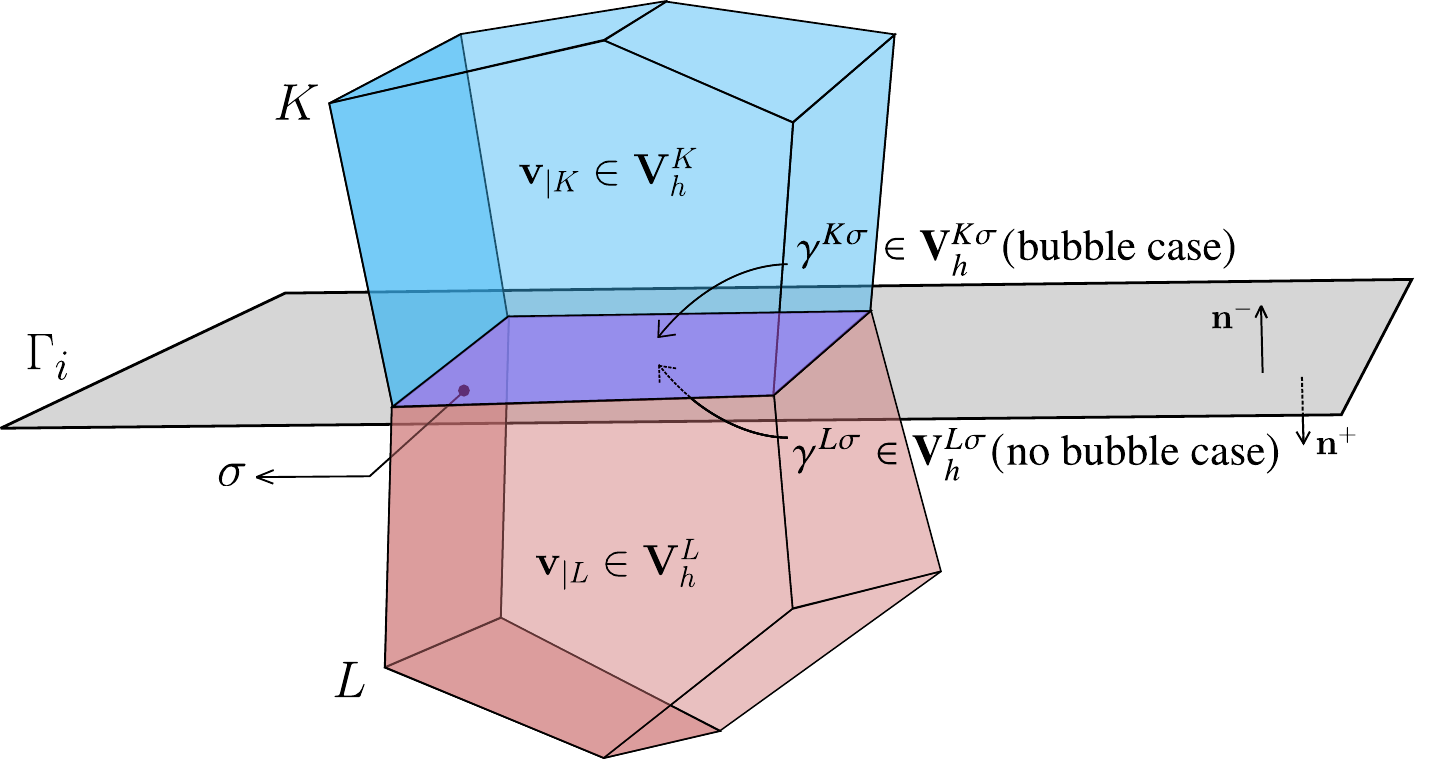}
		\caption{$\Po^1$-bubble VEM}\label{fig_bul}
	\end{center}
\end{figure}
The global VEM space for the displacement field is obtained as usual by patching together the local VEM spaces in a conforming way in $H^1( \Omega\setminus \ov{\Gamma})^d$. It is defined by
\begin{equation*}
	\begin{split}
		\bV_{h} = 
		\Big\{ \bv \in  H^1( \Omega\setminus \ov{\Gamma})^d \cap\,  \mathcal{C}^0( \Omega\setminus \ov{\Gamma})^d |~  \bv_{|K} \in \bV_h^{K},~\forall K \in \cells  \Big\},
	\end{split} 
\end{equation*}
and we denote by $\bV_h^0$ its subspace  with vanishing values on the boundary $\partial \Omega$. 
Note that the vector of all degrees of freedom of $\bv \in \bV_h$ is precisely $\ID \bv \in \UD$. We define $\pi^h$ as the global projection operator onto the broken polynomial space $\mathbb{P}^1(\mathcal{M})^d$, such that, for all $\bv \in \mathbf{V}_h$, $(\pi^h \bv)_{|K} = \pi^K ({\bv}_{|K})$. The diagram \eqref{eq:commut} illustrates the fundamental relation between $\pi^h$ and $\Pi_{\D}$.
\begin{equation}\label{eq:commut}
\begin{tikzcd}[column sep=huge, row sep=huge]
	\mathbf{V}_{h}^0 \arrow[r,"\mathcal{I}_{\D}"] \arrow[rd,"\pi^h"] & \mathbf{U}_{\D}^0 \arrow[d,"\Pi_{\D}"] \\
	{} & \mathbb{P}^{1}(\cells)^d
\end{tikzcd}
\end{equation}
The mixed variational formulation for the contact-mechanical problem in the $\Po^1$-bubble VEM framework is defined by: find $( \bu^n, \l_{\D}^n )_{n=1,\ldots,N} \in (\bV_{h}^0 \times \CD( \lambda_{\D,\n}^n ))^N$  such that, for all $n=1,\ldots,N$ and all $( \bv, \m_{\D} ) \in \bV_{h}^0 \times \CD(\lambda_{\D,\n}^n)$,
\begin{subequations}\label{mixed_vem}
	\begin{align}
		& \dsp \sum_{K\in\cells}\int_{K} \bbsigma (\pi^K \bu^n): \bbeps(\pi^K\bv )  + S_{\mu,\lambda,\D}( \ID\bu^n, \ID\bv ) -  \sum_{K\in\cells}\int_{K} b  \Pi_{\D_m} p_{\D_m}^n  \div ( \pi^K \bv )  \nonumber \\[1ex]
		& \qquad+  \sum_{\sigma\in\faces_{\Gamma}}\int_{\sigma}  \Pi_{\D_f} p_{\D_f}^n  \jump{\bv }_{\n}   + \sum_{\sigma\in\faces_{\Gamma}}\int_{\sigma} \l_{\D}^n \cdot \jump{\bv } = \sum_{K \in \cells}  \int_{K}  \mathbf{f}^n_K \cdot \pi^K \bv, \\[2ex]
		& \sum_{\sigma\in\faces_{\Gamma}}\int_{\sigma} \(( \mu_{\D,\n} - \lambda_{\D,\n}^n ) \jump{\bu^n}_{\n} + ( \m_{\D,\tang} - \l_{\D,\tang}^n ) \cdot\jump{\delta_t^n \bu }_{\tang} \) \le 0.
	\end{align}
\end{subequations}
It is straightforward to observe that the variational formulation \eqref{mixed_vem} is equivalent to \eqref{mixed_discrete} based on the correspondence $\bu_\D = \ID \bu$ and $\bv_\D = \ID \bv$. 

We note that the term $\int_{\sigma} \jump{\bu }\d\sigma$ is computable from the degrees of freedom since, with $\sigma=K|L$,
$$
\int_{\sigma} \jump{\bu }\d\sigma= \int_\sigma (\gamma^\Ksig\bu-\gamma^\Lsig\bu)= \int_{\sigma} \( \pi^{\Ksig} (\gamma^{\Ksig} \bu ) - \pi^{\Lsig} (\gamma^{\Lsig} \bu ) \) +|\sigma| (\ID \bu )_{\Ksig},
$$
where we have used \eqref{dof_bubble} to express $\int_\sigma \gamma^\Ksig\bu$, and the condition in \eqref{eq:def.VKsig} for $\bV_h^\Lsig$ with $\p$ constant (which is valid since $L$ is not on the bubble side of $\sigma$) to write $\int_\sigma \gamma^\Lsig \bu = \int_\sigma \pi^\Lsig(\gamma^\Lsig\bu)$. 
 
The stabilisation term $S_{\mu,\lambda,\D}( \ID\bu^n, \ID\bv )$ matches the classical VEM "$\textit{dofi}$-$\textit{dofi}$" approach \cite{da2015virtual} based on the degrees of freedom (see Appendix \ref{App_stab} for a detailed proof).   
The consistency of the local bilinear form 
\begin{equation*}
a_{\D}^K( \bu,\bv) =  \dsp \int_{K} \bbsigma (\pi^K \bu): \bbeps(\pi^K\bv )  + (2 \mu_K+\lambda_K)S_K(\bu_\D,\bv_D)
\end{equation*}
derives from \eqref{eq:proj_sigma} and the fact that $\pi^K \mathbf{w} = \mathbf{w}$ for all $\mathbf{w} \in \Po^1( K)^d$.

\subsubsection{Unisolvence of the degrees of freedom}

\begin{proposition}
For all $K\in\cells$, the local degrees of freedom associated to $K$ are unisolvent for
$\bV_h^K$. As a consequence, the degrees of freedom in $\UD$ are unisolvent for $\bV_h$.
\end{proposition}

\begin{proof}
	Let us consider a mesh element $K$ that includes at least a fracture face (the case without any fracture face being done similarly, and in a simpler way). Let us show that the local interpolation operator ${\ID}_{|{\bV_h^K}}:\bV_h^K \to \UD$ that extracts the degrees of freedom from a given function of $\bV_h^K$ is injective. We therefore need to prove that any function $\bv \in \bV_h^K$ satisfying		
\begin{subequations}
	\begin{align}
		&( \ID \bv )_{\Ks} = 0, \quad \forall s \in \nodes_K, \label{eq:deg_zero_1} \\
		&( \ID \bv )_{\Ksig} = 0,\quad \forall \sigma \in \faces_{\Gamma,K}^+, \label{eq:deg_zero_2}
	\end{align}
\end{subequations}	
vanishes in $K$. From \eqref{eq:deg_zero_1}, we get that $\gamma^{\Ksig} \bv = 0$ on $\partial \sigma$, for all $\sigma \in \faces_K$. In order to show that $\gamma^{\Ksig}\bv = 0$ on $\sigma$, it suffices to show that 	
	\begin{equation}\label{douze}
		\int_{\sigma} ( \gamma^{\Ksig}\bv ) \cdot \mathbf{p} = 0,\quad \forall \mathbf{p}  \in \Po^1(\sigma)^d.
	\end{equation}	
	Indeed, suppose that \eqref{douze} is satisfied, then one can write	
	\begin{equation*}
		0 \underbrace{=}_{\Delta_{\tang} \gamma^{\Ksig}\bv \in \Po^1( \sigma)^d}	\int_{\sigma} ( \gamma^{\Ksig} \bv ) \cdot \Delta_{\tang} \gamma^{\Ksig} \bv  \underbrace{=}_{ \gamma^{\Ksig}\bv \,=\, 0 ~ \text{on}~ \partial \sigma}	\int_{\sigma}  \grad_{\tang} \gamma^{\Ksig}\bv \cdot \grad_{\tang} \gamma^{\Ksig} \bv ,
	\end{equation*}	
	which directly implies $ \gamma^{\Ksig} \bv = 0$ on $\sigma$ since $\gamma^{\Ksig} \bv = 0$ on $\partial \sigma$. By the integral condition in \eqref{eq:def.VKsig}, we have, for all $\mathbf{p}\in(\Q_\Ksig)^d$,
	 \begin{equation}\label{eq:int.Qsig}
       \int_{\sigma} (\gamma^{\Ksig} \bv ) \cdot \mathbf{p}  = \int_{\sigma} \pi^{\Ksig} (\gamma^{\Ksig}\bv )\cdot \mathbf{p} = 0,
   \end{equation}
   where the conclusion follows from \eqref{eq:deg_zero_1} which implies $\pi^{\Ksig} (\gamma^{\Ksig}\bv )=0$. If $\sigma\not\in \faces_{\Gamma,K}^+$ then $\Q_\Ksig=\Po^1(\sigma)$ and \eqref{douze} follows. Otherwise, using \eqref{eq:deg_zero_2}, \eqref{dof_bubble} and $\pi^{\Ksig} (\gamma^{\Ksig}\bv )=0$ we have $\int_{\sigma}  \gamma^{\Ksig}{\bv}=0$. Combined with \eqref{eq:int.Qsig} which is valid for any $\mathbf{p}$ in a complement space of $\Po^0(\sigma)^d$, this proves that \eqref{douze} also holds. 
 It results that $\bv = 0$ on $\partial K$. 
 
 We then repeat a similar (but simpler, since the integral condition in \eqref{eq:def.VhK} is already expressed against test functions in $\Po^1(K)^d$) procedure on $K$, to finally obtain $\bv = 0$ on $K$. Therefore,  ${\ID}_{|{\bV_h^K}}$ is injective.
        Proceeding as in \cite{ahmad2013equivalent}, it is easy to show that $\text{dim}( \bV_h^K) \ge d \cdot ( \#\nodes_K + \#\faces_{\Gamma,K}^+ )$
which implies that $\ID$ defines a bijection from $\bV_h^K$ to the vector space $(\R^d)^{\#\nodes_K + \#\faces_{\Gamma,K}^+}$ of degrees of freedom of the cell $K$. Consequently, the degrees of freedom in $K$ are unisolvent for $\bV_h^K$. 
\end{proof}

\section{Numerical experiments}
\label{num.experiments}

We assess here the numerical convergence of the discretisation of the poromechanical model with frictional contact at matrix--fracture interfaces defined by \eqref{GD_DarcyFlow}--\eqref{eq:meca.var}--\eqref{eq:semismoothnewton}--\eqref{discrete_coupling_law}. 
Section \ref{test_meca_standalone} investigates the discretisation of the contact-mechanics on the stand-alone static contact-mechanical model. Then, in Section \ref{test_poromeca}, the discretisation of the fully coupled poromechanical model is considered. 

In the following test cases, the Lam\'e coefficients can be defined from the Young modulus $E$ and the Poisson coefficient $\nu$
by $\mu = \frac{E}{2(1+\nu)}$ and $\lambda = \frac{\nu E}{(1+\nu)(1-2\nu)}$.
The 2D test cases are performed with the 3D code using meshes obtained by extrusion in the $z$ direction of the 2D meshes of the $(x,y)$ domain, with one layer of cells of thickness $1$. The $z$ components of the displacement field and of the Lagrange multiplier are set to zero and homogeneous Neumann boundary conditions are imposed at $z=0$ and $z=1$. The resulting discretisation is equivalent to the 2D version of the scheme.
Note that the discrete fracture networks of Sections \ref{bergen_meca}, \ref{bergen_coup} and \ref{cubic_coup} are chosen to include difficulties representative of the geological complexity, such as fractures with corners, fractures intersecting the boundary, or fractures intersecting each other. 

\subsection{Stand alone static contact-mechanics}\label{test_meca_standalone}

The numerical convergence of the mixed $\Po^1$-bubble VEM--$\Po^0$ discretisation \eqref{mixed_discrete}--\eqref{eq:semismoothnewton} is investigated on three static contact-mechanical test cases obtained from \eqref{eq_edp_contact_meca}--\eqref{eq_edp_stress}--\eqref{eq_def_T}  by setting the matrix $p_m$ and fracture $p_f$ pressures to zero and replacing $\jump{\partial_t \bu}_\tang$ by $\jump{\bu}_\tang$ in the contact term.
The first one considers a manufactured 3D analytical solution with a single non-immersed fracture and a frictionless contact model. The second test case is based on an analytical solution for a single fracture in contact slip state immersed in an unbounded 2D domain. The last test case compares our discretisation to a Nitsche $\Po^1$ Finite Element Method (FEM) on a 2D domain with 6 fractures. In all simulations, the contact-mechanical model is solved using the semi-smooth Newton method based on the contact equations \eqref{eq:semismoothnewton}. It is combined with a direct sparse linear solver.

\subsubsection{3D manufactured solution for a frictionless static contact-mechanical model}\label{analytical_test}

We consider the 3D domain $\Omega = (-1,1)^3$ with the single non-immersed fracture $\Gamma =  \{0\} \times (-1,1)^2$. The friction coefficient $F$ is set to zero, which corresponds to a frictionless contact, and the Lam\'e coefficients are $\mu=\lambda = 1$. The exact solution 
\begin{equation*}
	\bu(x,y,z)= 
	\left\{\hspace{.2cm}
	\begin{array}{lll}
		\left( {\begin{array}{c}
				g(x,y) p(z)\\ 
				p(z) \\   
				x^2 p(z) \\ 
		\end{array} } \right)  & \mbox{ if } z \ge 0,\\[1ex]
		\left( {\begin{array}{c} 	
				h(x) p^+(z)\\    	
				h(x) \left(p^+(z)\right)' \\   
				-\int_{0}^{x} h(\xi) \d \xi  \left(p^+(z)\right)'\\ 
		\end{array} } \right)  & \mbox{ if } z < 0,~x\ < 0,  \\[1ex]
		\left( {\begin{array}{c} 	
				h(x) p^-(z)\\    	
				h(x) \left(p^-(z)\right)' \\   
				-\int_{0}^{x} h(\xi) \d \xi  \left(p^-(z)\right)'\\ 
		\end{array} } \right)  & \mbox{ if } z < 0,~x\ \ge 0,
	\end{array}
	\right.
\end{equation*}
with $g(x,y) = -\sin(\frac{\pi x}{2}) \cos(\frac{\pi y}{2})$, $p(z) = z^2 $, $h(x) = \cos(\frac{\pi x}{2}) $, $p^+(z) = z^4 $ and $p^-(z) = 2z^4$, is designed to satisfy the frictionless contact conditions at the matrix--fracture interface $\Gamma$. The right hand side
$\mathbf{f} = -\div \bbsig(\bu)$ and the Dirichlet boundary conditions on $\partial\Omega$ are deduced from $\bu$.
Note that the fracture $\Gamma$ is in contact state for $z  > 0$ ($\jump{\bu}_\n = 0$) and open for $z < 0$, with a normal jump $\jump{\bu}_\n = -\min(z,0)^4$ depending only on $z$. The convergence of the mixed $\Po^1$-bubble VEM--$\Po^0$ formulation is investigated on families of uniform Cartesian, tetrahedral and hexahedral meshes. Starting from uniform Cartesian meshes, the hexahedral meshes are generated by random perturbations of the nodes and by cutting non-planar faces into two triangles (see Figure \ref{rand_mesh}). 
\begin{figure}[H]
	\begin{center}
       \includegraphics[height=8em]{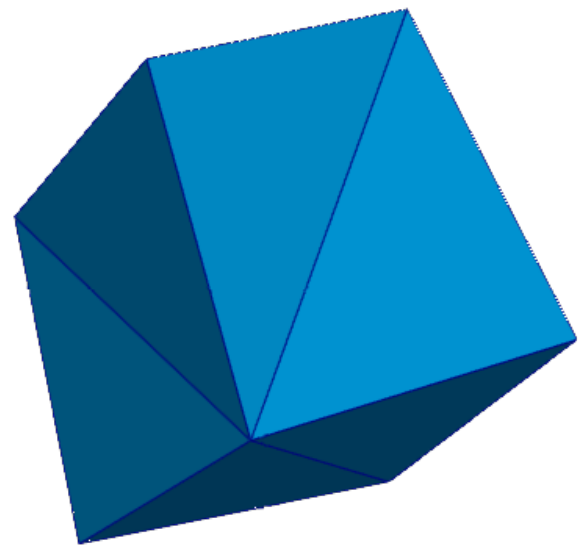}
		\caption{Example of randomly perturbated Cartesian cell with non planar faces cut into two triangles.}
		\label{rand_mesh}
	\end{center}
\end{figure} 
Figure \ref{Error_anal} exhibits the relative $L^2$ norms of the errors $\bu- \Pi_\D \bu_\D$, $\jump{\bu} - \jump{\bu_\D}_\D$, $\grad{\bu}- \nabla_\D \bu_\D$ and $\lambda_\n - \lambda_{\D,\n}$ on the three family of refined meshes as functions of the cubic root of the number of cells.
It shows, as expected for such a smooth solution, a second-order convergence for $\bu$ and $\jump{\bu}$ for all families of meshes. A first-order convergence is obtained for $\grad{\bu}$ and $\lambda_\n$ with both the hexahedral and tetrahedral families of meshes, while a second order super convergence is observed with the family of Cartesian meshes.
\begin{figure}[H]
	\centering
	\begin{tikzpicture}
				\node (img)  {\includegraphics[scale=0.48]{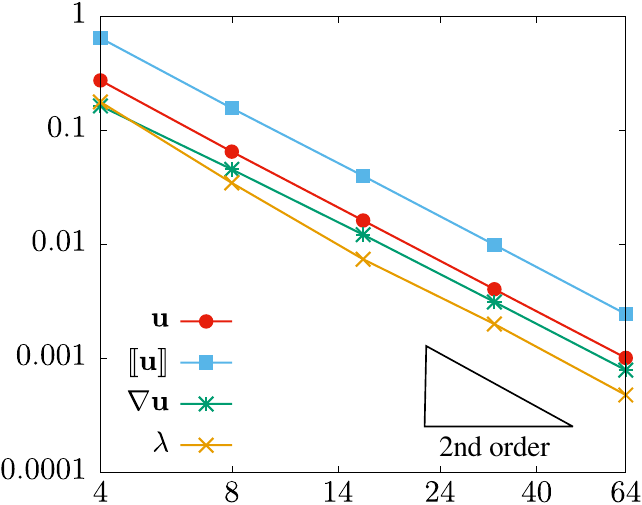}};
				\node[left=of img, node distance=0cm, rotate=90, anchor=center,yshift=-0.7cm] {\footnotesize{ $L^2$ Error }};
				\node[below=of img, node distance=0cm, yshift=1cm]  {~~~~~~~\footnotesize{ $N_{\text{cell}}^{\frac{1}{3}}$ }};
				\node[above=of img, node distance=0cm, rotate=0, anchor=center,yshift=-0.7cm]   {\footnotesize{ ~~~~(a)}};
   	\end{tikzpicture}
		\begin{tikzpicture}
			\node (img)  {\includegraphics[scale=0.48]{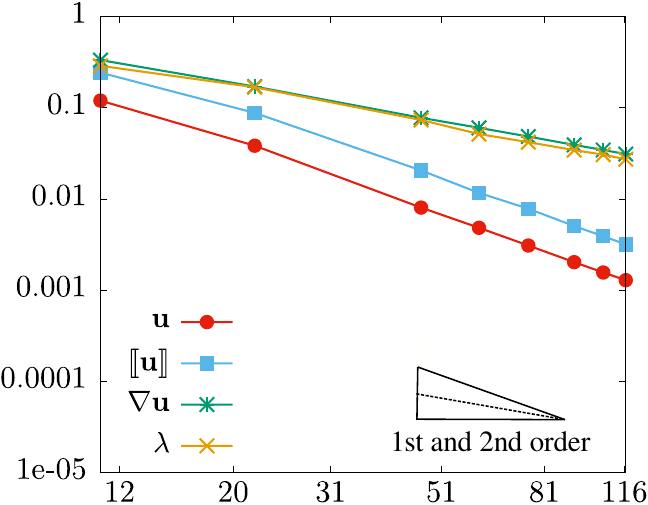}};
			\node[below=of img, node distance=0cm, yshift=1cm]  {~~~~~~~\footnotesize{ $N_{\text{cell}}^{\frac{1}{3}}$ }};
			\node[above=of img, node distance=0cm, rotate=0, anchor=center,yshift=-0.7cm]   {\footnotesize{ ~~(b)}};
	\end{tikzpicture}
	\begin{tikzpicture}
		\node (img)  {\includegraphics[scale=0.48]{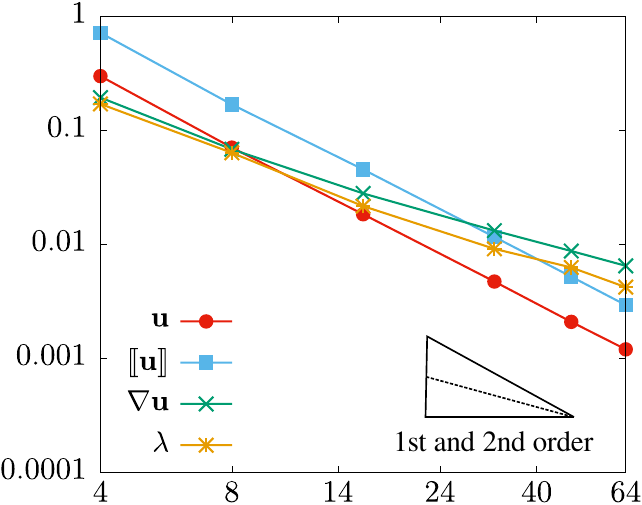}};
		\node[below=of img, node distance=0cm, yshift=1cm]  {~~~~~~~\footnotesize{ $N_{\text{cell}}^{\frac{1}{3}}$ }};
		\node[above=of img, node distance=0cm, rotate=0, anchor=center,yshift=-0.7cm]   {\footnotesize{ ~~~~~(c)}};
    \end{tikzpicture}
	\caption{Relative $L^2$ norms of the errors $\bu- \Pi_\D \bu_\D$, $\jump{\bu} - \jump{\bu_\D}_\D$, $\grad{\bu}- \nabla_\D \bu_\D$ and
$\lambda_\n - \lambda_{\D,\n}$ as functions of the cubic root of the number of cells, using the families of Cartesian (a), tetrahedral (b) and hexahedral (c) meshes. Test case of Section \ref{analytical_test}.}
	\label{Error_anal}
\end{figure}
Figure \ref{test_analytical_saut_n} plots, for the hexahedral meshes, the face-wise constant normal jump $\jump{\bu_\D}_{\D,\n}$ on $\Gamma$ and the nodal normal jumps as functions of $z$ along the ``broken'' line corresponding, before perturbation of the mesh, to $x=y=0$. We recall that the continuous normal jump depends only on $z$. 
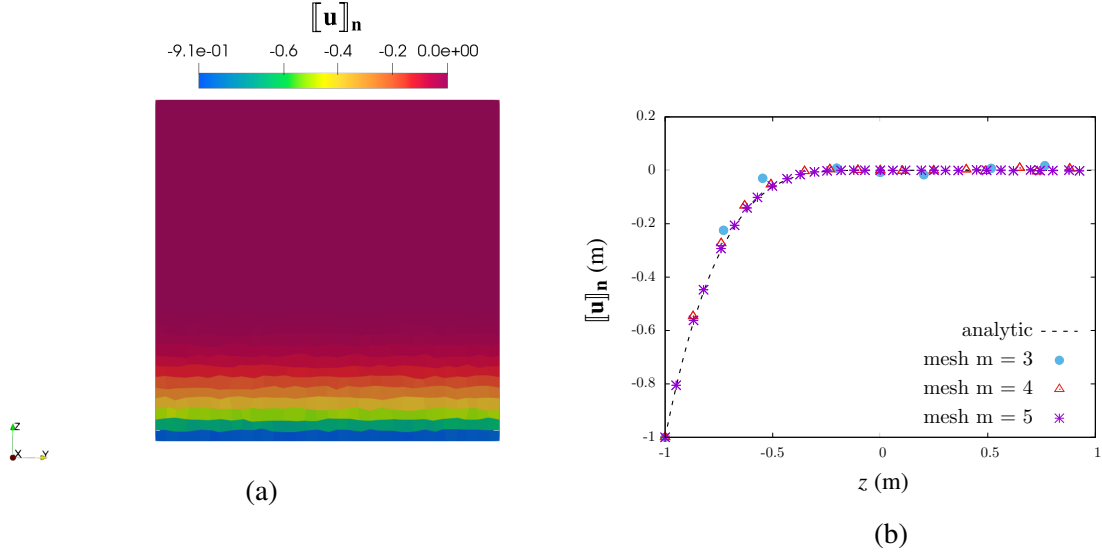
\begin{figure}[H]
	\centering
	\hspace{-2.5cm}
	\begin{tikzpicture}
		\raisebox{0.38cm}{\node (img)  {\includegraphics[scale=0.25,trim=0em 0em 20em 0em]{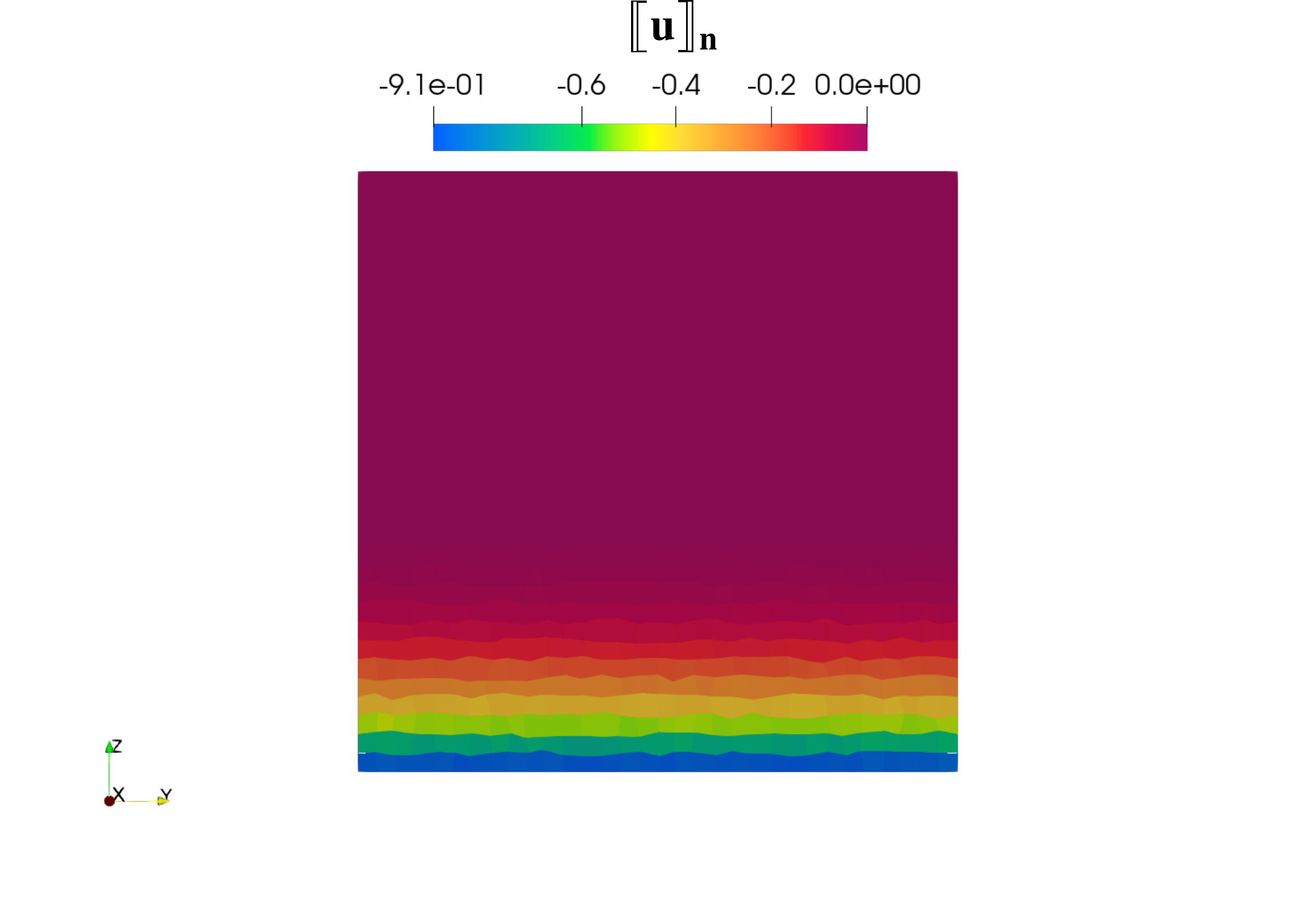}};
     \node[below=of img, node distance=0cm, rotate=0, anchor=center,yshift=1.5cm] {~{ (a)}};}
	\end{tikzpicture}
	\begin{tikzpicture}
	\node (img)  {\includegraphics[scale=0.55]{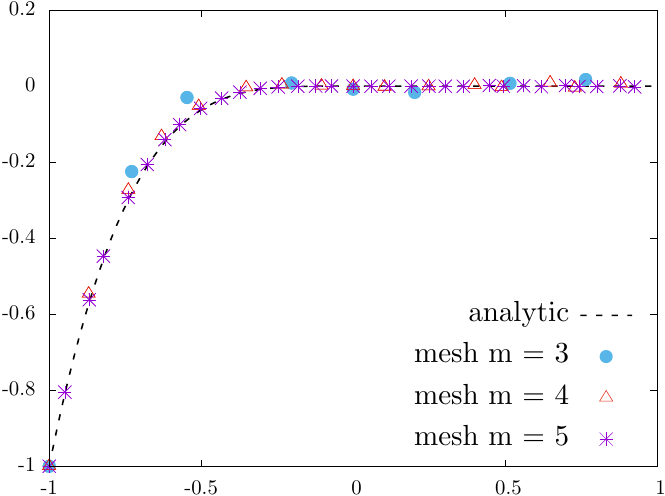}};
	\node[left=of img, node distance=0cm, rotate=90, anchor=center,yshift=-0.7cm] {\footnotesize{ $\jump{\bu }_{\mathbf{n}}$ (m)}};
	\node[below=of img, node distance=0cm, rotate=0, anchor=center,yshift=0.8cm] {~~~~\footnotesize{ $z$ (m)}};
	\node[below=of img, node distance=0cm, rotate=0, anchor=center,yshift=0.1cm] {~~~~~~{ (b)}};
	\end{tikzpicture}
	\caption{(a) Face-wise constant normal jump $\jump{\bu_\D}_{\D,\n}$ on $\Gamma$ obtained on the hexahedral mesh with $2^{3m}$ cells, $m=5$. (b) Nodal normal jumps along the line $x=y=0$ as functions of $z$ both for the discrete solutions on the hexahedral meshes with $2^{3m}$ cells, $m=3,4,5$ and for the continuous solution depending only on $z$. Test case of Section \ref{analytical_test}.}
	\label{test_analytical_saut_n}
\end{figure}

\subsubsection{Unbounded 2D domain with a single fracture under compression}
\label{tchelepi_meca}
This test case presented in~\cite{contact-BEM,tchelepi-castelletto-2020,GKT16} consists of a 2D unbounded domain containing a single fracture and subject to a compressive remote stress $\ov\sigma$ = 100 MPa. The fracture inclination with respect to the $x$-direction is $\psi=\pi/9$, its length is $2 \ell=2$ m, and the friction coefficient is $F=1/\sqrt{3}$. Young's modulus and Poisson's ratio are set to $E=25$ GPa and $\nu=0.25$. The analytical solution is such that: 
\begin{equation}\label{sol.compression}
	\lambda_{\n}  = \ov\sigma \sin^2(\psi),\quad | {\jump{\bu } }_{\tang}| = \frac{4(1-\nu)}{E}  \ov\sigma \sin (\psi) \left( \cos(\psi) - F\sin(\psi)\right)\sqrt{\ell^2 - (\ell^2 - \tau^2)},
\end{equation}
where $0\le \tau\le2\ell$ is the curvilinear abscissa along the fracture. Since $\lambda_{\n}>0$, we have $\jump{\bu}_{\n}=0$ on the fracture. For this simulation, we sample a large square domain $(-{L\over 2}, {L\over 2})^2$ with $L= 160 \,\rm m$. 
The top and bottom boundaries are free, while a compression $\bbsigma(\bu)\n = - \ov\sigma \n$ is imposed at the left and right boundaries. Moreover, to get rid of the rigid body motions while preserving the symmetry of the expected solution and compression boundary condition, homogeneous Dirichlet boundary conditions are imposed on $u_x$ at $\x = (0,\pm {L\over 2})$ and on $u_y$ at $\x = (\pm {L\over 2},0)$ as shown in Figure~\ref{test_compression}. 
An initial triangular mesh of the domain is created, with a local refinement in a neighborhood of the fracture; this mesh is then uniformly refined to give rise to meshes containing 100, 200, 400, and 800 faces on the fracture (corresponding, respectively, to 12\,468, 49\,872, 199\,488, and 797\,952 triangular elements).

Figure~\ref{compression_comparison} shows the comparison between the analytical and numerical Lagrange multipliers $\lambda_{\n}$  and tangential displacement jump $\jump{\bu }_{\tang}$, computed on the finest mesh with either one-sided or two-sided bubbles along the fracture (see Remark \ref{rem:twobubbles}). The Lagrange multiplier $\lambda_{\n}$ presents some oscillations in a neighborhood of the fracture tips. As already explained in~\cite{tchelepi-castelletto-2020}, this is due to the sliding of faces close to the fracture tips (in this test case, all fracture faces are in a contact-slip state). The strong singularity of the solution at such points, together with the weak control of the Lagrange multiplier in $H^{-1/2}(\Gamma)$ norm induced by the inf-sup condition, can also explain such oscillatory behaviour of the solution. As could be expected, the two-sided bubble case significantly reduces the Lagrange multiplier oscillations compared with the one-sided bubble case, due to a better stabilisation (but at the cost of more degrees of freedom). In both cases, the discrete tangential displacement jump cannot be distinguished from the analytical solution on this fine mesh. 
Figure \ref{compression_error} and Table \ref{tab:Table_11} display, for the one-sided bubble case, the convergences of the tangential displacement jump and of the normal Lagrange multiplier with respect to the size of the largest fracture face denoted by $h$. Note that the $L^2$ error for the Lagrange multiplier is computed $5\%$ away from each tip to circumvent the lack of convergence induced by the oscillations as in \cite{tchelepi-castelletto-2020}. A first-order convergence for the displacement jump and a $1.5$ convergence order for the Lagrange multiplier are observed. The former (low) rate is related to the low regularity of $\jump{\bu}_\tang$ close to the tips (cf.~the analytical expression~\eqref{sol.compression}); the latter (higher than expected) rate is likely related to the fact that $\lambda_{\n}$ is constant. 
Table \ref{tab:Table_11} also shows the robust convergence  of the semi-smooth Newton algorithm on the family of refined meshes. 
\begin{figure}[H]
	\centering
	
	\subfloat[]{
		\raisebox{.45cm}{
			\includegraphics[keepaspectratio=true,scale=.35]{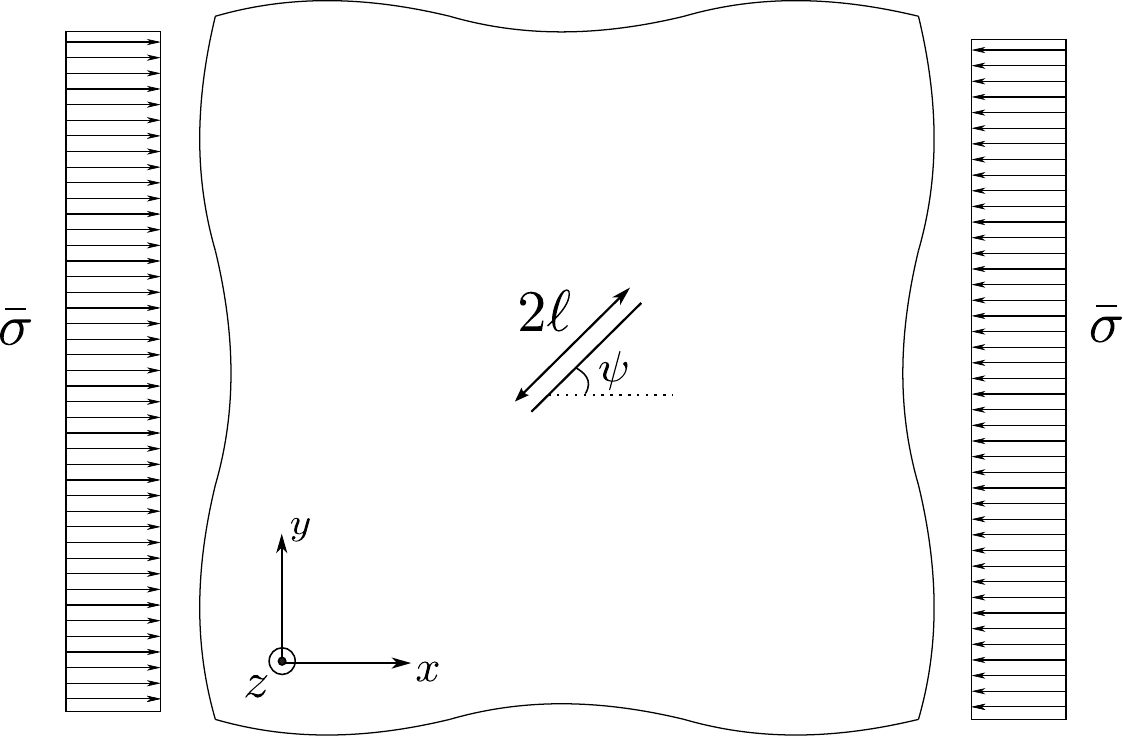}}
	}
	\hspace{2cm}
	\subfloat[]{\includegraphics[keepaspectratio=true,scale=.1]{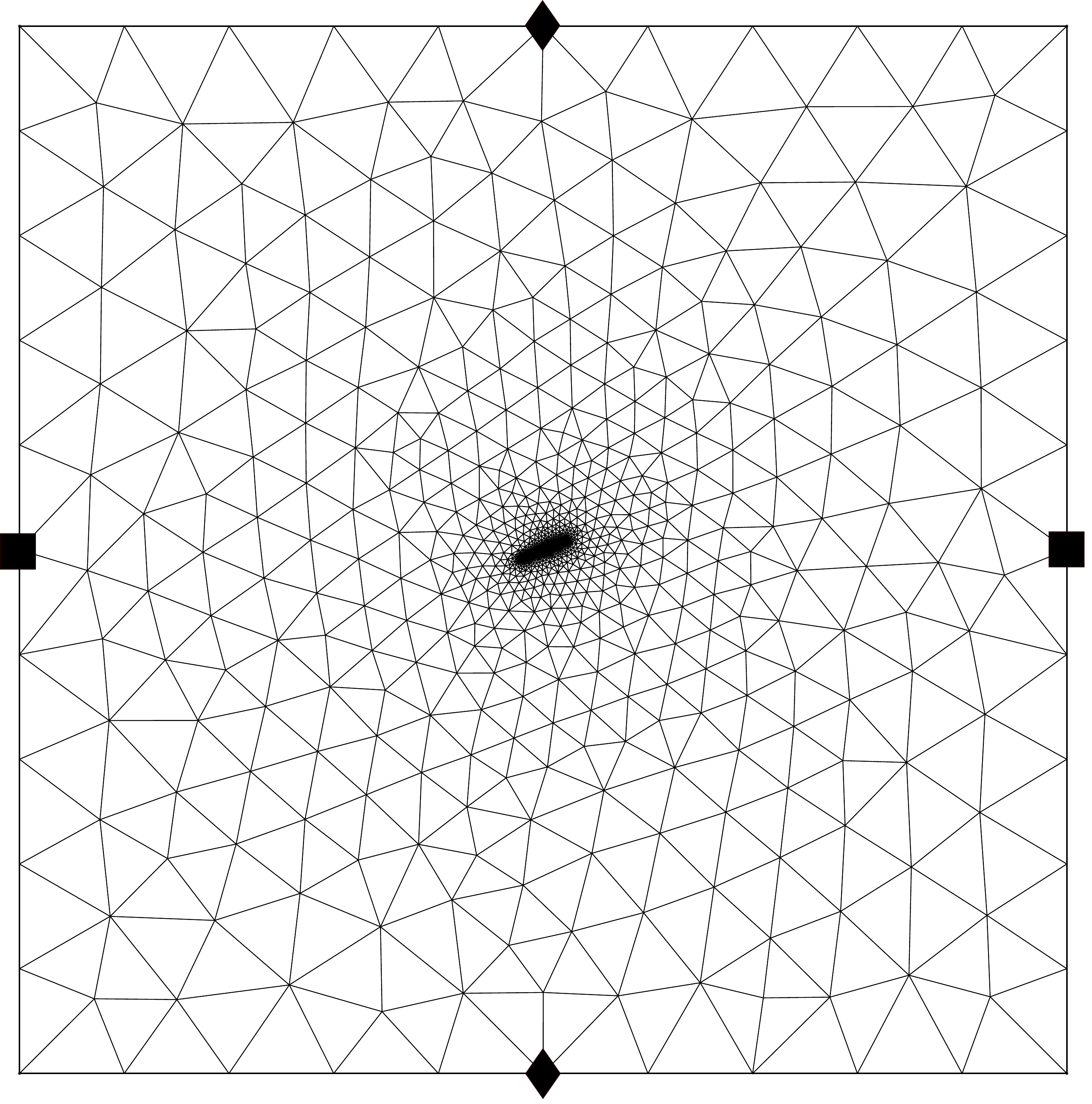}}\\
	\caption{Unbounded domain containing a single fracture under uniform compression (a) and mesh including nodes for boundary conditions ($\blacklozenge$: $u_x = 0$, $\blacksquare$: $u_y=0$) (b), for the example of Section~\ref{tchelepi_meca}.}
	\label{test_compression}
\end{figure}
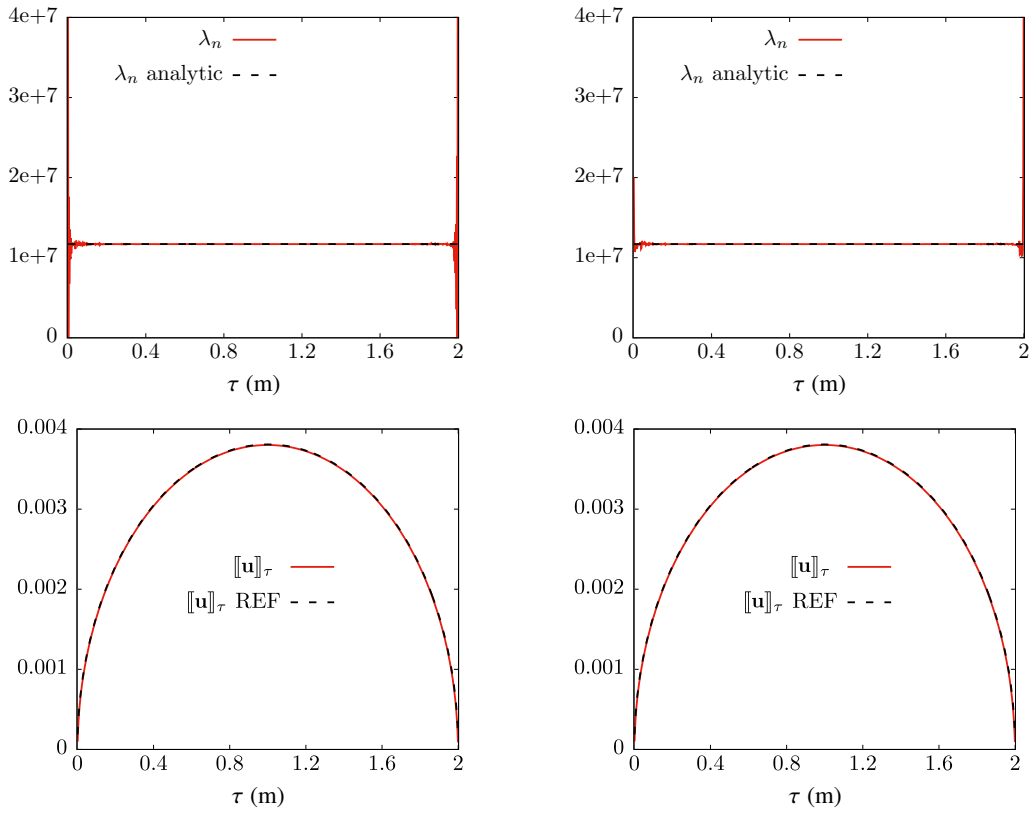
\begin{figure}[H]
	\centering
	\begin{tikzpicture}
		\node (img)  {\includegraphics[scale=0.55]{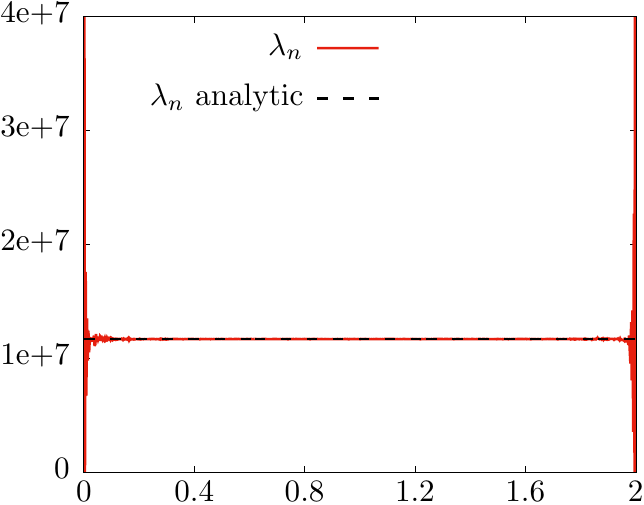}};
		\node[below=of img, node distance=0cm, rotate=0, anchor=center,yshift=0.8cm] {~~~~\footnotesize{ $\tau$ (m)}};
	\end{tikzpicture}
	\hspace{1cm}
	\begin{tikzpicture}
		\node (img)  {\includegraphics[scale=0.55]{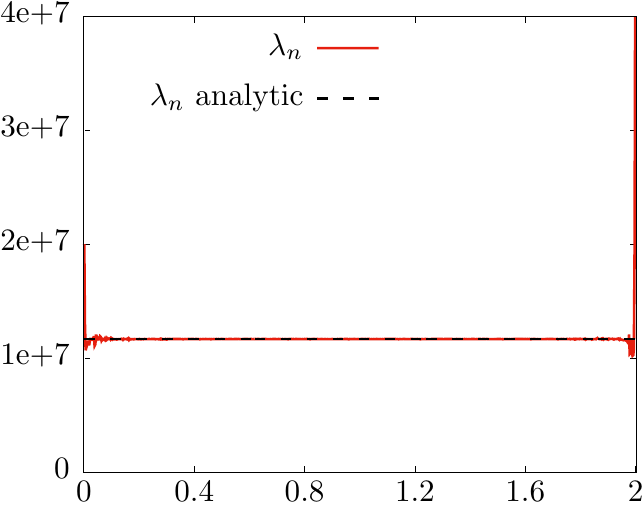}};
		\node[below=of img, node distance=0cm, rotate=0, anchor=center,yshift=0.8cm] {~~~~\footnotesize{ $\tau$ (m)}};
	\end{tikzpicture}

	\begin{tikzpicture}
	\node (img)  {\includegraphics[scale=0.55]{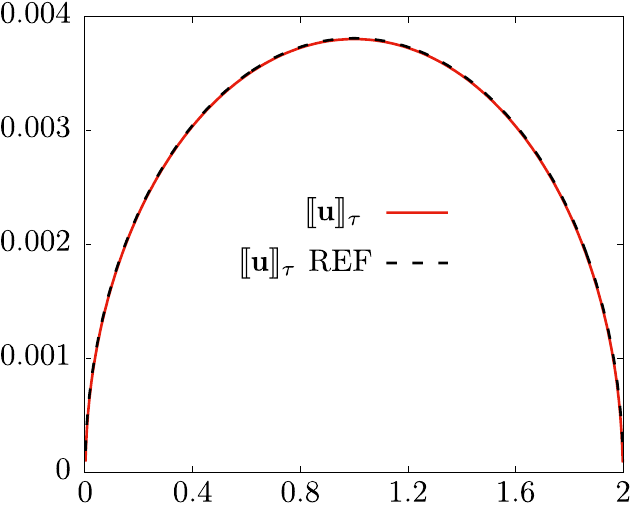}};
	\node[below=of img, node distance=0cm, rotate=0, anchor=center,yshift=0.8cm] {~~~~\footnotesize{ $\tau$ (m)}};
\end{tikzpicture}
\hspace{1cm}
\begin{tikzpicture}
	\node (img)  {\includegraphics[scale=0.55]{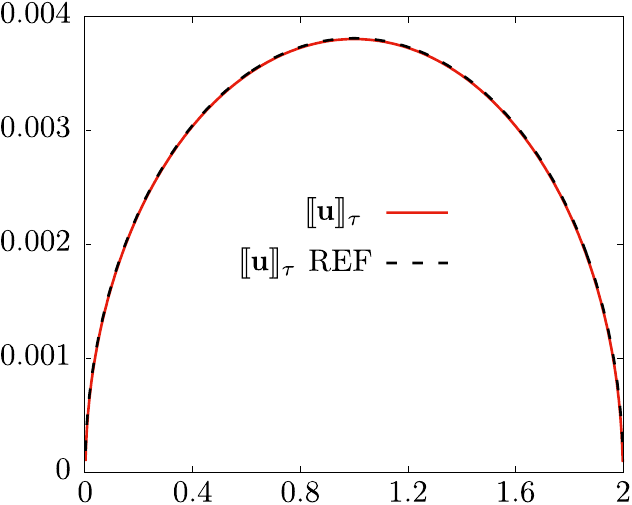}};
	\node[below=of img, node distance=0cm, rotate=0, anchor=center,yshift=0.8cm] {~~~~\footnotesize{ $\tau$ (m)}};
\end{tikzpicture}
	\caption{Comparison between the numerical and analytical solutions on the finest mesh (with 800 fracture faces), in terms of $\lambda_{\n}$ and $\jump{\bu}_\tang$ with one-sided bubbles  (left) and two-sided bubbles  (right), example of Section~\ref{tchelepi_meca}.}
	\label{compression_comparison}
\end{figure}
\begin{figure}[H]
	\centering
	\begin{tikzpicture}
	\node (img)  {\includegraphics[scale=0.6]{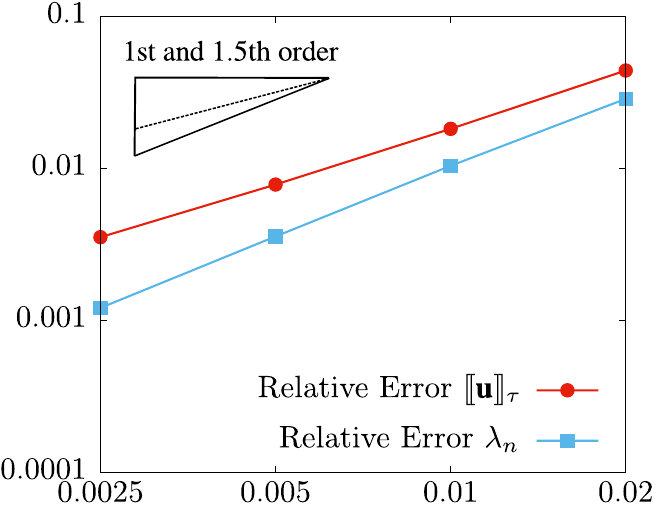}};
	\node[below=of img, node distance=0cm, rotate=0, anchor=center,yshift=0.8cm] {~~~~\footnotesize{ $h$ (m)}};
	\node[left=of img, node distance=0cm, rotate=90, anchor=center,yshift=-0.7cm] {\footnotesize{Relative $L^2$ Error}};
    \end{tikzpicture}
	\caption{Relative $L^2$ norms of the errors $\jump{\bu}_\tang - \jump{\bu_\D}_{\D,\tang}$ and $\lambda_\n - \lambda_{\D,\n}$ away from the tip, with respect to the size of the largest fracture face denoted by $h$, one-sided bubble case. Test case of Section~\ref{tchelepi_meca}.}
	\label{compression_error}
\end{figure}

\begin{table}[!ht]
	\begin{center}
		\begin{tabular}{ |c|c|c|c|c|c|c|c|c| } 
			\hline 
			$\#\faces_\Gamma$ & $N^{\text{dof}}$ & $E_{L^2}^{\jump{\bu }_{\tang}}$ & order $\jump{\bu }_{\tang}$& $E_{L^2}^{\lambda_{\n}}$ & order $\lambda_{\n}$  &  $N^{\text{Newton}}$\\
			\hline \hline
			100  &13028 & 4.36E-2  & -  & 2.23E-2  & -  &  2   \\
			\hline
			200 & 50992  & 1.80E-2 &  1.27 & 8.84E-3  &  1.34 &   2\\ 
			\hline
			400 & 201728  & 7.71E-3  &  1.25 & 2.91E-3 & 1.60 &  2  \\ 
			\hline
			800 & 802432	& 3.46E-3  &  1.15 & 9.89E-4 & 1.56  &   2 \\ 
			\hline
		\end{tabular} 
	\end{center}
	\caption{Relative $L^2$ errors and convergence orders for $\jump{\bu}_\tang - \jump{\bu_\D}_{\D,\tang}$ and $\lambda_\n - \lambda_{\D,\n}$ away from the tip, and number $N^{\text{Newton}}$ of semi-smooth Newton iterations for the different meshes with $N^{\text{dof}}$ scalar degrees of freedom (3D mesh) and $\#\faces_\Gamma$ fracture faces. One-sided bubble case. Test case of Section~\ref{tchelepi_meca}.}
	\label{tab:Table_11}
\end{table}

\subsubsection{2D Discrete Fracture Matrix model with 6 fractures: static contact-mechanics test case}
\label{bergen_meca}

To illustrate the behaviour of our scheme on a more complex fracture network, we consider the Discrete Fracture Matrix (DFM) model test case presented in~\cite[Section~4.1]{contact-norvegiens}, where a $2 \text{m} \times 1 \text{m} \times 1 \text{m}$ domain including a network $\Gamma = \bigcup_{i=1}^{6} \Gamma_i$  of fractures is considered, see Figure \ref{6_fractures}. Fracture 1 is made up of two sub-fractures forming a corner, whereas one of the tips of Fracture 5 lies on the boundary of the domain.
We use the same values of Young's modulus and Poisson's ratio, $E=4$ GPa and $\nu=0.2$, and the same set of boundary conditions as in \cite{contact-norvegiens}, that is, the two left and right sides of the domain are free, and we impose $\bu = 0$ on the lower side and $\bu= \prescript{t}{}[0.005\,\text{m}, -0.002\,\text{m}]$ on the top side. The friction coefficient is $F_i(\x) = 0.5(1 + 10  e^{-D_i^2(\x)/0.005})$, with $i \in  \{1,...,6 \} $ the fracture index, $\x \in \Gamma_i$ a generic point on fracture $i$, and $D_i(\x)$ the minimum distance from $\x$ to the tips of fracture $i$ (the bend in Fracture 1 is not considered a tip). 

Since no closed-form solution is available for this test case, the numerical convergence is evaluated with respect to a reference solution computed on a fine mesh made of 730\,880 triangular elements. Figure~\ref{conv_6_frac_meca} shows the convergence rates obtained for both $\jump{\bu}$ and $\bm\lambda$. The family of triangular meshes is obtained by successive uniform refinements of a given initial coarse mesh. As in~\cite{contact-norvegiens}, an asymptotic first-order convergence is observed for the vector Lagrange multiplier for all fractures, except Fracture 4 which exhibits a convergence rate close to 2 owing to its entire contact-stick state, and Fracture 1 which exhibits a lower rate due to the additional singularity induced by the corner. For the jump of the displacement field across fractures, we obtain an asymptotic convergence rate equal to 1.5 for all fractures. In Figure \ref{compar_mix_nit}, we compare the error curves of  the displacement jump $\jump{\bu }$ and the traction mean value $ \left(\T^{\scriptstyle{+}} - \T^{-} \right)/2$ between our method and the Nitsche $\Po^1$ FEM for contact-mechanics \cite{beaude2023mixed}, for Fractures 1, 2, and 3. It is noticeable that we have approximately the same convergence for both methods.
Taking advantage of the flexibility of the polytopal VEM method, we then consider a modified mesh obtained by inserting a node at the midpoint of each fracture-edge, generating 4-node triangles on both sides of the fractures. Figure \ref{tri_3_4_nodes} exhibits the tangential displacement jumps obtained with the original coarse mesh, and the one refined only along the fractures with these 4-node triangles. The numerical solution better captures the stick-slip transition on the refined mesh than on the original mesh. This improvement is achieved with a minor computational overcost, as the volumetric discretisation remains unchanged, and we exclusively refine along the fractures.
\begin{figure}[H]
	\centering
	
	\begin{minipage}{0.45\textwidth}
		\begin{tikzpicture}
			\node (img)  {\includegraphics[scale=0.2]{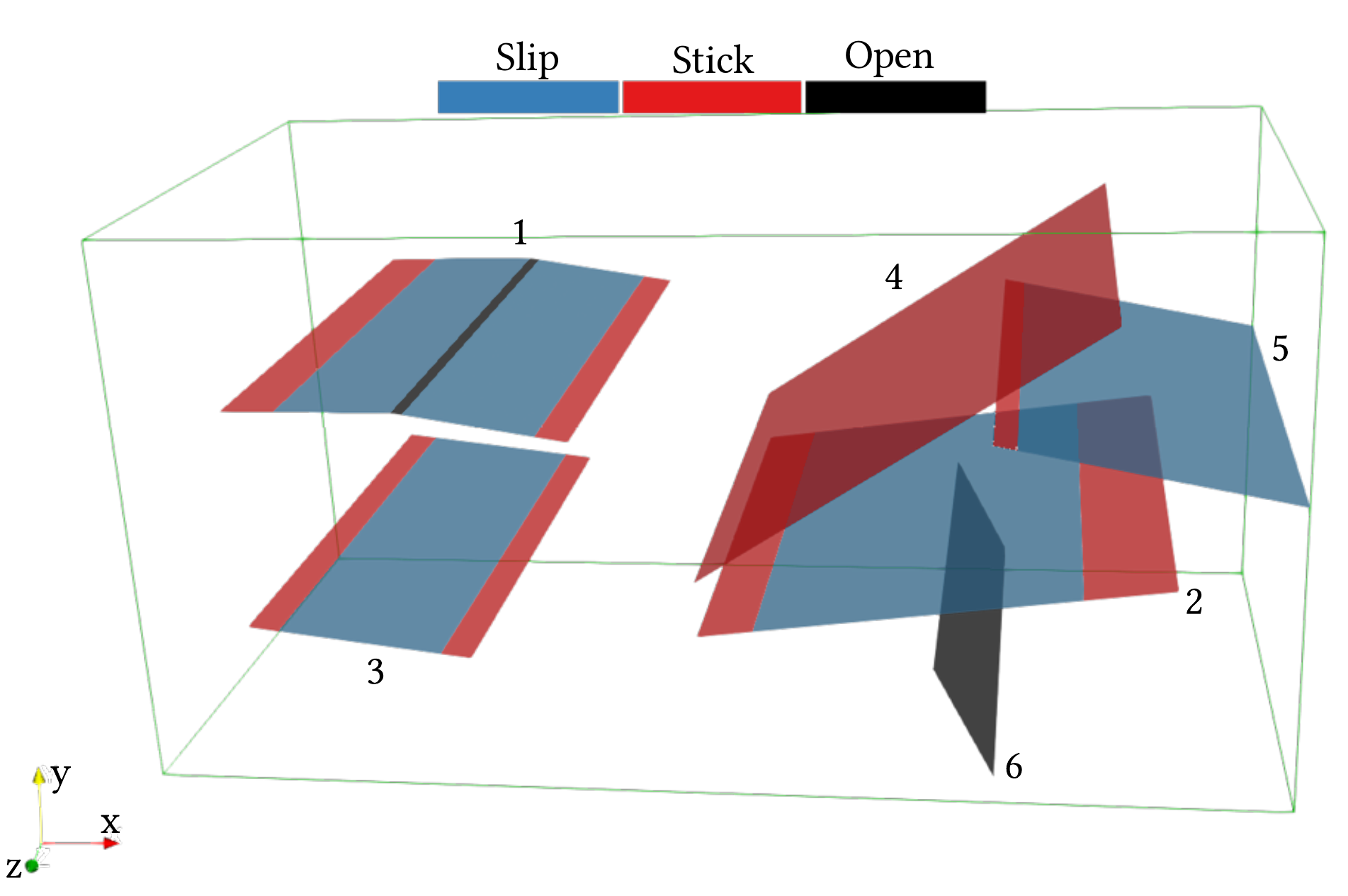}};
		\end{tikzpicture}
	\end{minipage}%
	\begin{minipage}{0.45\textwidth}
		\begin{tikzpicture}
			\node (img)  {\includegraphics[scale=0.14]{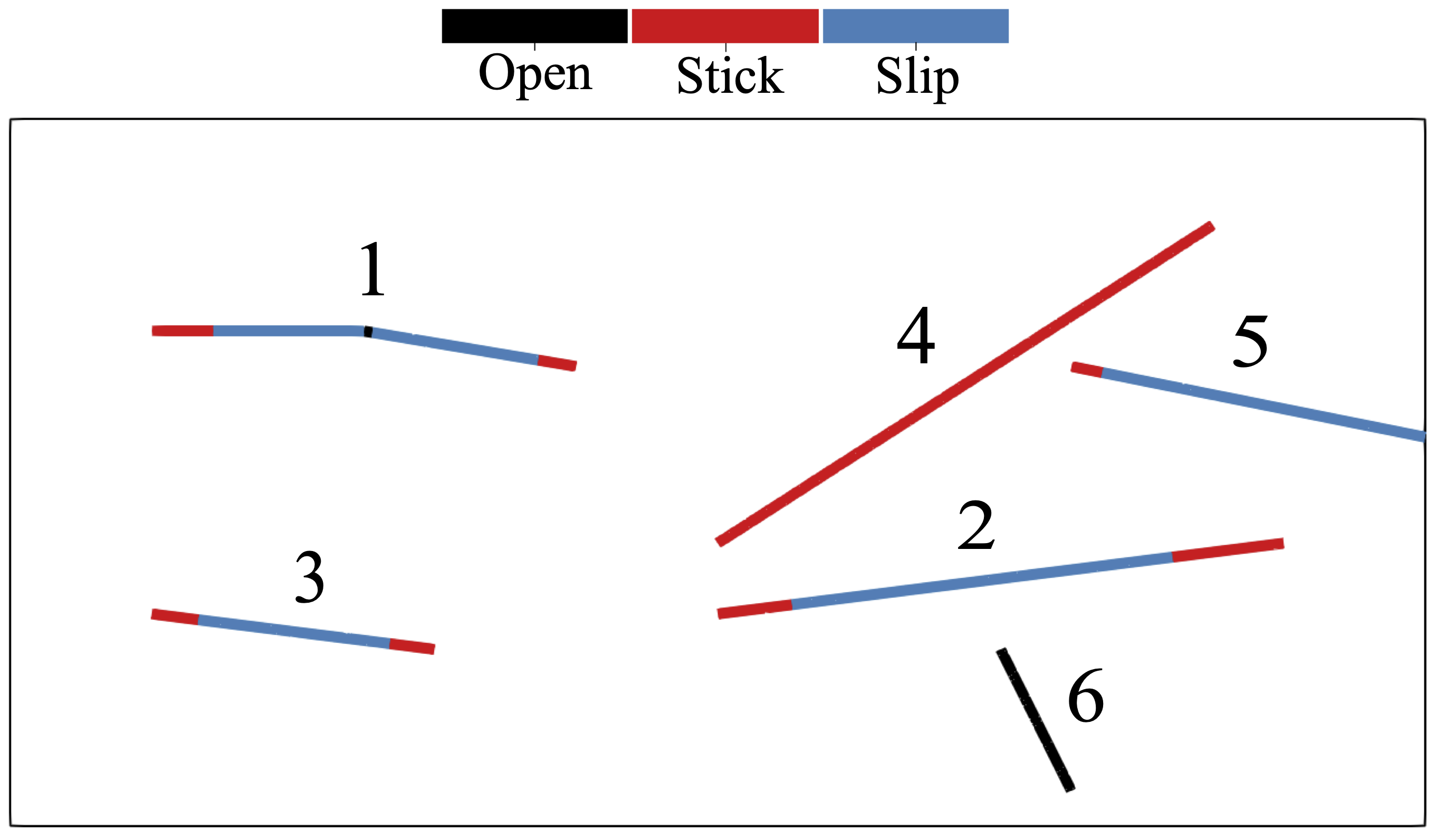}};
		\end{tikzpicture}
	\end{minipage}%
	\caption{3D prism / 2D Rectangular domain with six fractures. Fracture 1 comprises two sub-fractures making a corner, and Fracture 5 has a tip on the boundary. The contact state of each fracture obtained by the simulation is also shown on the same mesh: Mixed $\Po^1$-bubble VEM--$\Po^0$ (left) vs.~Nitsche $\Po^1$ FEM (right). Test case of Section \ref{bergen_meca}.}\label{6_fractures}	
\end{figure}
\begin{figure}[H]
	\centering
	\begin{minipage}{0.45\textwidth}
		\hspace{-1cm}
		\begin{tikzpicture}
			\node (img)  {\includegraphics[scale=0.3]{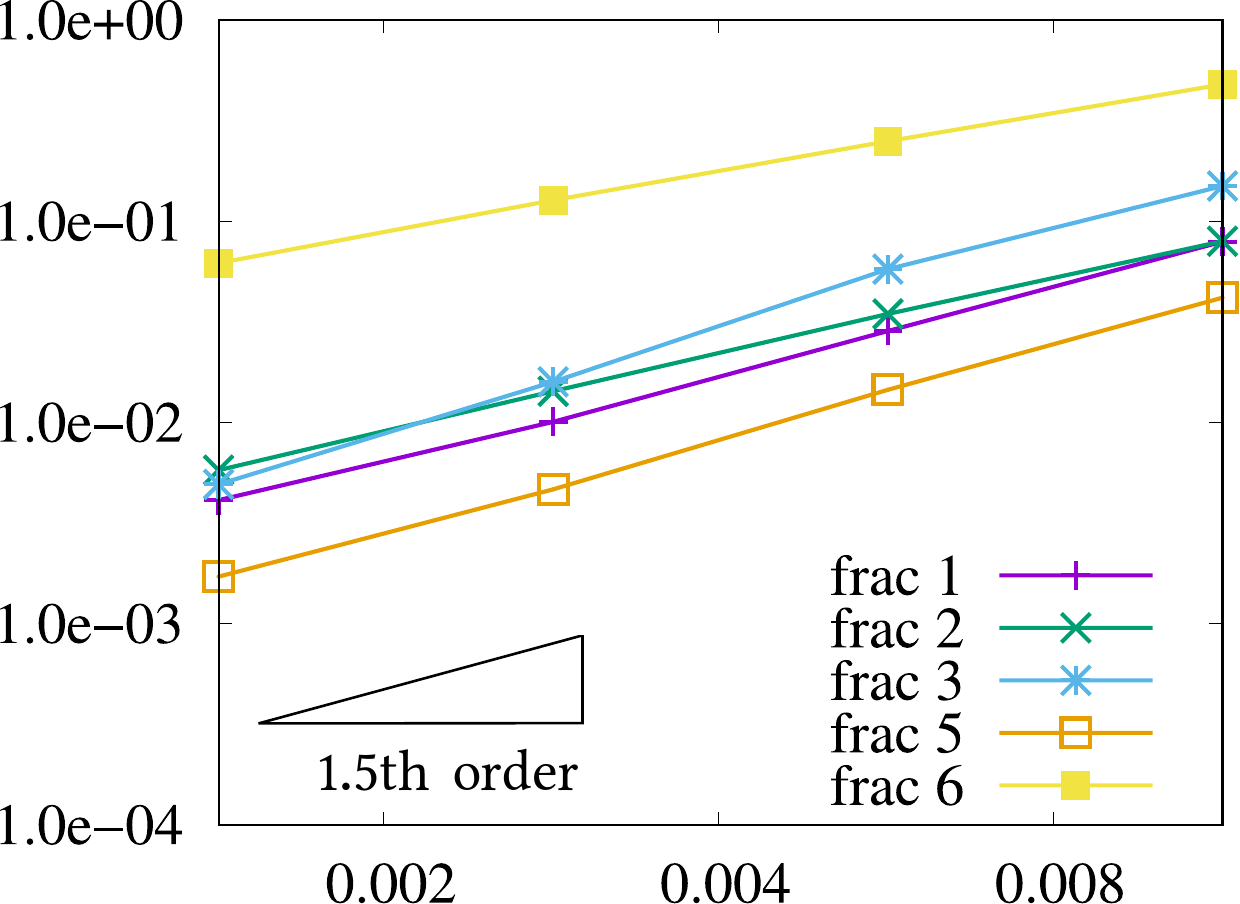}};
			\node[left=of img, node distance=0cm, rotate=90, anchor=center,yshift=-0.7cm] {\footnotesize{  $ \jump{\bu }$}};
		\end{tikzpicture}
	\end{minipage}%
	\begin{minipage}{0.45\textwidth}	 
		\hspace{-0.8cm}
		\begin{tikzpicture}
			\node (img)  {\includegraphics[scale=0.3]{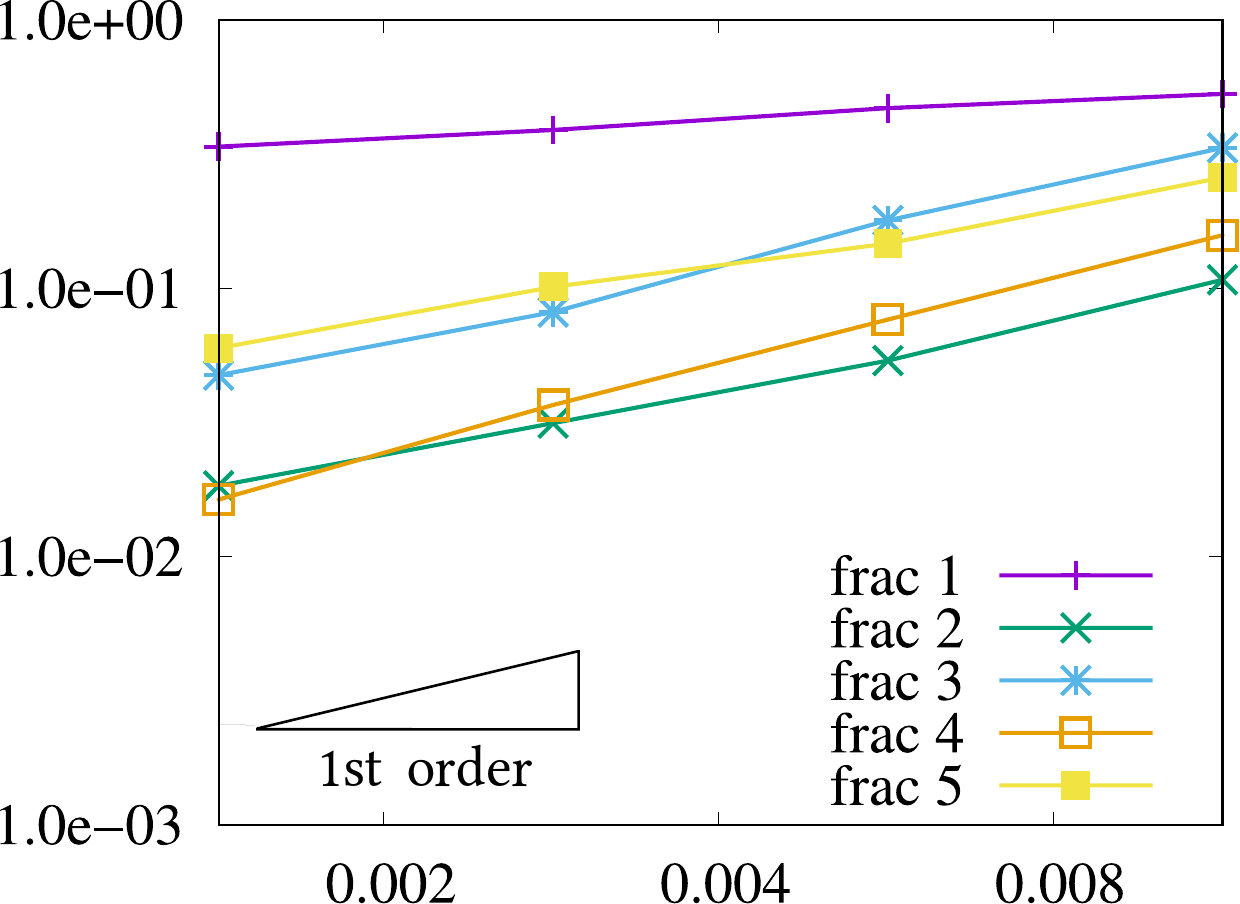}};
			\node[left=of img, node distance=0cm, rotate=90, anchor=center,yshift=-0.7cm] {\footnotesize{  $\l$}};
		\end{tikzpicture}
	\end{minipage}%
	\caption{Relative $L^2$ errors (using a reference solution) for $\jump{\bu }$ and $\l$, as functions of the size of the largest fracture face, yielding a 1.5-order of convergence for $\jump{\bu }$  and 1st order of convergence for $\l$. Test case of Section \ref{bergen_meca}.}\label{conv_6_frac_meca}	
\end{figure}

	\begin{figure}[H]
		\centering
		\begin{minipage}{0.3\textwidth}
			\hspace{-0.6cm}
			\begin{tikzpicture}
				\node (img)  {\includegraphics[scale=0.2]{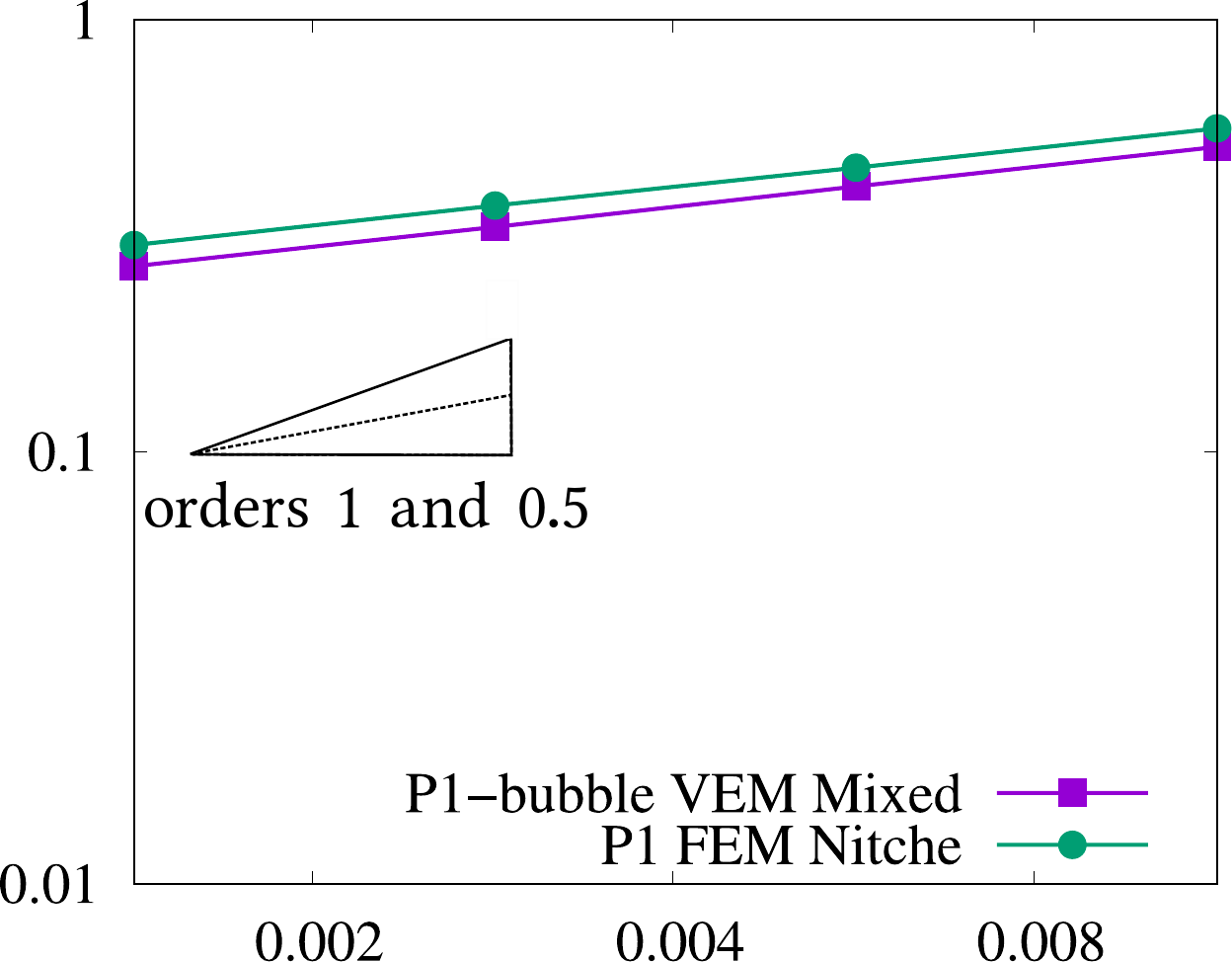}};
				\node[left=of img, node distance=0cm, rotate=90, anchor=center,yshift=-0.7cm] {\footnotesize{  $ \frac{\T^{\scriptstyle{+}} - \T^{-} }{2}$}};
			\end{tikzpicture}
		\end{minipage}%
	    \hspace{0.3cm}
		\begin{minipage}{0.3\textwidth}
			\begin{tikzpicture}
				\node (img)  {\includegraphics[scale=0.2]{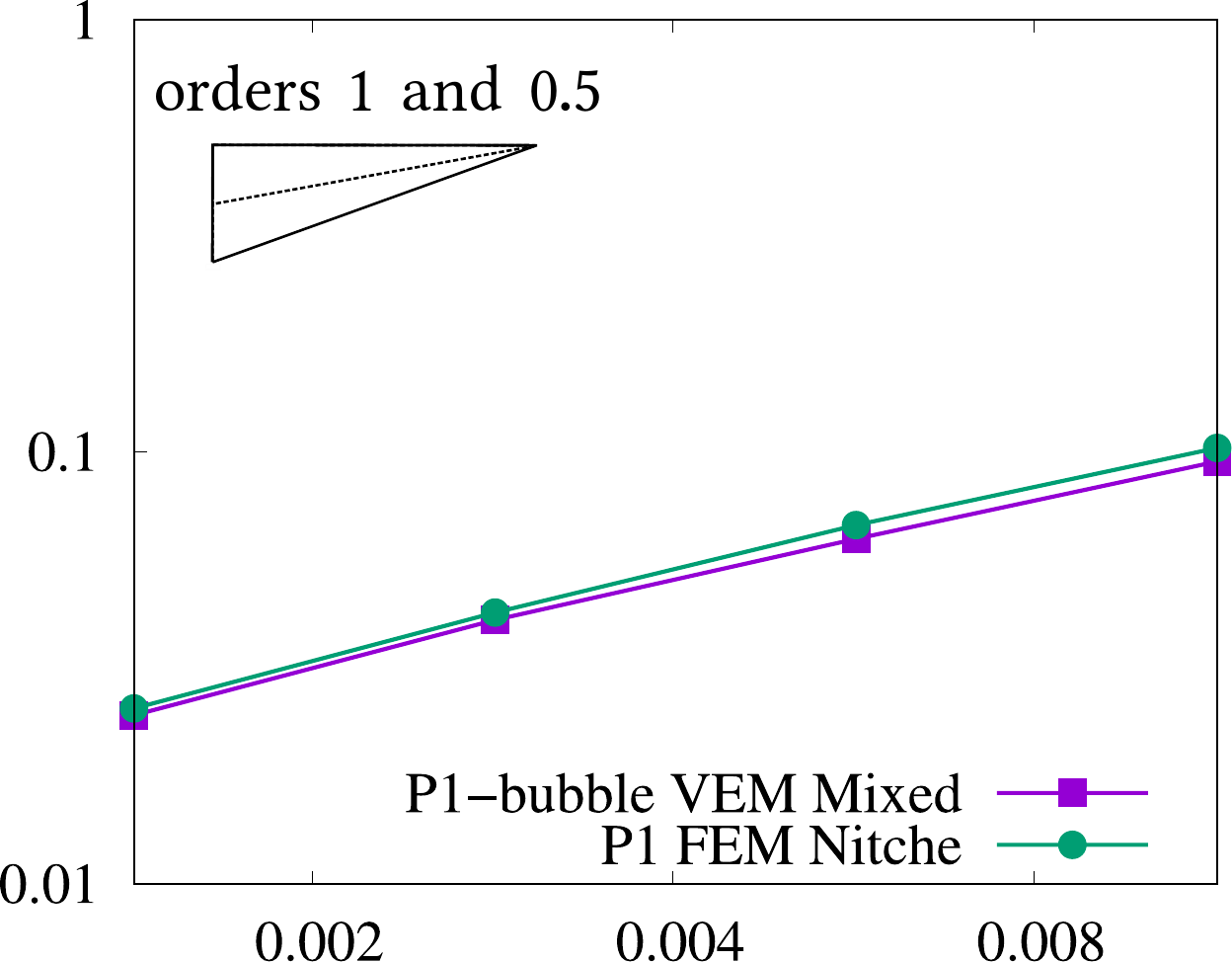}};
			\end{tikzpicture}
		\end{minipage}%
		\hspace{0.35cm}
		\begin{minipage}{0.3\textwidth}
			\hspace{0.15cm}
			\begin{tikzpicture}
				\node (img)  {\includegraphics[scale=0.2]{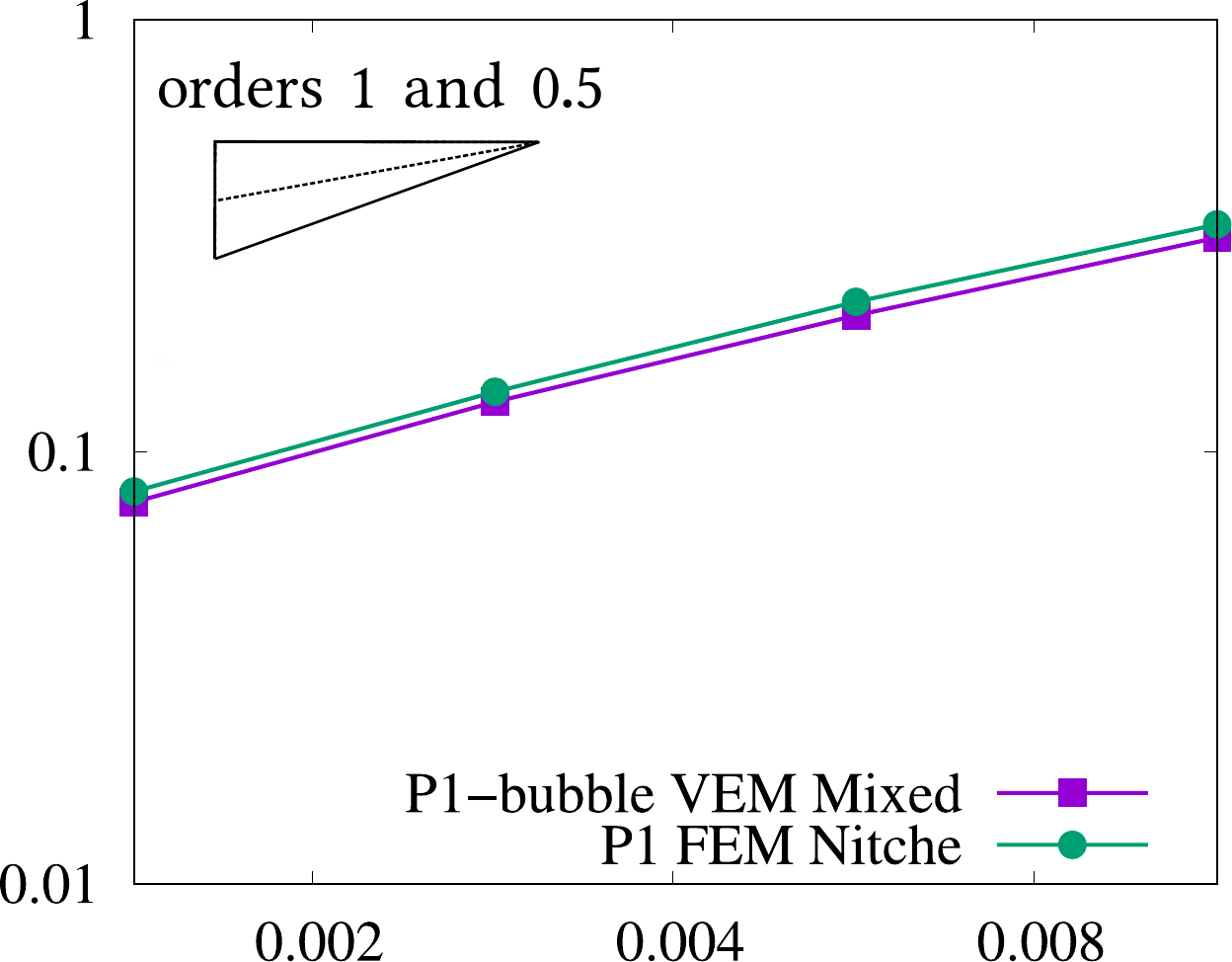}};
			\end{tikzpicture}
		\end{minipage}%
		
		\hspace{-0.4cm}
		\begin{minipage}{0.3\textwidth}
			\hspace{-1cm}
			\begin{tikzpicture}
				\node (img)  {\includegraphics[scale=0.22]{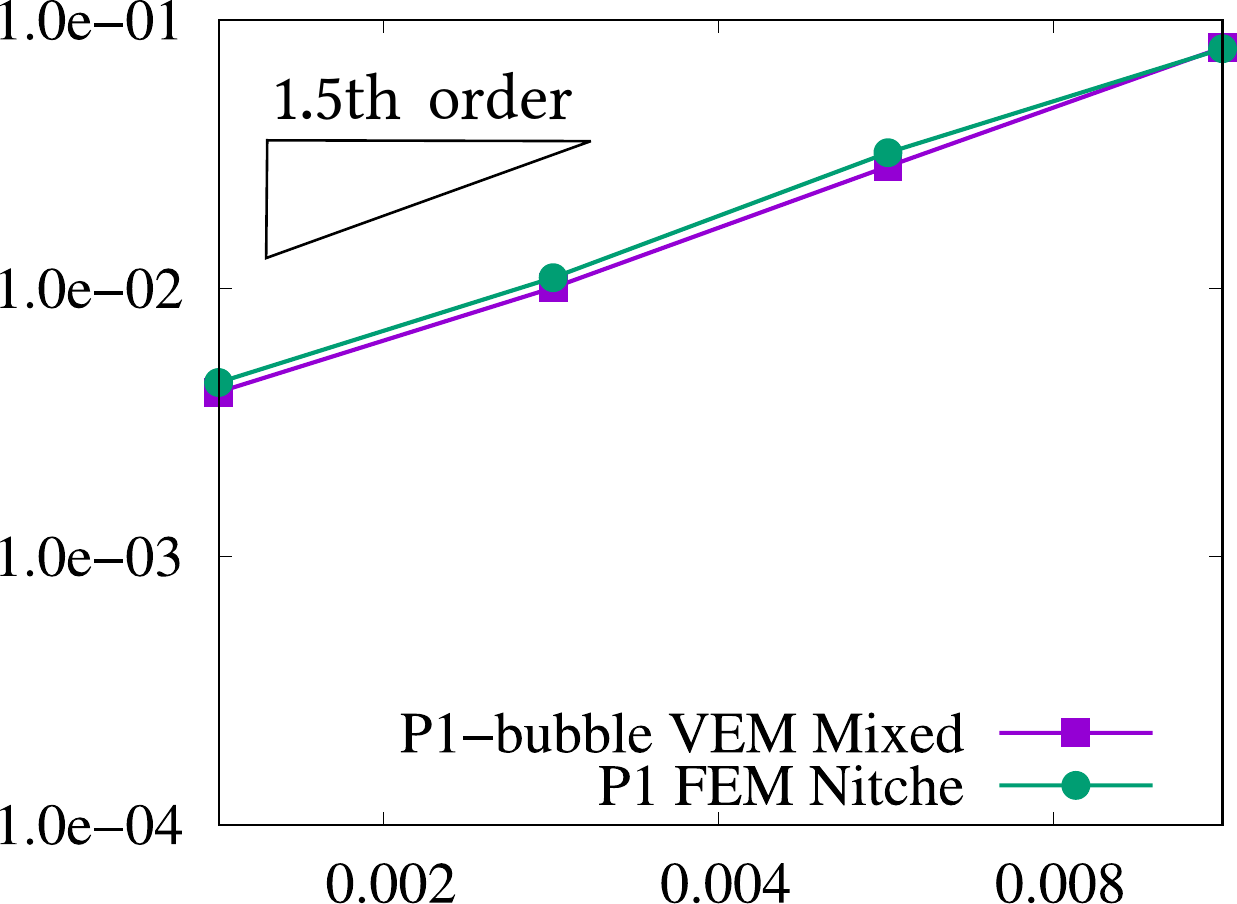}};
				\node[left=of img, node distance=0cm, rotate=90, anchor=center,yshift=-0.7cm] {\footnotesize{  $ \jump{\bu }$}};
			\end{tikzpicture}
		\end{minipage}%
		\begin{minipage}{0.3\textwidth}
			\begin{tikzpicture}
				\node (img)  {\includegraphics[scale=0.22]{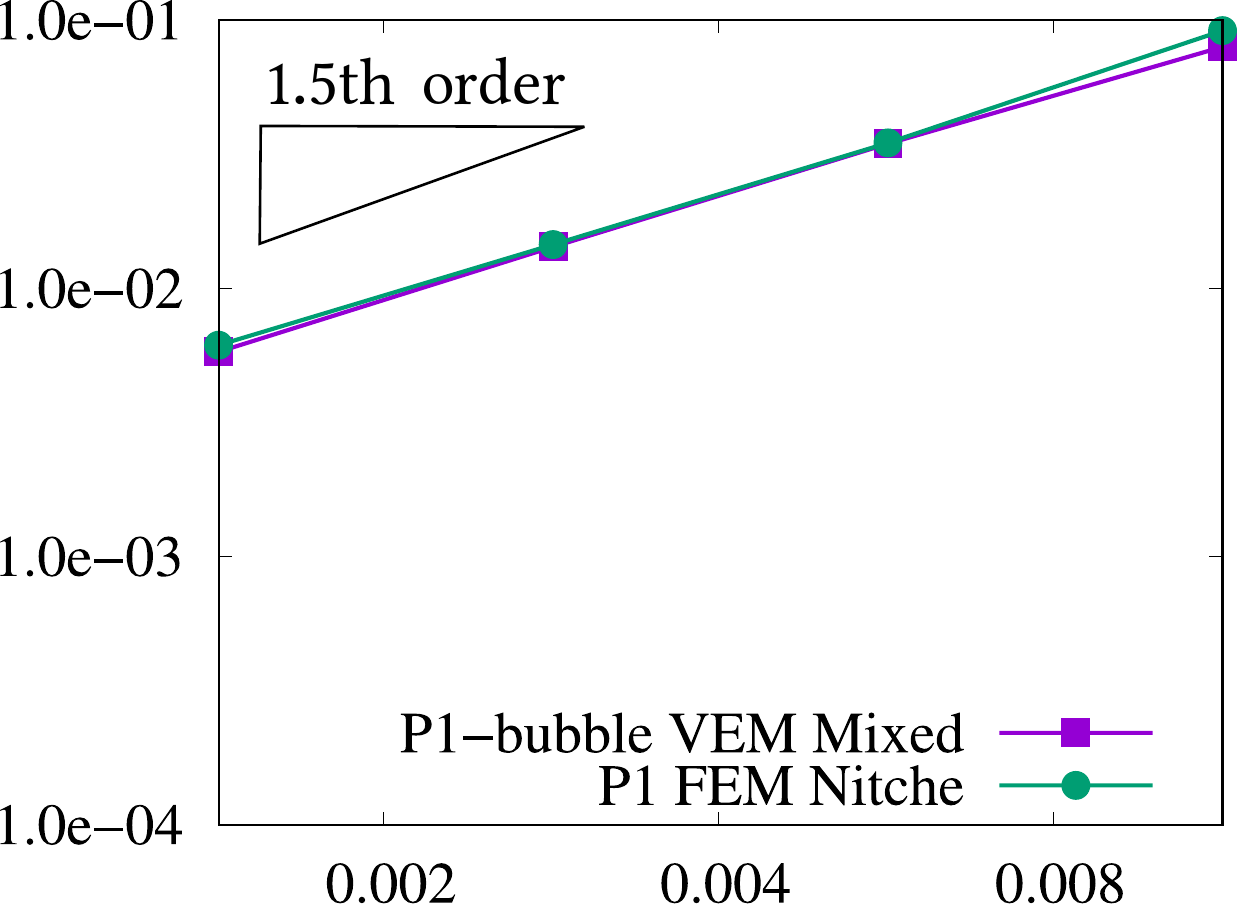}};
			\end{tikzpicture}
		\end{minipage}%
		\hspace{0.3cm}
		\begin{minipage}{0.3\textwidth}
			\hspace{0.5cm}
			\begin{tikzpicture}
				\node (img)  {\includegraphics[scale=0.22]{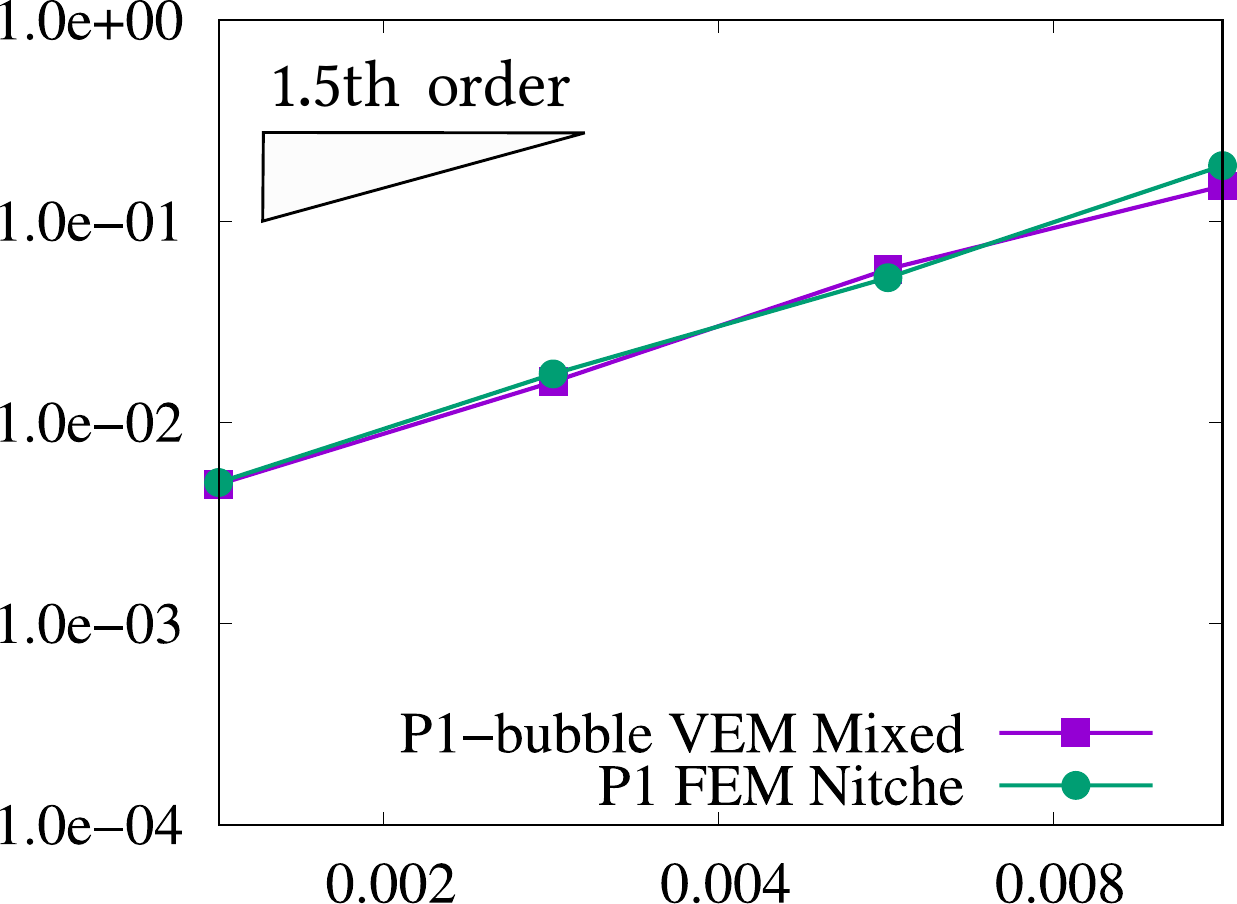}};
			\end{tikzpicture}
		\end{minipage}%

		\caption{Comparison of the Mixed $\Po^1$-bubble VEM--$\Po^0$ vs.~Nitsche $\Po^1$ FEM: relative $L^2$ errors on $ \left(\T^{\scriptstyle{+}} - \T^{-} \right)/2$ (top) and $\jump{\bu }$ (bottom), as functions of the size of the largest fracture face. These errors are computed along Fractures 1, 2 and 3 (left to right), using a reference solution. Test case of Section \ref{bergen_meca}.}\label{compar_mix_nit}
	\end{figure}

	\begin{figure}[H]
	\centering
	\begin{minipage}{0.45\textwidth}
		\hspace{-1cm}
		\begin{tikzpicture}
			\node (img)  {\includegraphics[scale=0.6]{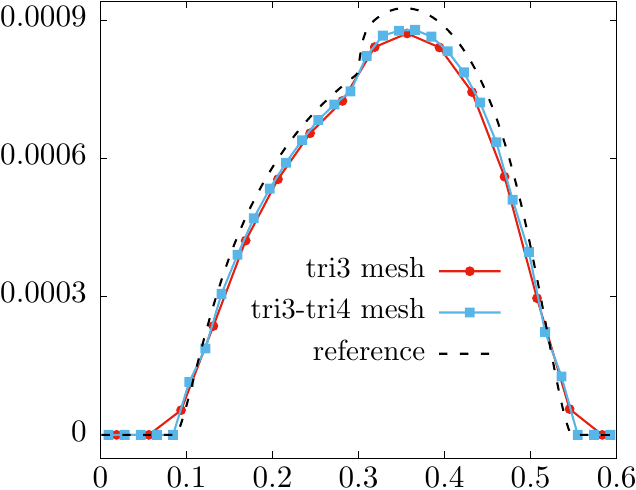}};
			\node[left=of img, node distance=0cm, rotate=90, anchor=center,yshift=-0.7cm] {\footnotesize{  $ \jump{\bu }_{\tang}$}};
		\end{tikzpicture}
	\end{minipage}%
	\begin{minipage}{0.45\textwidth}
		\hspace{-0.5cm}
		\begin{tikzpicture}
			\node (img)  {\includegraphics[scale=0.6]{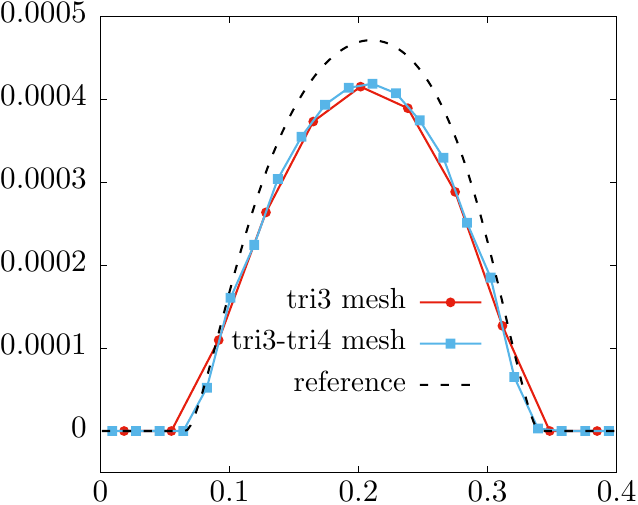}};
		\end{tikzpicture}
	\end{minipage}%
	\caption{Tangential jumps $\jump{\bu}_{\tang}$ on Fracture 1 (left) and Fracture 3 (right) obtained  on the original coarse triangular mesh ("tri3 mesh''), and the mesh obtained by refining, only along the fracture faces, using 4-node triangles ("tri3-tri4 mesh"). Test case of Section \ref{bergen_meca}.}\label{tri_3_4_nodes}	
\end{figure}

\subsection{Poromechanical test cases} \label{test_poromeca}

The objective here is first to compare our discretisation with the one presented in \cite{beaude2023mixed} combining a Nitsche $\Po^2$ FEM for the contact-mechanics with the HFV scheme for the flow on triangular 2D meshes. The second objective is to assess the robustness of our approach on a 3D test case with a fracture network including intersections. 
The coupled nonlinear system is solved at each time step of the simulation using the fixed-stress algorithm \cite{KTJ11} adapted to mixed-dimensional models following \cite{GKW16}. This algorithm comprises two nested loops: an outer one on the time steps (index $n$), and an inner one on the fixed-stress steps (index $k$). For each time step, the inner loop is:
\begin{itemize}
\item[$\bullet$] \emph{Initialization ($k=0$)}. Set an guess for the displacement and pressure by interpolating from the previous two time steps:
   $$
    \bu_\D^{n,0} = \bu_\D^{n-1} + \Delta t^n {\bu_\D^{n-1} - \bu_\D^{n-2} \over \Delta t^{n-1}},\qquad 
    p_\D^{n,0} = p_\D^{n-1} + \Delta t^n {p_\D^{n-1} - p_\D^{n-2} \over \Delta t^{n-1}}. 
   $$ 
\item[$\bullet$] \emph{Step from $k-1$ to $k$}. While ${\|\bu_\D^{n,k} - \bu_\D^{n,k-1}\|_\infty \over u_{\rm ref} } + {\|p_\D^{n,k} - p_\D^{n,k-1}\|_\infty \over p_{\rm ref} } \geq \epsilon_{\rm fs}$ do 
 \begin{itemize}
 \item[$\bullet$] Given $\bu_\D^{n,k-1}$, compute $p_\D^{n,k}$ solution of the linear Darcy flow system, using the following porosity and fracture width:
   $$
   \begin{aligned}
     & \phi_\D^{n,k} = \phi_\D^{n-1} + b\div_\D (\bu_\D^{n,k-1}-\bu_\D^{n-1})   + {1\over M} \Pi_{\D_m} (p_{\D_m}^{n,k}-p_{\D_m}^{n-1}) \\ & \qquad\qquad\qquad \qquad\qquad\qquad\qquad + C_{r,m} 
      \Pi_{\D_m} (p_{\D_m}^{n,k}-p_{\D_m}^{n,k-1}),\\
     & \d_{f,\D}^{n,k} =   \d_{f,\D}^{n-1} - \jump{\bu_{\D}^{n,k-1} - \bu_{\D}^{n-1} }_{\D,\n} + C_{r,f}
     \Pi_{\D_m} (p_{\D_f}^{n,k}-p_{\D_f}^{n,k-1}). 
   \end{aligned}
   $$
   \item[$\bullet$] Given $p_\D^{n,k}$, compute $\bu_\D^{n,k}$ solution of the contact-mechanical system using the semi-smooth Newton algorithm.
 \end{itemize}  
\end{itemize}
Note that the fixed-stress iterative algorithm is initialised in such a way that the first iterate corresponds to the sequential fixed-stress solution \cite{BM2000}. 
At each subsequent iteration, the Darcy linear problem is solved using a GMRes iterative solver preconditioned by AMG with stopping criteria $10^{-8}$. The contact-mechanical model is then solved using the semi-smooth Newton method combined with a direct sparse linear solver; the stopping criteria is $10^{-10}$ on the relative residual, or on the maximum displacement increment (whichever comes first).  The fixed-stress stopping criteria is fixed to $\epsilon_{\rm fs} = 10^{-5}$ with $u_{\rm ref} = 10^{-3}$ m, $p_{\rm ref} = 10^5$ Pa, and the relaxation parameters  are given by $C_{r,m} = {3 b^2 \over 2 \mu + 3 \lambda}$ and $C_{r,f} = 0$.

\subsubsection{2D DFM with 6 fractures: poromechanical test case}\label{bergen_coup}

This test case presented in \cite[Section 5]{beaude2023mixed}  adds the fluid flow to the contact-mechanical test case of Section \ref{bergen_meca}. On the mechanical side, the only changes are related to the friction coefficient fixed here to $F=0.5$ and to the following time dependent Dirichlet boundary conditions on the top boundary:
\begin{equation*}
	\bu(t,\x)= 
	\left\{
	\begin{array}{l@{\quad}l}
		\dsp \prescript{t}{}[0.005\,\text{m}, -0.002\,\text{m}] \frac{4t}{T}&\text{if}~t\le\frac{T}{4},\\[1ex]
		\prescript{t}{} [0.005\,\text{m}, -0.002\,\text{m}] & \text{otherwise}.
	\end{array}
	\right.
\end{equation*} 
Concerning the flow boundary conditions, all sides are impervious except the left one, on which a pressure equal to the initial pressure $10^5$ Pa is prescribed.
To fully exploit the capabilities of the HFV flow discretisation, we consider the following anisotropic permeability tensor in the matrix:
$$
  \K_m=K_m\left(\begin{array}{cc}1 & 0 \\ 0 & 1/2\end{array}\right).
$$
The permeability coefficient is set to $K_m = 10^{-15}$~m$^2$, the Biot coefficient to $b=0.5$, the Biot modulus to $M=10$~GPa, and the dynamic viscosity to $\eta = 10^{-3}$~Pa${\cdot}$s.
For further details related to the Darcy flow, we refer the reader to \cite[Section 5]{beaude2023mixed}. The triangular meshes are the same as in Section \ref{bergen_meca}, the final time is set to $T=2000\,\text{s}$, and a uniform time stepping with 20 time steps is used for the Euler implicit time integration (the same time integration is used for the Nitsche $\Po^2$ FEM used for comparison).
 
Figure \ref{pm_mean} exhibits the good convergence behaviour of the mean matrix pressure as a function of time, on a family of three uniformly refined meshes and using a numerical reference solution computed on a finer mesh through the Nitsche $\Po^2$ FEM presented in \cite{beaude2023mixed}. 
The matrix over-pressure (w.r.t.~the initial pressure) obtained at final time is also compared in Figure \ref{pm_mix_vs_nitsh} to the one obtained by the Nitsche $\Po^2$ FEM of \cite{beaude2023mixed}. Notably, we found that the numerical results of the two methods are very similar. 
Figure \ref{df_pf} shows the mean fracture aperture and pressure as functions of time for the family of three uniformly refined meshes, showing a good spatial convergence to the numerical reference solution. Note that the fracture pressure is extremely close to the trace of the matrix pressure due to the high conductivity of the fractures. 
Figure \ref{error_m_f_var} exhibits a 1.5-order convergence rate of the discrete $l^2$ errors in time of the matrix mean pressure, the porosity, the fracture mean pressure, the aperture, and the tangential jump. The reference solution is computed using the Nitsche $\Po^2$ FEM on a fine mesh. 

\begin{figure}[H]
    \centering
	 \begin{tikzpicture}
    	\node (img)  {\includegraphics[scale=0.35]{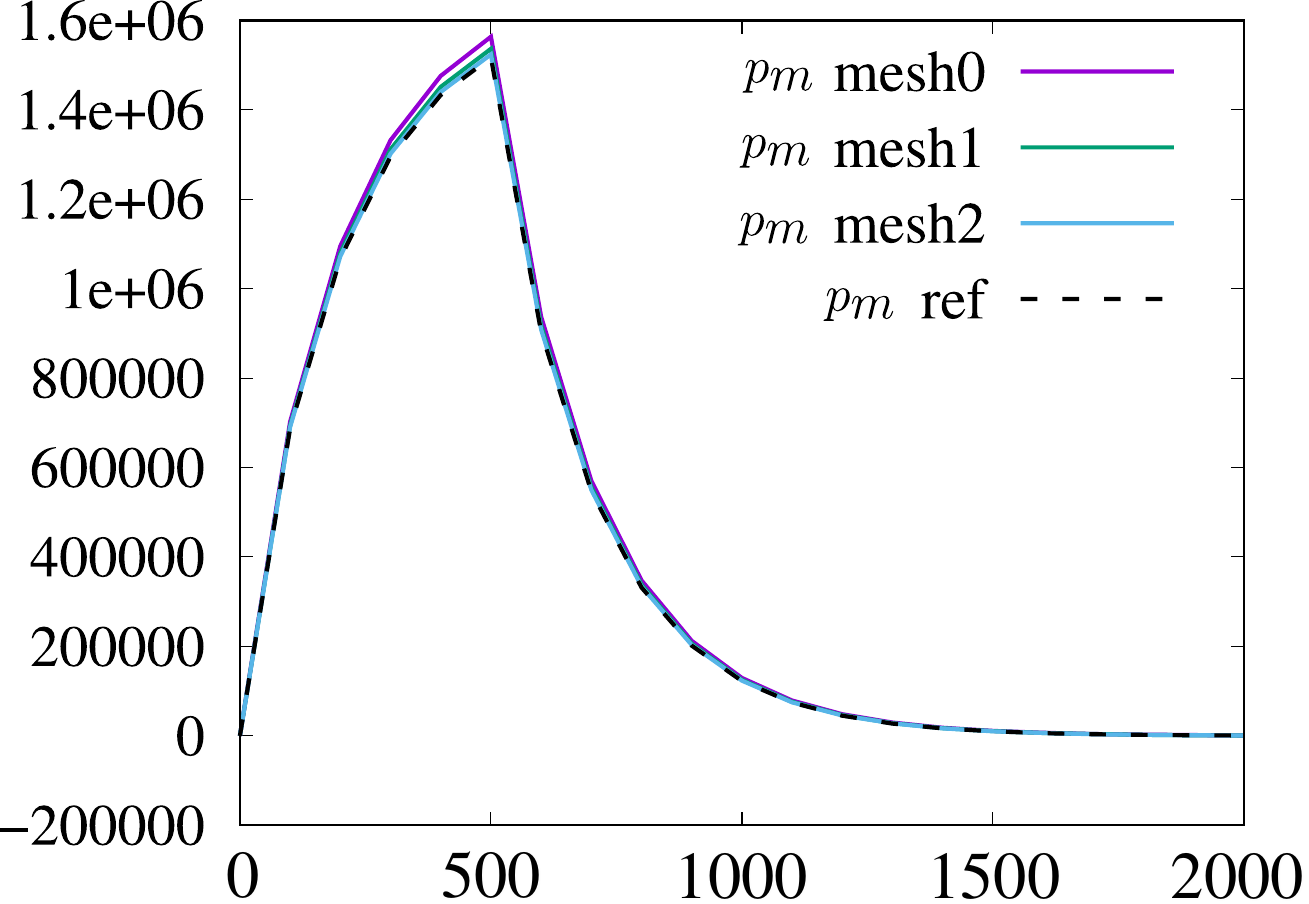}};
	    \node[left=of img, node distance=0cm, rotate=90, anchor=center,yshift=-0.6cm] {\footnotesize{  \text{Mean pressure in the matrix}}};
	    \node[below=of img, node distance=0cm, rotate=0, anchor=center,yshift=0.6cm,xshift=1.5em] {\footnotesize{ $t$ (\text{s})}};
     \end{tikzpicture}
	\caption{Mean pressure in the matrix as a function of time, for a family of three uniformly refined meshes. The reference solution is computed on a finer mesh using the Nitsche $\Po^2$ FEM presented in \cite{beaude2023mixed}. Test case of Section \ref{bergen_coup}.}
	\label{pm_mean}
\end{figure} 

\begin{figure}[!ht]
	\begin{minipage}{0.45\textwidth}
		\hspace{-0.4cm}
		\begin{tikzpicture}
			\node (img)  {\includegraphics[scale=0.155]{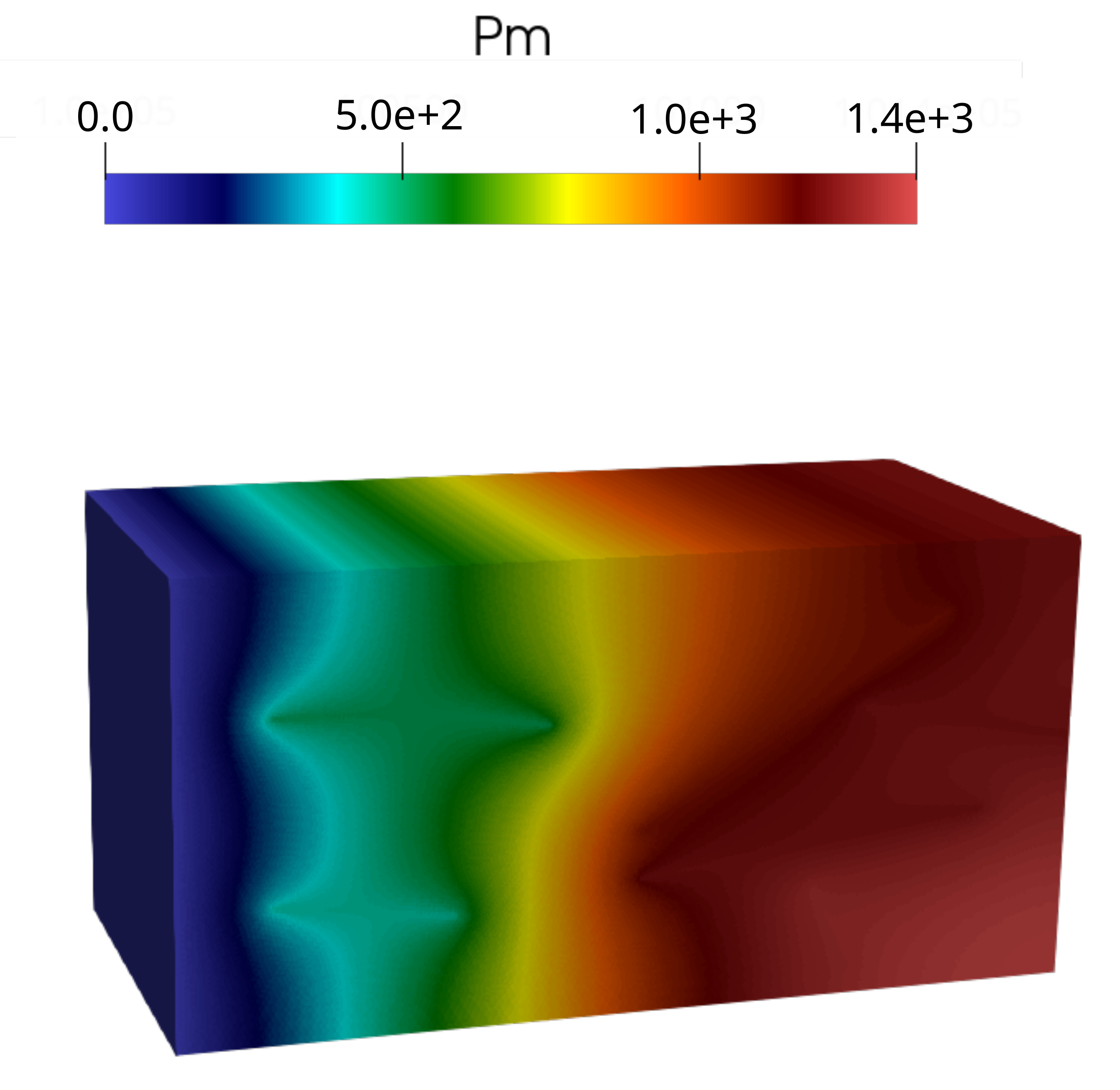}};
			\node[left=of img, node distance=0cm, rotate=90, anchor=center,yshift=-0.7cm] {\footnotesize{}};
		\end{tikzpicture}
	\end{minipage}\hspace*{.05\linewidth}
	\begin{minipage}{0.45\textwidth}     
		\hspace{-0.2cm}
		\vspace{-2.2cm} 
		\begin{tikzpicture}
			\node (img)  {\includegraphics[scale=0.35]{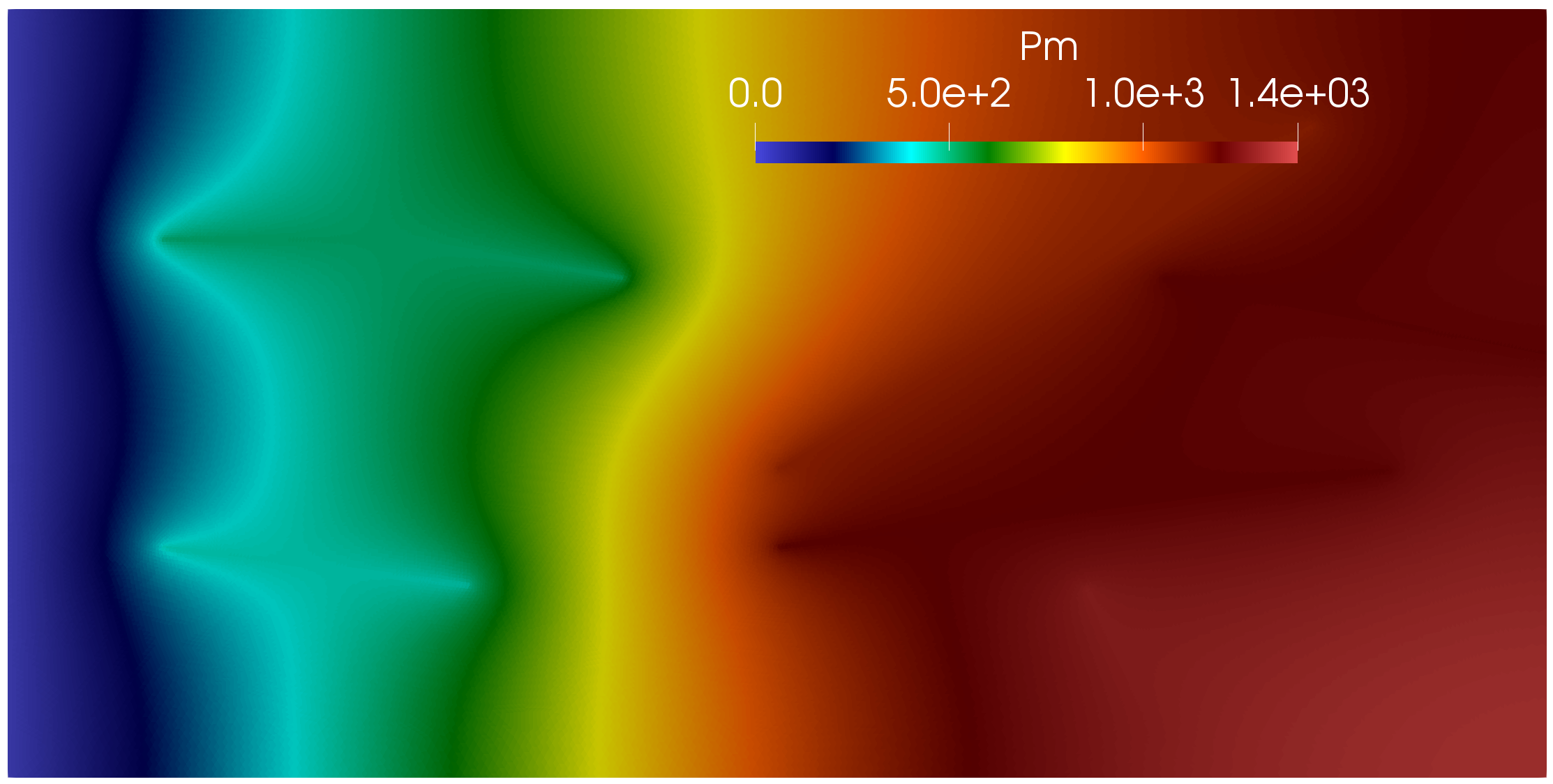}};
		\end{tikzpicture}
	\end{minipage}%
	
	\caption{Matrix over-pressures (w.r.t.~the initial pressure) in Pa at final time obtained with the mixed $\Po^1$-bubble VEM--$\Po^0$/HFV scheme (left) vs.~the Nitsche $\Po^2$ FEM/HFV scheme (right). Test case of Section \ref{bergen_coup}.}
	\label{pm_mix_vs_nitsh}
\end{figure}

\begin{figure}[H]
	\centering
	\begin{minipage}{0.45\textwidth}
		\hspace{-1.2cm}
		\begin{tikzpicture}
			\node (img)  {\includegraphics[scale=0.3]{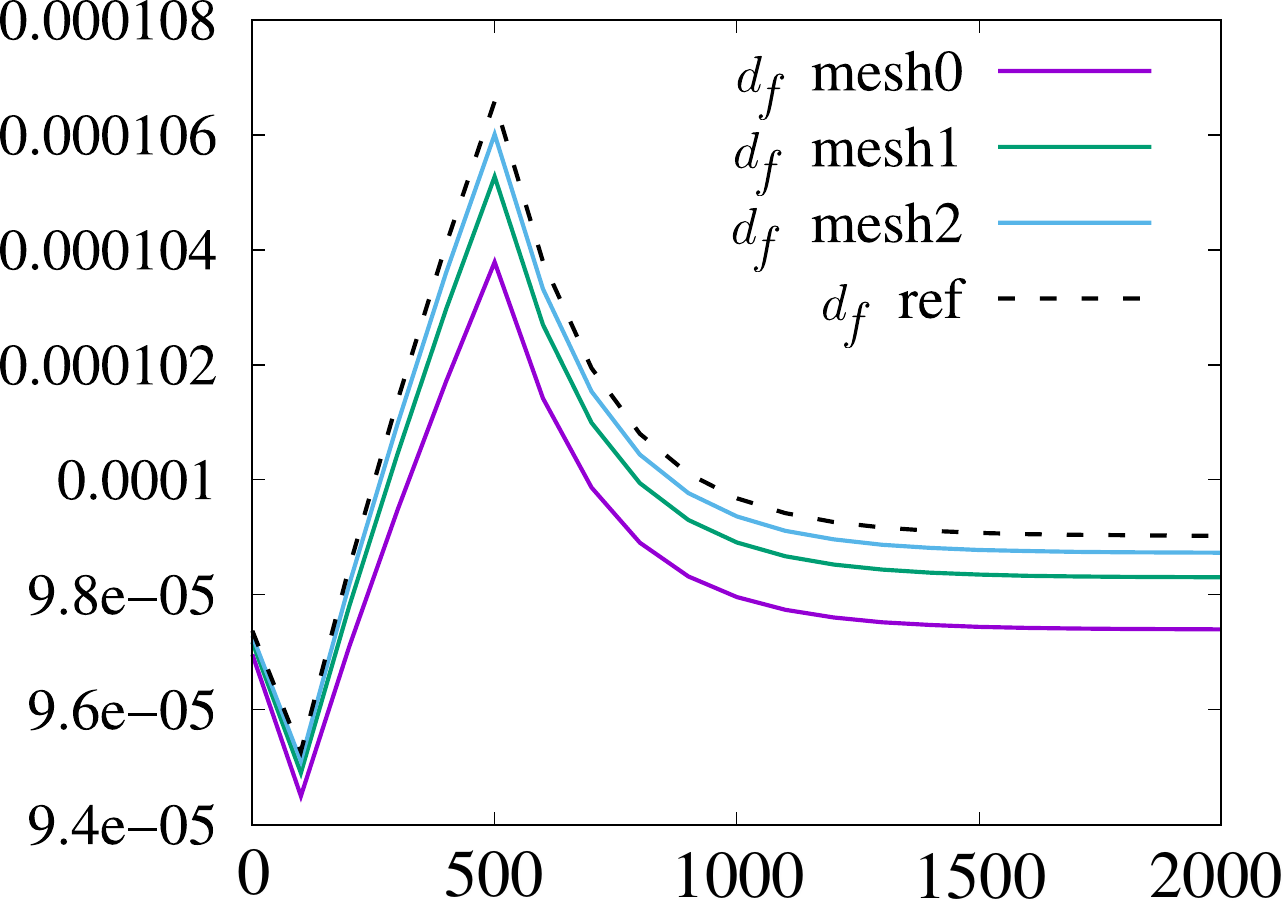}};
			\node[left=of img, node distance=0cm, rotate=90, anchor=center,yshift=-0.6cm] {\footnotesize{  \text{Mean aperture in the fracture}}};
		\node[below=of img, node distance=0cm, rotate=0, anchor=center,yshift=0.6cm,xshift=1.5em] {\footnotesize{ $t$ (\text{s})}};
		\end{tikzpicture}
	\end{minipage}%
	\begin{minipage}{0.45\textwidth}	 
		\hspace{-0.5cm}
		\begin{tikzpicture}
			\node (img)  {\includegraphics[scale=0.3]{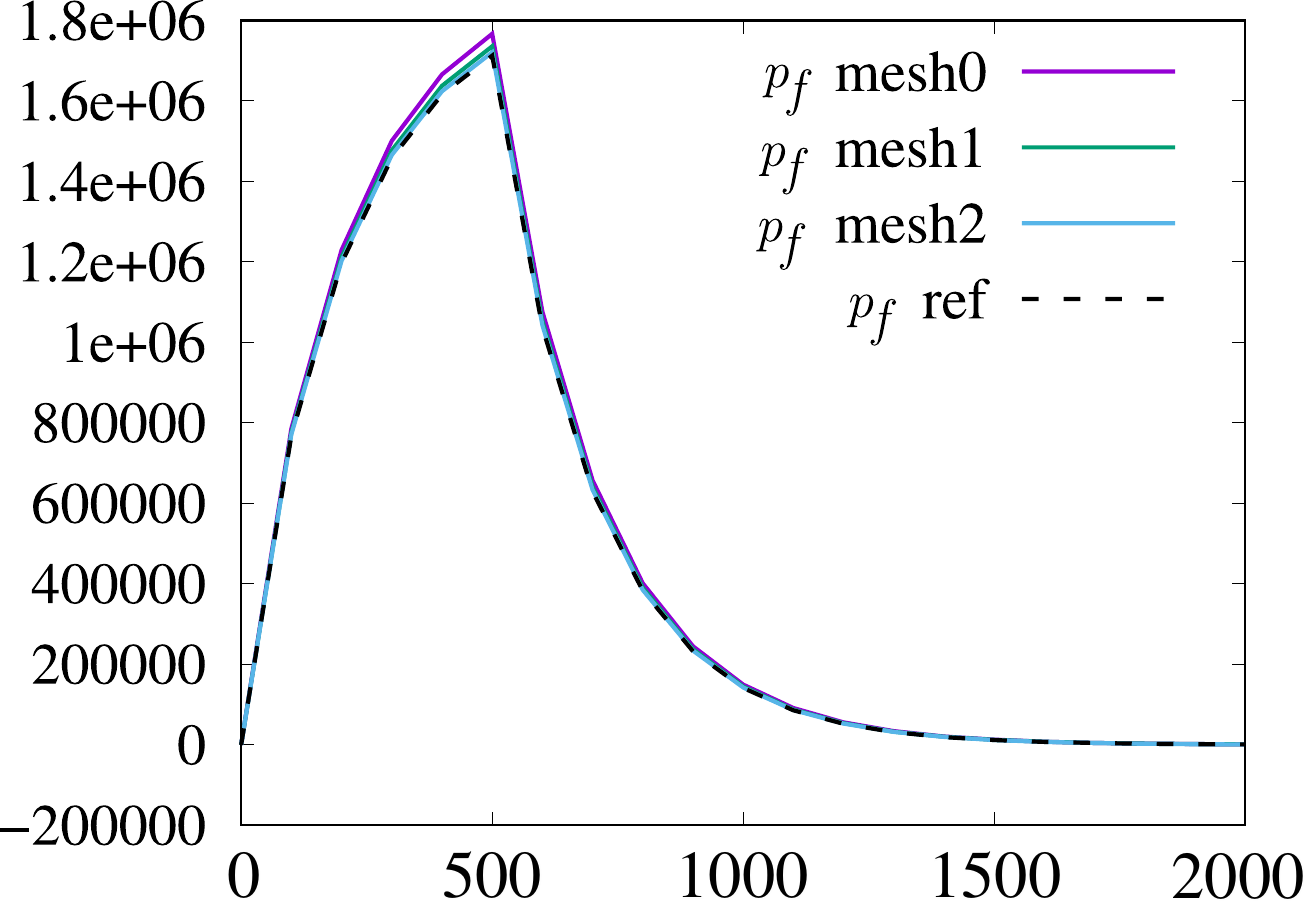}};
			\node[left=of img, node distance=0cm, rotate=90, anchor=center,yshift=-0.6cm] {\footnotesize{  \text{Mean pressure in the fracture}}};
			\node[below=of img, node distance=0cm, rotate=0, anchor=center,yshift=0.6cm,xshift=1.5em] {\footnotesize{ $t$ (\text{s})}};
		\end{tikzpicture}
	\end{minipage}%
	
	\caption{Mean fracture aperture and pressure as functions of time, for a family of three uniformly refined meshes. The reference solution is computed on a finer mesh using the Nitsche $\Po^2$ FEM presented in \cite{beaude2023mixed}. Test case of Section \ref{bergen_coup}.}\label{df_pf}	
\end{figure}
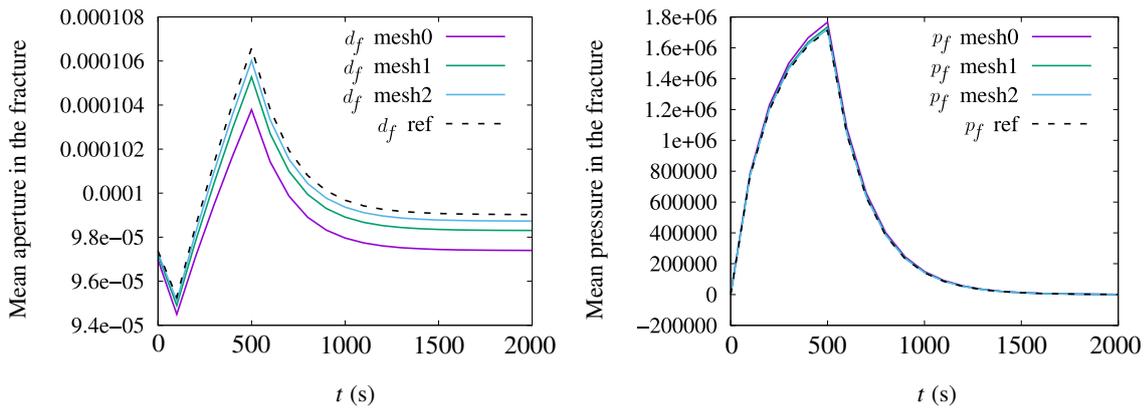

\begin{figure}[H]
	\centering
	\begin{minipage}{0.45\textwidth}
		\hspace{-1cm}
		\begin{tikzpicture}
			\node (img)  {\includegraphics[scale=0.3]{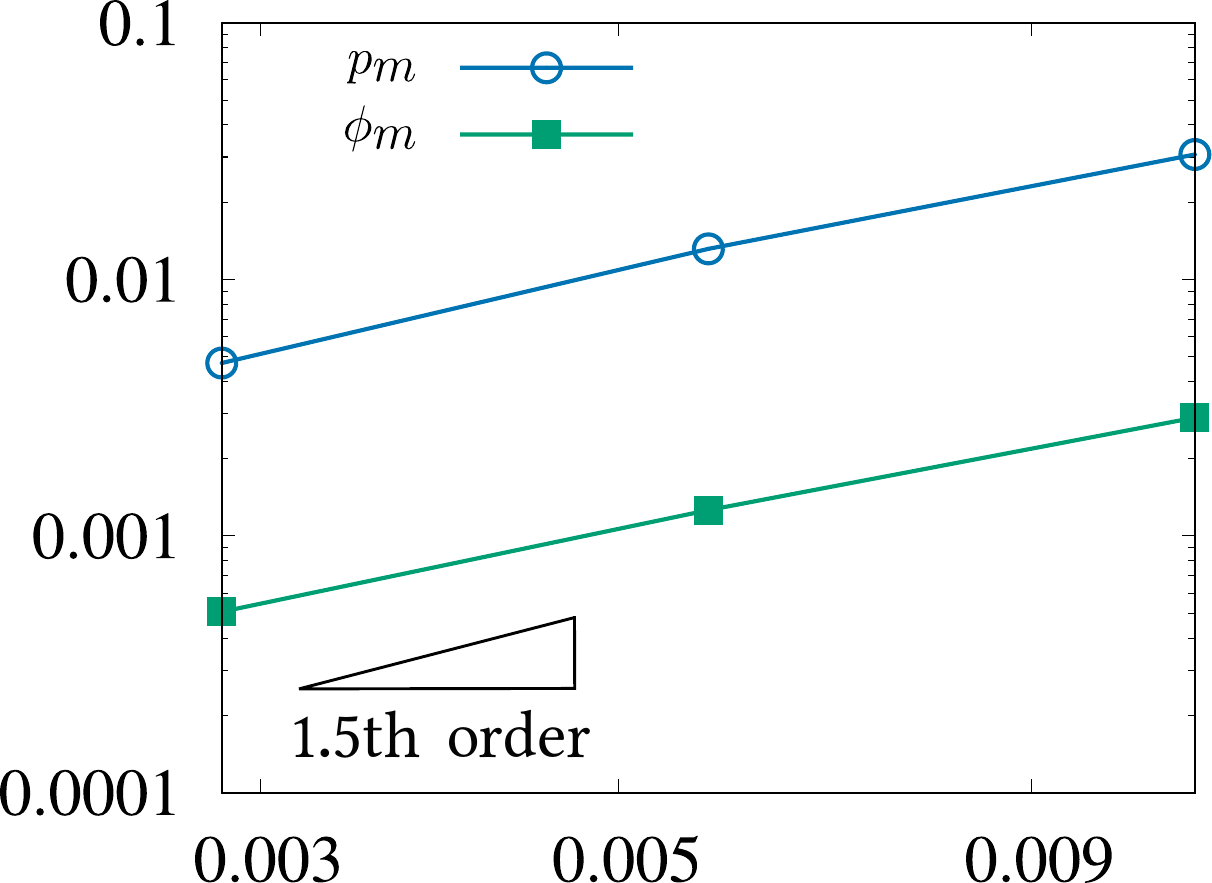}};
			\node[left=of img, node distance=0cm, rotate=90, anchor=center,yshift=-0.6cm] {\footnotesize{  \text{Relative $l^2$ Error}}};
			\node[below=of img, node distance=0cm, rotate=0, anchor=center,yshift=0.6cm,xshift=1.5em] {\footnotesize{ $h$ (\text{m})}};
		\end{tikzpicture}
	\end{minipage}%
	\begin{minipage}{0.45\textwidth}	 
		\begin{tikzpicture}
			\node (img)  {\includegraphics[scale=0.3]{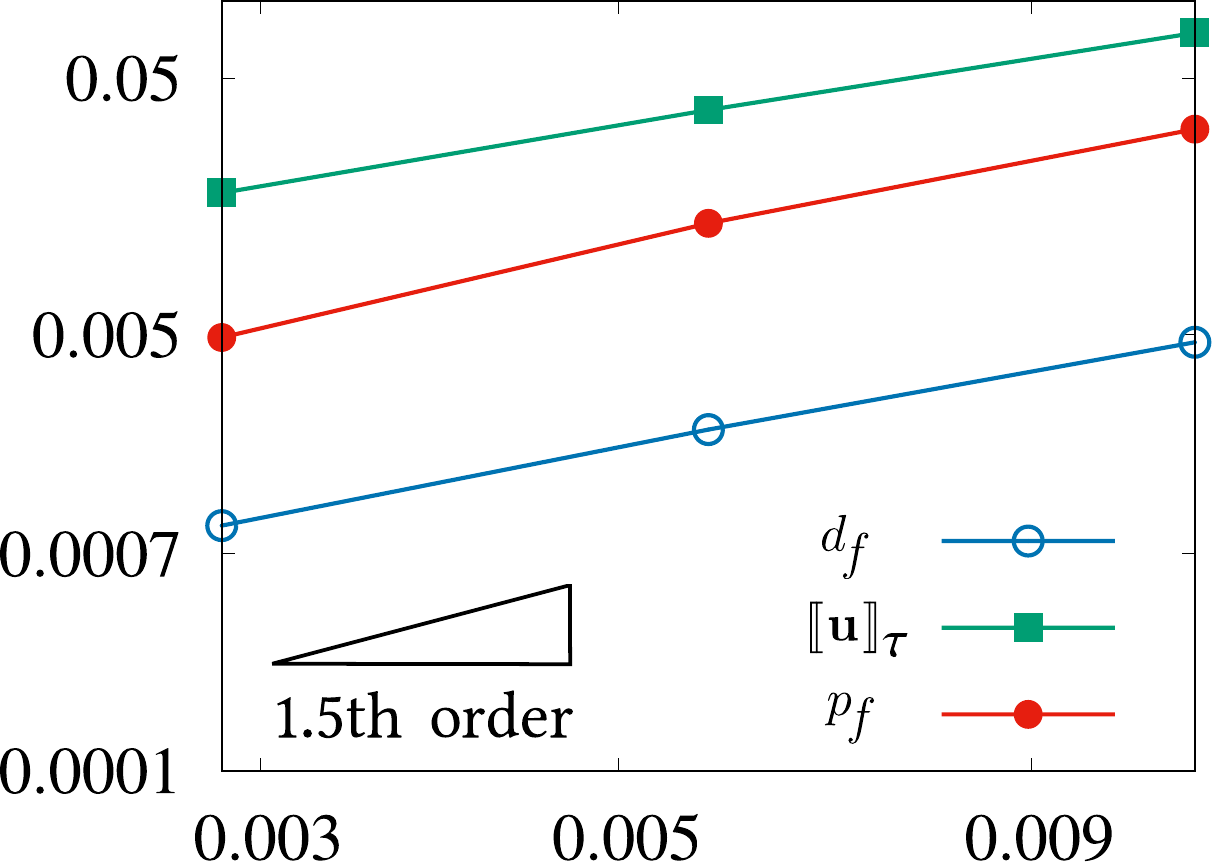}};
			\node[below=of img, node distance=0cm, rotate=0, anchor=center,yshift=0.6cm,xshift=1.5em] {\footnotesize{ $h$ (\text{m})}};
		\end{tikzpicture}
	\end{minipage}%
	\caption{Relative discrete $l^2$ error in time, as a function of the size $h$ of the largest fracture face, between the numerical and reference solutions for the mean matrix pressure and porosity (left), and the mean fracture pressure, aperture, and tangential jump $\jump{\bu }_{\tang}$ (right). Test case of Section \ref{bergen_coup}.}\label{error_m_f_var}	
\end{figure}
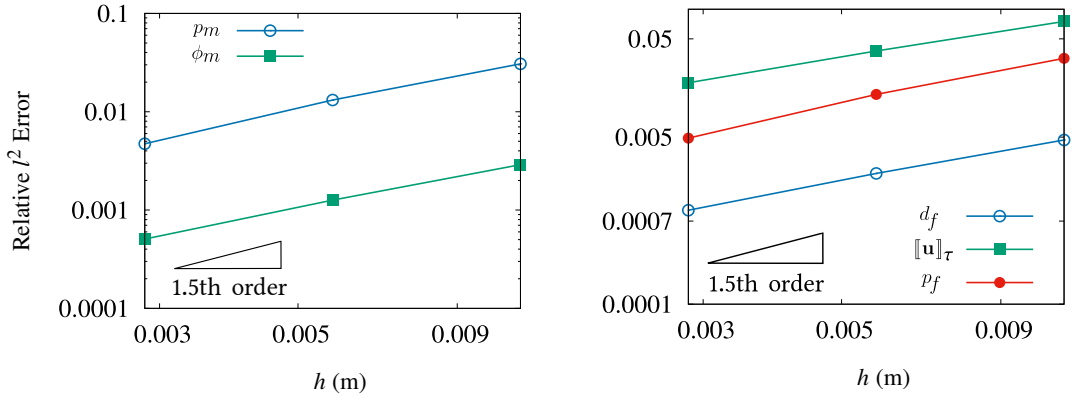

\subsubsection{3D DFM with intersecting fractures}\label{cubic_coup}

The objective of this test case is to assess the ability of the discretisation and of the nonlinear solver to simulate a poromechanical test case on a 3D DFM with intersecting fractures.
We consider the domain $\Omega = (0,1 \, \mbox{\rm m})^3$ with the fracture network $\Gamma$ of Figure \ref{cube_1_2}, which we discretise by a tetrahedral mesh consisting of either $47$k or $127$k cells. Both meshes are generated using TetGen \cite{tetgen}, starting with a triangulation of the fracture network and extended it to a conforming tetrahedral mesh of the 3D domain.

The Young's modulus and Poisson's ratio are set to $E=4 \,\mathrm{GPa}$ and $\nu=0.2$, and the friction coefficient to $F=0.5$. The Biot coefficient is set to $b=0.5$ and the Biot modulus to $M=10 \,\mathrm{GPa}$. Dirichlet boundary conditions are imposed at the bottom and top boundaries for the displacement field. We set $\bu=0$ at the bottom boundary $z=0$, and the following time dependent displacement at the top boundary $z=1$:
\begin{equation*}
	\bu(t,x,y,1)= 
	\left\{
	\begin{array}{l@{\quad}l}
		\dsp \prescript{t}{}[0.002\,\text{m}, 0.002\,\text{m}, -0.002\,\text{m}] \frac{2t}{T} &\text{if}~t\le\frac{T}{2},\\[1ex]
		\prescript{t}{}[0.002\,\text{m}, 0.002\,\text{m}, -0.002\,\text{m}] & \text{otherwise}.
	\end{array}  
	\right.
\end{equation*} 
Homogeneous Neumann boundary conditions are set on the lateral sides for the mechanics. 
Regarding the Darcy flow, the matrix permeability tensor is set to $\mathbb{K}_{m}=K_{m}\mathbb{I}$ with $K_{m}=10^{-14} \mathrm{~m}^{2}$, and the dynamic viscosity to $\eta=10^{-3} \mathrm{~Pa} \cdot \mathrm{s}$. 
The initial matrix porosity is $\phi_{m}^{0}=0.2$ and the fracture aperture corresponding to both contact state and zero displacement field is given by $\text{d}_f^c = 10^{-3} \, \mathrm{m}$.
The initial pressure in the matrix and fracture network is $p_{m}^{0}=p_{f}^{0}=10^{5} \mathrm{~Pa}$. Notice that the initial fracture aperture differs from $\text{d}_f^c$, since it is computed by solving the contact-mechanics given the initial pressures $p_{m}^{0}$ and $p_{f}^{0}$. The final time is set to $T=20 \mathrm{~s}$ and the time integration uses a uniform time stepping with $20$ time steps.
The boundary conditions for the flow are impervious except at the lateral boundaries $y=0$ and $y=1$, where a fixed pressure  $p_m = 10^{5} \mathrm{~Pa}$ is prescribed.

We fix an orthonormal coordinate system $(\n,\tang_1,\tang_2)$ on each fracture, and represent in Figures \ref{u_n} and \ref{u_t2} the normal and the $\tang_2$-tangential components of the displacement jump at the final time $t=T$. These pictures illustrate qualitatively the convergence of the displacement jump along the fracture when the mesh is refined.  

Figure \ref{new_curve} plots the cumulated total number of semi-smooth Newton iterations for the contact-mechanical model as a function of time, for both the one-sided and two-sides bubble cases. It shows the robustness of the nonlinear solver with respect to the mesh size and the (moderate)  benefit of the stronger stabilisation obtained with the two-sided bubble discretisation.  
\begin{figure}[H]
	\centering
	\includegraphics[scale=0.15]{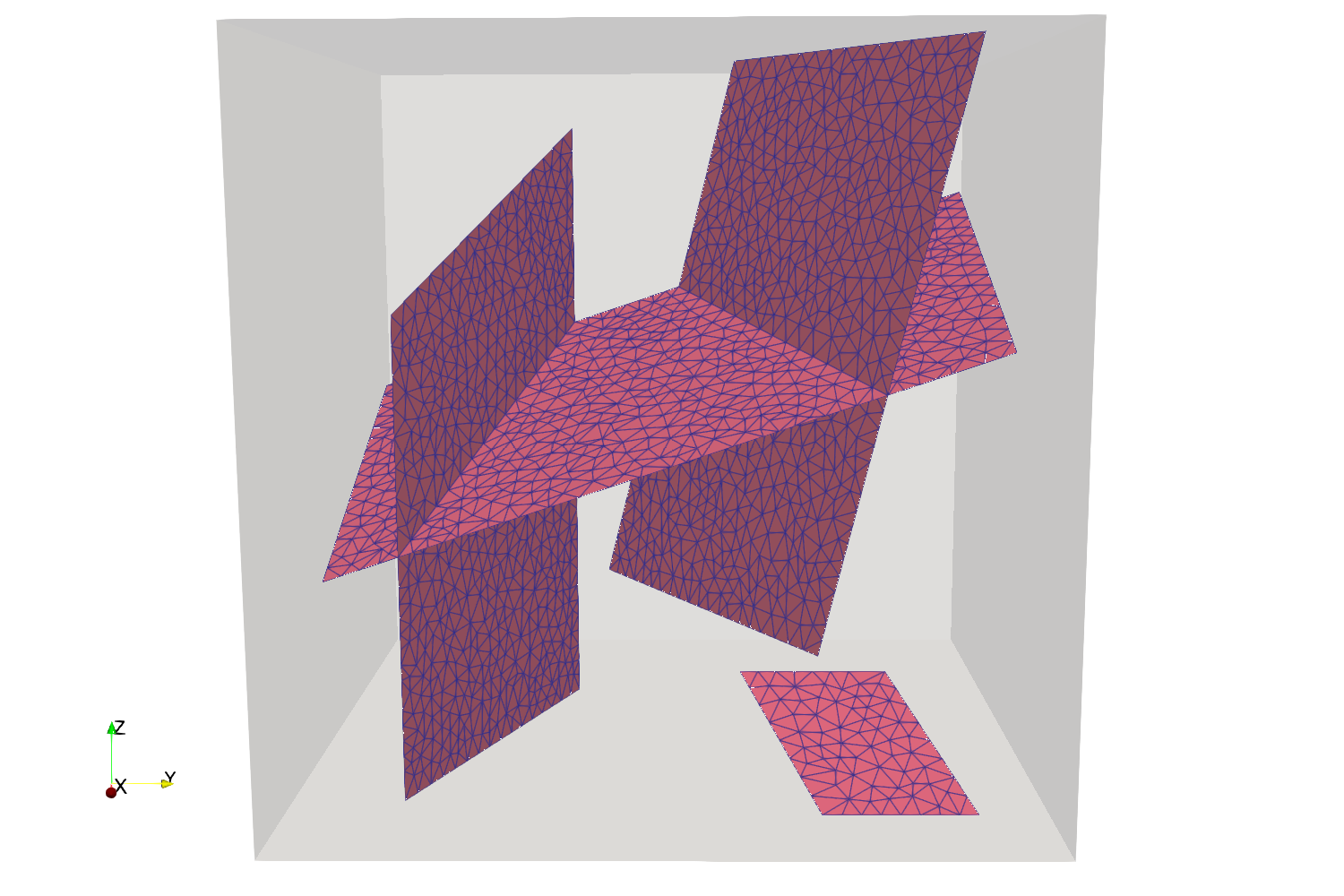}	
	\caption{Tetrahedral mesh of the 3D DFM with intersecting fractures comprising roughly $127$k cells. Test case of Section \ref{cubic_coup}.}
        \label{cube_1_2}	
\end{figure}
\begin{figure}[H]
	\centering
	
	\begin{minipage}{0.45\textwidth}
		\hspace{-2cm}
		\begin{tikzpicture}
			\node (img)  {\includegraphics[scale=0.2]{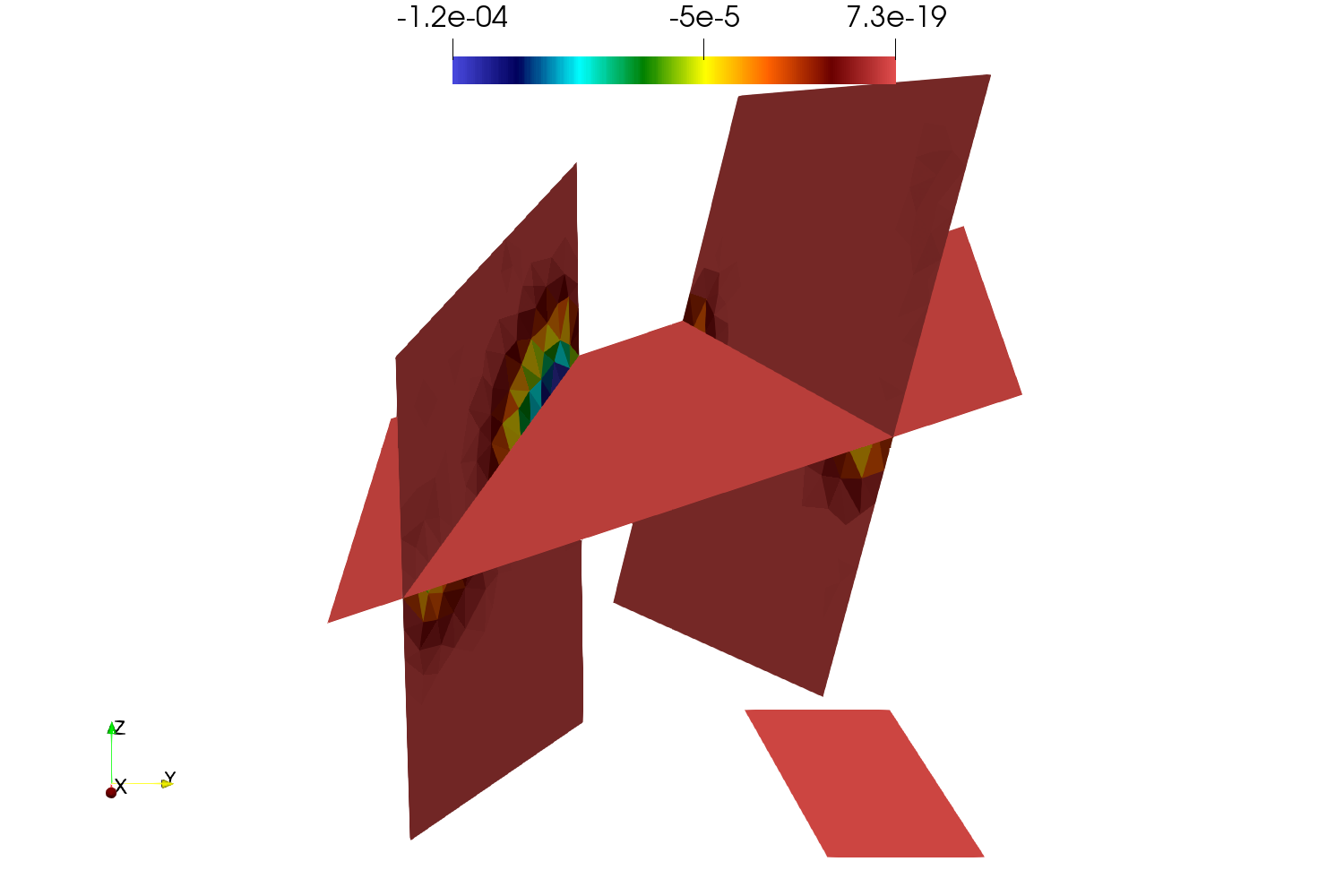}};
		\node[left=of img, node distance=0cm, rotate=90, anchor=center,yshift=-2cm] {\footnotesize{  $\llbracket \bu \rrbracket_{\mathbf{n}}$}};
		\end{tikzpicture}
	\end{minipage}%
	\begin{minipage}{0.45\textwidth}	 
		\hspace{-1.5cm}
		\begin{tikzpicture}
			\node (img)  {\includegraphics[scale=0.2]{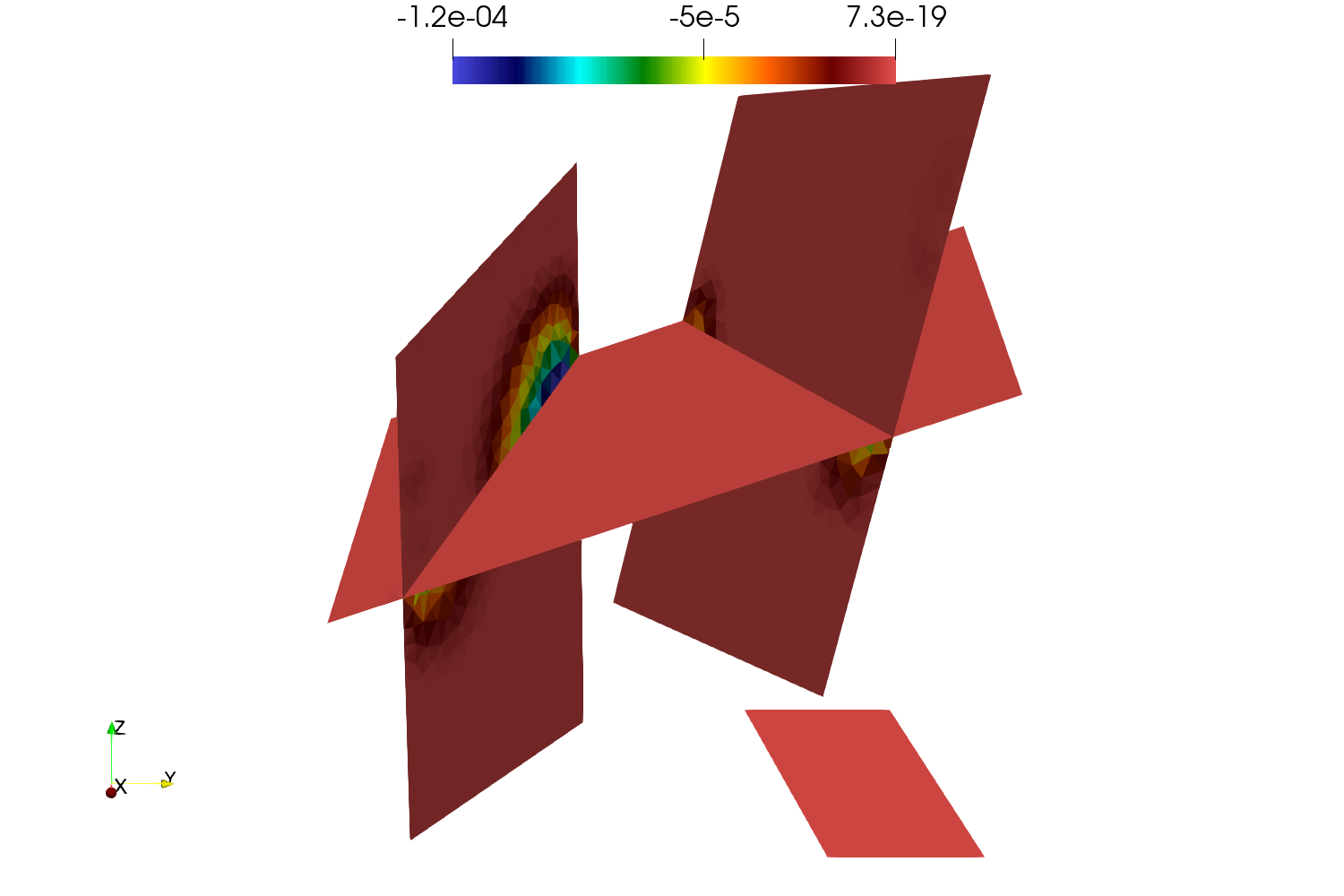}};
		\end{tikzpicture}
	\end{minipage}%
	\caption{Normal jump with the $47$k cells mesh (left) and the $127$k cells mesh (right), obtained at final time $t = T$. Test case of Section \ref{cubic_coup}.}\label{u_n}	
\end{figure}
\begin{figure}[H]
	\centering
	
	\begin{minipage}{0.45\textwidth}
		\hspace{-2cm}
		\begin{tikzpicture}
			\node (img)  {\includegraphics[scale=0.2]{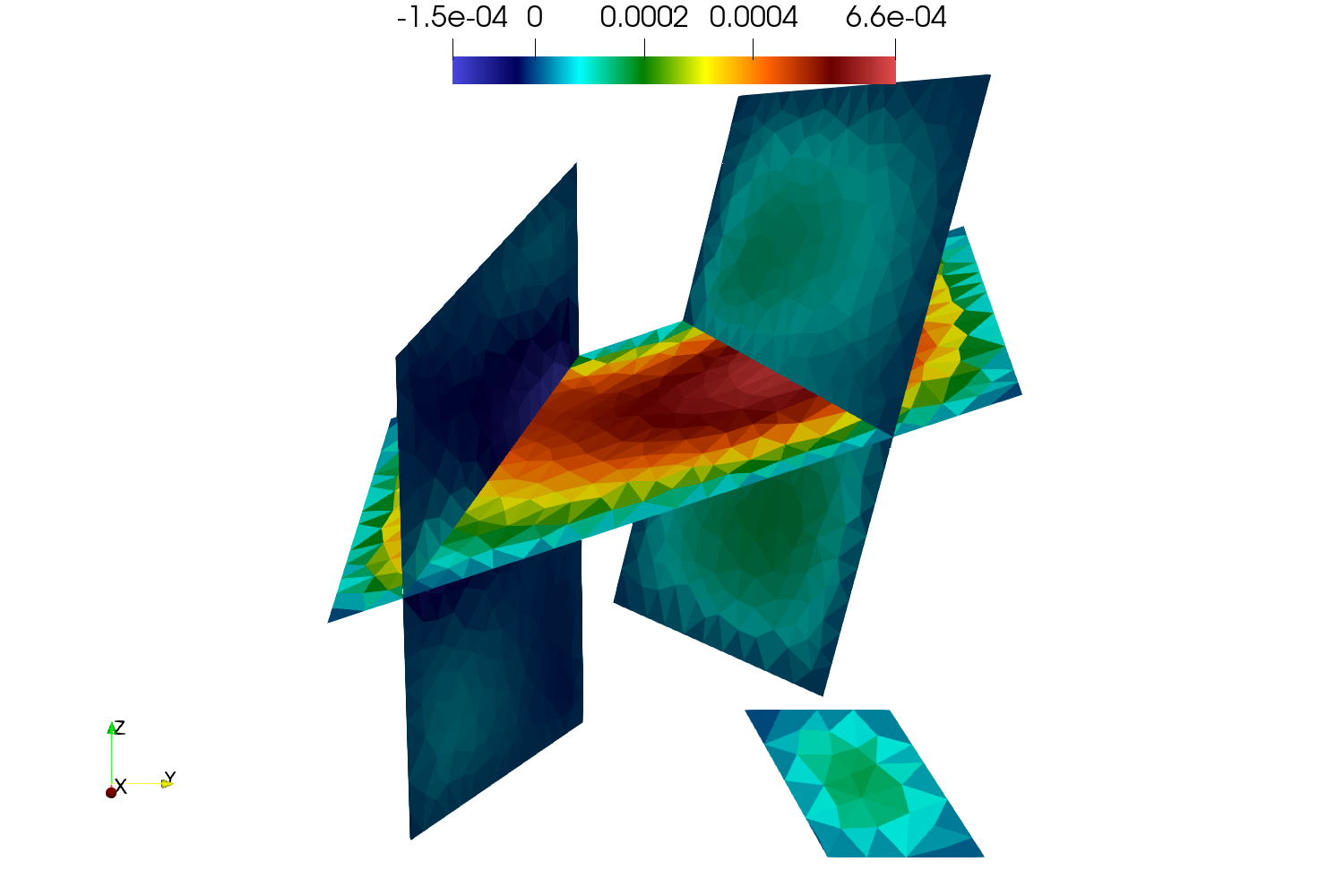}};
			\node[left=of img, node distance=0cm, rotate=90, anchor=center,yshift=-2cm] {\footnotesize{  $\llbracket \bu \rrbracket_{\tau_2}$}};
		\end{tikzpicture}
	\end{minipage}%
	\begin{minipage}{0.45\textwidth}	 
		\hspace{-1.5cm}
		\begin{tikzpicture}
			\node (img)  {\includegraphics[scale=0.2]{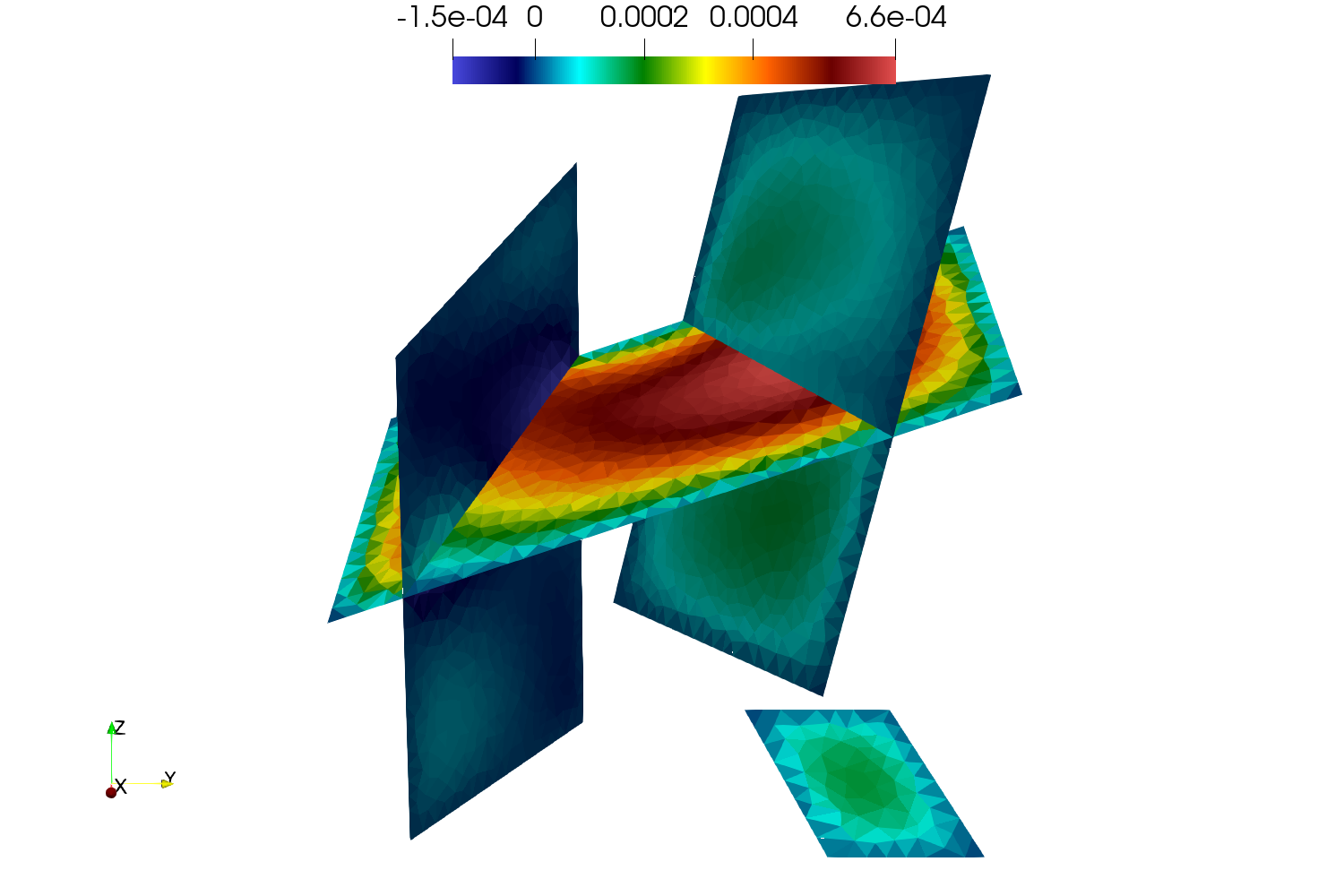}};
		\end{tikzpicture}
	\end{minipage}%
	\caption{The $\tang_2$ component of the tangential jump with the $47$k cells mesh (left) and the $127$k cells mesh (right), obtained at final time $t = T$. Test case of Section \ref{cubic_coup}.}\label{u_t2}	
\end{figure}
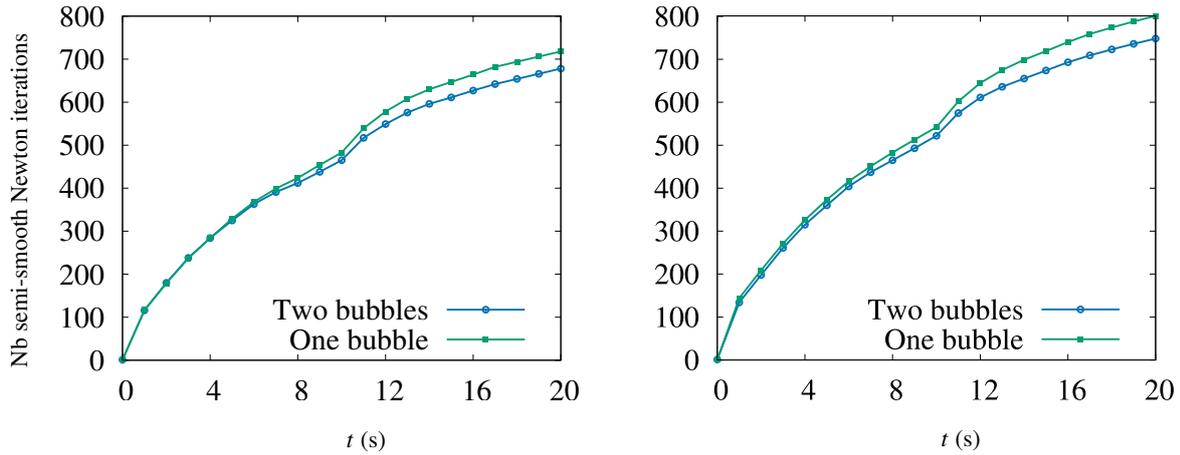
\begin{figure}[H]
	\centering
	\begin{minipage}{0.55\textwidth}
		\hspace{-0.5cm}
		\begin{tikzpicture}
			\node (img)  {\includegraphics[scale=0.35]{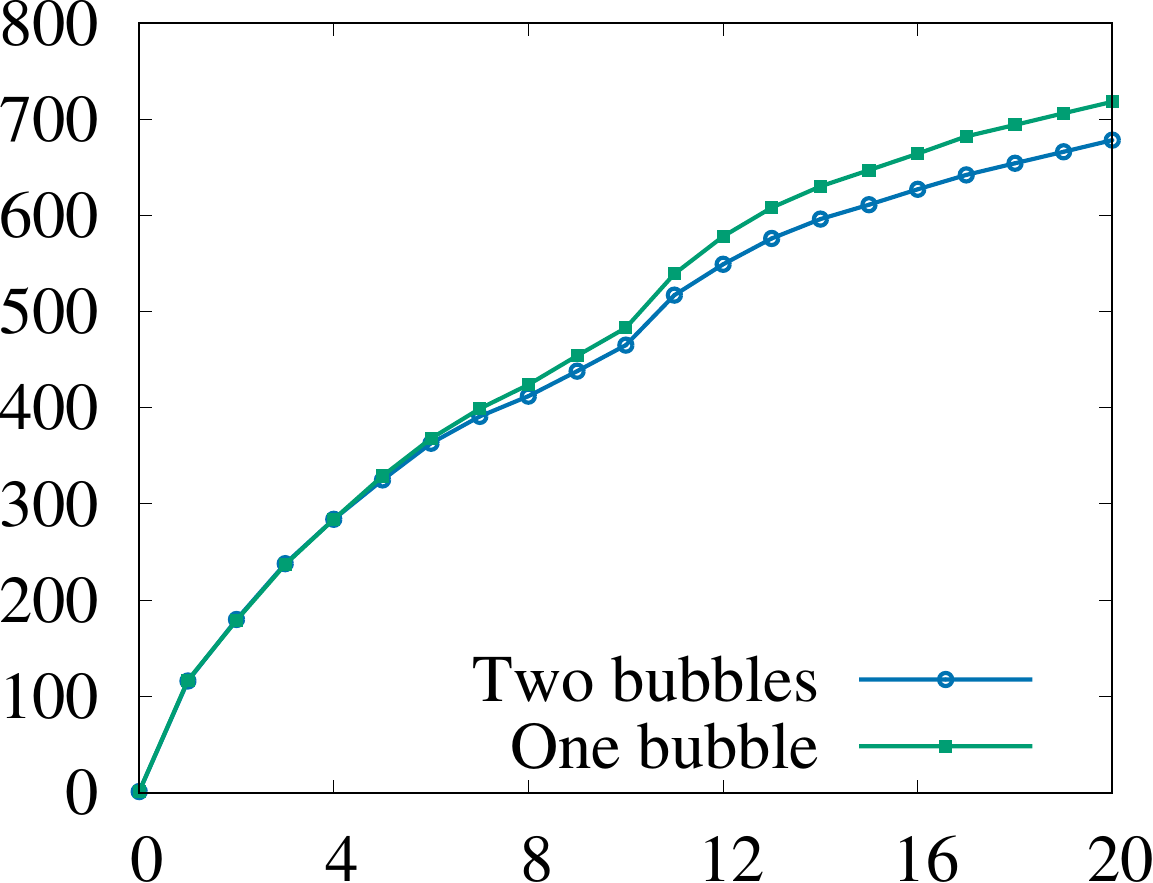}};
			\node[left=of img, node distance=0cm, rotate=90, anchor=center,yshift=-0.6cm] {\footnotesize{  \text{Nb semi-smooth Newton iterations}}};
			\node[below=of img, node distance=0cm, rotate=0, anchor=center,yshift=0.6cm,xshift=1.5em] {\footnotesize{ $t$ (\text{s})}};
		\end{tikzpicture}
	\end{minipage}%
	\begin{minipage}{0.55\textwidth}	
		\hspace{-0.5cm} 
		\begin{tikzpicture}
			\node (img)  {\includegraphics[scale=0.35]{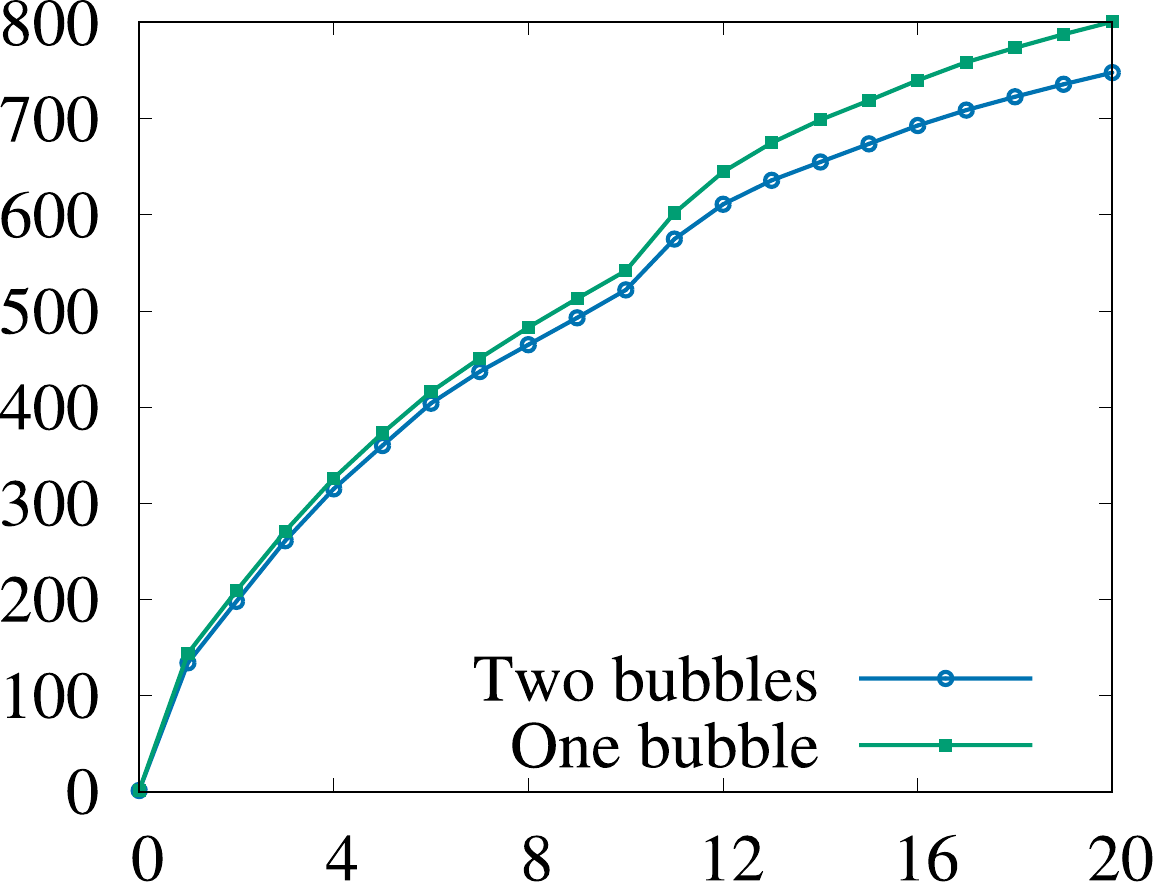}};
			\node[below=of img, node distance=0cm, rotate=0, anchor=center,yshift=0.6cm,xshift=1.5em] {\footnotesize{ $t$ (\text{s})}};
		\end{tikzpicture}
	\end{minipage}%
	\caption{Total number of semi-smooth Newton iterations for the contact-mechanical model as a function of time, with both one-sided and two-sided bubbles and for both meshes with $47$k cells (left) and $127$k cells (right). Test case of Section \ref{cubic_coup}.}\label{new_curve}	
\end{figure}

\section{Conclusions}
\label{sec:conclusions}

We have developed a novel numerical scheme for contact-mechanics in fractured/faulted porous media. It is based on a mixed formulation, using face-wise constant approximations of the Lagrange multipliers and a polytopal scheme for the displacement with fully discrete spaces and reconstruction operators. This scheme is applicable on meshes with generic elements, and employs a bubble degree of freedom to ensure the inf--sup stability with the Lagrange multiplier space. This fully discrete scheme is equivalent to a low-order bubble-VEM scheme, which is to our knowledge the first of its kind. 
Numerical validations were carried out on several 2D and 3D test cases, both for the stand alone contact-mechanical and the fully coupled mixed-dimensional poromechanical models.
Our future plan is to investigate the robustness of these polytopal discretisations to simulate fault reactivation in CO$_2$ geological storage, by using polyhedral meshes based on Corner Point Geometries.
The stability and convergence analysis of this mixed $\Po^1$ VEM-bubble--$\Po^0$ discretisation of the contact-mechanics requires new developments related to the additional bubble unknowns and to fracture networks including tips and intersections. This is a work in progress that will prove the discrete Korn inequality and the inf--sup condition.

\appendix

\section{Analysis of the elliptic projectors and stabilisation in the VEM space}
\label{appendice_VEM}

\subsection{Proof of Lemma \ref{lem:reconstruction.projection}}
\label{App_proj}

For $K \in \cells$ and $\sigma \in \faces_K$, let us show that  the local projectors $\pi^{K}$ and $\pi^{\Ksig}$  defined in \eqref{piK}  and \eqref{pisigma}, respectively, satisfy conditions \eqref{proj_p1_K} and \eqref{proj_p1_sigma} in $\bV_{h}^K$. 

For  $\bv \in \bV_{h}^K$, one can write:
\begin{align*}
		\int_K \nabla \bv  
		={}&  \sum_{\sigma \in \faces_{K}}\int_{\sigma} \gamma^{\Ksig} \bv \otimes \n_{\Ksig} \\ 
		={}& \sum_{\sigma \in  \faces_{\Gamma,K}^+}\int_{\sigma} \gamma^{\Ksig} \bv  \otimes \n_{\Ksig} +  \frac{1}{|K|} \sum_{\sigma \in \faces_{K} \setminus  \faces_{\Gamma,K}^+ }\int_{\sigma} \gamma^{\Ksig}\bv  \otimes \n_{\Ksig} \\
		={}& \sum_{\sigma \in \faces_{\Gamma,K}^+}  \left[\int_{\sigma} \( \gamma^{\Ksig} \bv  -\pi^{\Ksig} ( \gamma^{\Ksig}\bv ) \) \otimes \n_{\Ksig}  \right]+ \sum_{\sigma \in \faces_{K}} \int_{\sigma} \pi^{\Ksig} ( \gamma^{\Ksig}\bv ) \otimes \n_{\Ksig},
\end{align*}
where we have used the Stokes formula in the first equality and, in the conclusion, subtracted and added 
$$
\sum_{\sigma \in \faces_{\Gamma,K}^+}  \int_{\sigma} \pi^{\Ksig} ( \gamma^{\Ksig}\bv ) \otimes \n_{\Ksig}
$$
and used the relation $\int_{\sigma}\gamma^{\Ksig} \bv  = \int_{\sigma} \pi^{\Ksig} ( \gamma^{\Ksig}\bv )$ for all $\sigma \in \faces_{K} \setminus  \faces_{\Gamma,K}^+$ (see the integral condition in \eqref{eq:def.VKsig}). We therefore have
\begin{equation}\label{eq:pi.Pi.1}
		\int_K \nabla \bv  
       =  \sum_{\sigma \in \faces_{\Gamma,K}^+}  |\sigma| (\ID\bv)_{\Ksig} \otimes \n_{\Ksig}  +  \sum_{\sigma \in \faces_{K}} \int_{\sigma} \pi^{\Ksig} ( \gamma^{\Ksig}\bv ) \otimes \n_{\Ksig}.
\end{equation}
The relation \eqref{pisigma} between $\pi^\Ksig$ and $\Pi^\Ksig$ and the definition \eqref{PiKsigma} of $\Pi^\Ksig$ (recalling that $\int_\sigma(\x-\ov{\x}_\sigma)=0$) yield
$$
\int_{\sigma} \pi^{\Ksig} ( \gamma^{\Ksig}\bv )
=\int_{\sigma} \Pi^{\Ksig} ( \ID\gamma^{\Ksig}\bv )
=|\sigma|\overline{(\ID\gamma^{\Ksig}\bv)}_\Ksig=|\sigma|\overline{(\ID\bv)}_\Ksig
$$
(where $\overline{(\ID\bv)}_\Ksig$ is defined by \eqref{v_x_bar_sigma} with $\ID\bv$ instead of $\bv$).
Hence, \eqref{eq:pi.Pi.1} and the definition \eqref{eq:def.nablaK} of $\nabla^K=\nabla\Pi^K$ show that
$$
  \int_K \nabla \bv = |K|  \grad^{K}( \ID \bv)=\int_K  \nabla\Pi^K(\ID \bv)=\int_K\nabla\pi^K\bv.
$$
Taking $\mathbf{q}\in\Po^1(K)^d$ and multiplying this relation with $\nabla\mathbf{q}$ (which is constant) yields \eqref{proj_1_1}.

To get the relation \eqref{proj_1_2}, on the other hand, we use , the fact that $\pi^K\bv=\Pi^K\ID\bv$ is linear and \eqref{eq:def.piK} to write
$$
(\pi^K\bv)(\ov{\x}_K)=\frac{1}{|K|}\int_K \Pi^K\ID\bv =\ov{(\ID\bv)}_K=\sum_{s\in\nodes_K}\omega_s^K(\ID\bv)_{\Ks}=\sum_{s\in\nodes_K}\omega_s^K\bv(\x_s).
$$

We now turn to \eqref{proj_p1_sigma}. If $\bv \in \bV_{h}^\Ksig$, then $\bv\in\Po^1(e)$ on each edge $e$ of $\sigma$ and thus, by Stokes formula,
$$
	\int_{\sigma} \nabla_{\tang} \bv  =  \sum_{e \in \edges_{\sigma}} \int_{e} \gamma^{\sigma_e} \bv  \otimes \n_{\sige} 
	= \sum_{e= s_1 s_2 \in \edges_{\sigma}} |e| {(\ID \bv)_{K_{s_1}} + (\ID \bv)_{K_{s_2}}\over 2} \otimes \n_{\sige}
	= |\sigma|\grad^{\Ksig} ( \ID \bv ).
$$
The conclusion of \eqref{proj_p1_sigma} then follows as above.

\subsection{Discrete stability term $S_{\mu,\lambda,\D}$ as a VEM $\textit{dofi}$-$\textit{dofi}$ stabilisation} \label{App_stab}

Given $\bu, \bv \in \bV^K_h$, let us set $\bu_\D = \ID \bu$,  $\bv_\D = \ID \bv$.
The  usual $\textit{dofi}$-$\textit{dofi}$ approach first introduces the bilinear form based on the VEM degrees of freedom 
$$
s_K(\bu,\bv) =  h^{d-2}_K  \(\sum_{s\in \nodes_K} \bu_{\Ks} \bv_{\Ks} + \sum_{\sigma \in \faces_{\Gamma,K}^+  } \bu_\Ksig \bv_\Ksig \),
$$
and defines the stabilisation bilinear form as 
$$
\mathcal{S}_K(\bu,\bv) = s_K(\bu - \pi^K \bu, \bv - \pi^K \bv). 
$$
The fact that $\mathcal{S}_K(\bu,\bv) = S_K(\ID \bu, \ID \bv)$ directly follows from the definition of $S_K$ in \eqref{eq:StabK} and from the identities $(\bw - \pi^K \bw)_\Ks = (\ID\bw)_\Ks - (\Pi^K \ID\bw)(\x_s)$ and
$(\bw - \pi^K \bw)_\Ksig = (\ID\bw)_\Ksig - (\ID\pi^K \bw)_\Ksig = (\ID\bw)_\Ksig$ for all $\bw \in \bV^K_h$ (we have used $(\ID \mathbf{q})_{\Ksig}=0$ whenever $\mathbf{q}\in\Po^1(K)$, which follows from the fact that $\Pi^\Ksig(\ID \mathbf{q})=\mathbf{q}$ since $\mathbf{q}$ is linear).

%

\section*{Declarations}

\subsection*{Data availability}
Data will be made available on request.

\subsection*{Conflicts of interest}
The authors declare that they have no known competing financial interests or personal relationships that could have appeared to influence the work reported in this paper.

\section*{Acknowledgement}
We acknowledge the partial support from the joint laboratory IFPEN-Inria Convergence HPC/AI/HPDA for the energetic transition in realizing this project.

This work was also partially funded by the European Union (ERC Synergy, NEMESIS, project number 101115663). Views and opinions expressed are however those of the authors only and do not necessarily reflect those of the European Union or the European Research Council Executive Agency. Neither the European Union nor the granting authority can be held responsible for them.

\bibliographystyle{plain}
\bibliography{Poromeca_contact}

\begin{thebibliography}{10}

\bibitem{ahmad2013equivalent}
B.~Ahmad, A.~Alsaedi, F.~Brezzi, L~D. Marini, and A.~Russo.
\newblock Equivalent projectors for virtual element methods.
\newblock {\em Computers \& Mathematics with Applications}, 66(3):376--391,
  2013.

\bibitem{AHR17}
O.~Andersen, H.~M. Nilsen, and X.~Raynaud.
\newblock Virtual element method for geomechanical simulations of reservoir
  models.
\newblock {\em Computational Geosciences}, 21(5):877--893, 2017.

\bibitem{beaude2023mixed}
L.~Beaude, F.~Chouly, M.~Laaziri, and R.~Masson.
\newblock Mixed and {Nitsche}'s discretizations of coulomb frictional
  contact-mechanics for mixed dimensional poromechanical models.
\newblock {\em Computer Methods in Applied Mechanics and Engineering},
  413:116124, 2023.

\bibitem{beirao2013basic}
L~Beir{\~a}o~da Veiga, F.~Brezzi, A.~Cangiani, G.~Manzini, L~D. Marini, and
  A.~Russo.
\newblock Basic principles of virtual element methods.
\newblock {\em Mathematical Models and Methods in Applied Sciences},
  23(01):199--214, 2013.

\bibitem{beirao-brezzi-marini}
L.~Beir{\~a}o Da~Veiga, F.~Brezzi, and L.D. Marini.
\newblock Virtual elements for linear elasticity problems.
\newblock {\em SIAM Journal on Numerical Analysis}, 51:794--812, 2013.

\bibitem{Renard03}
F.~Ben~Belgacem and Y.~Renard.
\newblock Hybrid finite element methods for the signorini problem.
\newblock {\em Mathematics of Computation}, 72:1117--1145, 2003.

\bibitem{contact-norvegiens}
R.~L. Berge, I.~Berre, E.~Keilegavlen, J.~M. Nordbotten, and B.~Wohlmuth.
\newblock Finite volume discretization for poroelastic media with fractures
  modeled by contact mechanics.
\newblock {\em International Journal for Numerical Methods in Engineering},
  121:644--663, 2019.

\bibitem{BM2000}
D.~Bevillon and R.~Masson.
\newblock Stability and convergence analysis of partially coupled schemes for
  geomechanical-reservoir simulations.
\newblock European Conference on the Mathematics of Oil Recovery, Eage, 2000.

\bibitem{GDM-poromeca-cont}
F.~Bonaldi, K.~Brenner, J.~Droniou, and R.~Masson.
\newblock Gradient discretization of two-phase flows coupled with mechanical
  deformation in fractured porous media.
\newblock {\em Computers and Mathematics with Applications}, 98:40--68, 2021.

\bibitem{GDM-poromeca-disc}
F.~Bonaldi, K.~Brenner, J.~Droniou, R.~Masson, A.~Pasteau, and L.~Trenty.
\newblock Gradient discretization of two-phase poro-mechanical models with
  discontinuous pressures at matrix fracture interfaces.
\newblock {\em ESAIM: Mathematical Modelling and Numerical Analysis}, 2021.
\newblock Accepted for publication. \doi{10.1051/m2an/2021036}.

\bibitem{BDMP:21}
F.~Bonaldi, J.~Droniou, R.~Masson, and A.~Pasteau.
\newblock Energy-stable discretization of two-phase flows in deformable porous
  media with frictional contact at matrix-fracture interfaces.
\newblock {\em Journal of Computational Physics}, 455:Paper No. 110984, 28,
  2022.

\bibitem{BoonNordbotten22}
W.~M. Boon and J.~M. Nordbotten.
\newblock Mixed-dimensional poromechanical models of fractured porous media.
\newblock {\em Acta Mechanica}, 2022.

\bibitem{BORIO2021}
A.~Borio, F.~Hamon, N.~Castelletto, J.A. White, and R.S. Settgast.
\newblock Hybrid mimetic finite-difference and virtual element formulation for
  coupled poromechanics.
\newblock {\em Computer Methods in Applied Mechanics and Engineering},
  383:113917, 2021.

\bibitem{BHMS2016}
K.~Brenner, J.~Hennicker, R.~Masson, and P.~Samier.
\newblock {Gradient Discretization of Hybrid Dimensional Darcy Flows in
  Fractured Porous Media with discontinuous pressure at matrix fracture
  interfaces}.
\newblock {\em IMA Journal of Numerical Analysis}, 37:1551--1585, 2017.

\bibitem{BHL2023}
E.~Burman, P.~Hansbo, and M.~G. Larson.
\newblock The augmented {Lagrangian} method as a framework for stabilised
  methods in computational mechanics.
\newblock {\em Archives of Computational Methods in Engineering},
  30(4):2579--2604, 2023.

\bibitem{Chouly2017}
F.~Chouly, M.~Fabre, P.~Hild, R.~Mlika, J.~Pousin, and Y.~Renard.
\newblock An overview of recent results on {N}itsche's method for contact
  problems.
\newblock In St{\'e}phane P.~A. Bordas, Erik Burman, Mats~G. Larson, and
  Maxim~A. Olshanskii, editors, {\em Geometrically Unfitted Finite Element
  Methods and Applications}, pages 93--141, Cham, 2017. Springer International
  Publishing.

\bibitem{CHLR2020}
F.~Chouly, P.~Hild, V.~Lleras, and Y.~Renard.
\newblock {Nitsche method for contact with Coulomb friction: existence results
  for the static and dynamic finite element formulations}.
\newblock Preprint \hal{hal-02938032}, 2020.

\bibitem{coulet2020fully}
J.~Coulet, I.~Faille, V.~Girault, N.~Guy, and F.~Nataf.
\newblock A fully coupled scheme using virtual element method and finite volume
  for poroelasticity.
\newblock {\em Computational Geosciences}, 24:381--403, 2020.

\bibitem{da2015virtual}
L~Beir{\~a}o Da~Veiga, Carlo Lovadina, and David Mora.
\newblock A virtual element method for elastic and inelastic problems on
  polytope meshes.
\newblock {\em Computer methods in applied mechanics and engineering},
  295:327--346, 2015.

\bibitem{dipietro-lemaire}
D.~Di~Pietro and S.~Lemaire.
\newblock {An extension of the Crouzeix-Raviart space to general meshes with
  application to quasi-incompressible linear elasticity and Stokes flow}.
\newblock {\em Mathematics of Computation}, 84:1--31, 2015.

\bibitem{hho-book}
D.~A. Di~Pietro and J.~Droniou.
\newblock {\em The Hybrid High-Order Method for Polytopal Meshes: Design,
  Analysis, and Applications}, volume~19 of {\em Modeling, Simulation and
  Applications}.
\newblock Springer International Publishing, 2020.

\bibitem{ddr-variant}
Daniele~A. Di~Pietro and J\'er\^ome Droniou.
\newblock An arbitrary-order discrete de rham complex on polyhedral meshes:
  Exactness, poincar\'e inequalities, and consistency.
\newblock {\em Found. Comput. Math.}, 23:85--164, 2023.

\bibitem{Droniou2010}
J.~Droniou, R.~Eymard, T.~Gallou\"{e}t, and R.~Herbin.
\newblock A unified approach to mimetic finite difference, hybrid finite volume
  and mixed finite volume methods.
\newblock {\em Mathematical Models and Methods in Applied Sciences},
  20(02):265--295, 2010.

\bibitem{Hild17}
G.~Drouet and P.~Hild.
\newblock An accurate local average contact method for nonmatching meshes.
\newblock {\em Numerische Mathematik}, 136(2):467--502, 2017.

\bibitem{CRMECA_2023__351_S1_A28_0}
G.~Ench\'ery and L.~Ag\'elas.
\newblock Coupling linear virtual element and non-linear finite volume methods
  for poroelasticity.
\newblock {\em Comptes Rendus. M\'ecanique}, 2023.
\newblock Online first.

\bibitem{Eymard2009}
R.~Eymard, T.~Gallou\"{e}t, and R.~Herbin.
\newblock {Discretization of heterogeneous and anisotropic diffusion problems
  on general nonconforming meshes SUSHI: a scheme using stabilization and
  hybrid interfaces}.
\newblock {\em IMA Journal of Numerical Analysis}, 30(4):1009--1043, 06 2009.

\bibitem{Farmer2005}
C.~L. Farmer.
\newblock {\em Geological Modelling and Reservoir Simulation}, pages 119--212.
\newblock Springer Berlin Heidelberg, Berlin, Heidelberg, 2005.

\bibitem{tchelepi-castelletto-2020}
A.~Franceschini, N.~Castelletto, J.A. White, and H.A. Tchelepi.
\newblock Algebraically stabilized {L}agrange multiplier method for frictional
  contact mechanics with hydraulically active fractures.
\newblock {\em Computer Methods in Applied Mechanics and Engineering},
  368:113161, 2020.

\bibitem{GKT16}
T.~T. Garipov, M.~Karimi-Fard, and H.A. Tchelepi.
\newblock Discrete fracture model for coupled flow and geomechanics.
\newblock {\em Computational Geosciences}, 20(1):149--160, 2016.

\bibitem{GH19}
T.T. Garipov and M.H. Hui.
\newblock Discrete fracture modeling approach for simulating coupled
  thermo-hydro-mechanical effects in fractured reservoirs.
\newblock {\em International Journal of Rock Mechanics and Mining Sciences},
  122:104075, 2019.

\bibitem{GKW16}
V.~Girault, K.~Kumar, and M.F. Wheeler.
\newblock Convergence of iterative coupling of geomechanics with flow in a
  fractured poroelastic medium.
\newblock {\em Computational Geosciences}, 20:997--1011, 2016.

\bibitem{tetgen}
S.~Hang.
\newblock Tetgen, a delaunay-based quality tetrahedral mesh generator.
\newblock {\em ACM Trans. on Mathematical Software}, 41(2), 2015.

\bibitem{hansbo-larson}
P.~Hansbo and M.G. Larson.
\newblock Discontinuous {G}alerkin and the {C}rouzeix--{R}aviart element:
  Application to elasticity.
\newblock {\em ESAIM: Mathematical Modelling and Numerical Analysis},
  37:63--72, 2003.

\bibitem{haslinger-96}
J.~Haslinger, I.~Hlav\'a{$\check{c}$}ek, and J.~Ne{$\check{c}$}as.
\newblock {\em Numerical methods for unilateral problems in solid mechanics},
  volume~IV of {\em Handbook of Numerical Analysis (eds. P.G. Ciarlet and J.L.
  Lions)}.
\newblock North-Holland Publishing Co., Amsterdam, 1996.

\bibitem{keilegavlen-nordbotten}
E.~Keilegavlen and J.M. Nordbotten.
\newblock Finite volume methods for elasticity with weak symmetry.
\newblock {\em International Journal for Numerical Methods in Engineering},
  112:939--962, 2017.

\bibitem{KTJ11}
J.~Kim, H.A. Tchelepi, and R.~Juanes.
\newblock Stability and convergence of sequential methods for coupled flow and
  geomechanics: Fixed-stress and fixed-strain splits.
\newblock {\em Computer Methods in Applied Mechanics and Engineering},
  200:1591--1606, 2011.

\bibitem{Lleras-2009}
V.~Lleras.
\newblock A stabilized {Lagrange} multiplier method for the finite element
  approximation of frictional contact problems in elastostatics.
\newblock {\em Mathematical Modelling of Natural Phenomena}, 4(1):163--182,
  2009.

\bibitem{NEJATI2016123}
M.~Nejati, A.~Paluszny, and R.~W. Zimmerman.
\newblock A finite element framework for modeling internal frictional contact
  in three-dimensional fractured media using unstructured tetrahedral meshes.
\newblock {\em Computer Methods in Applied Mechanics and Engineering},
  306:123--150, 2016.

\bibitem{contact-BEM}
A.V. Phan, J.A.L. Napier, L.J. Gray, and T.~Kaplan.
\newblock Symmetric-{G}alerkin {BEM} simulation of fracture with frictional
  contact.
\newblock {\em International Journal for Numerical Methods in Engineering},
  57:835--851, 2003.

\bibitem{SBN12}
T.H. Sandve, I.~Berre, and J.M. Nordbotten.
\newblock An efficient multi-point flux approximation method for discrete
  fracture-matrix simulations.
\newblock {\em Journal of Computational Physics}, 231:3784--3800, 2012.

\bibitem{settari94}
A.~Settari and F.~Mourits.
\newblock {Coupling of geomechanics and reservoir simulation models}.
\newblock In {\em Proceedings, 8th International Conference on Computer Methods
  and Advances in Geomechanics,}, volume~3, pages 2151--2158. Balkema, 1994.

\bibitem{thm-bergen}
I.~Stefansson, I.~Berre, and E.~Keilegavlen.
\newblock A fully coupled numerical model of thermo-hydro-mechanical processes
  and fracture contact mechanics in porous media.
\newblock {\em Computer Methods in Applied Mechanics and Engineering},
  386:114122, 2021.

\bibitem{WELLMANN20181}
F.~Wellmann and G.~Caumon.
\newblock Chapter one - 3-d structural geological models: Concepts, methods,
  and uncertainties.
\newblock volume~59 of {\em Advances in Geophysics}, pages 1--121. Elsevier,
  2018.

\bibitem{Wohlmuth11}
B.~Wohlmuth.
\newblock Variationally consistent discretization schemes and numerical
  algorithms for contact problems.
\newblock {\em Acta Numerica}, 20:569--734, 2011.

\bibitem{Wriggerscontact2024}
P.~Wriggers, F.~Aldakheel, and B.~Hudobivnik.
\newblock {\em Virtual Element Formulation for Contact}, pages 317--367.
\newblock Springer International Publishing, Cham, 2024.

\bibitem{Wriggers2024}
P.~Wriggers, F.~Aldakheel, and B.~Hudobivnik.
\newblock {\em Virtual Elements for Fracture Processes}, pages 243--315.
\newblock Springer International Publishing, Cham, 2024.

\bibitem{Wriggers2016}
P.~Wriggers, W.~T. Rust, and B.~D. Reddy.
\newblock A virtual element method for contact.
\newblock {\em Computational Mechanics}, 58(6):1039--1050, 2016.

\end{thebibliography}

\end{document}